\documentclass[aop]{imsart}

\RequirePackage{amsthm,amsmath,amsfonts,amssymb}
\RequirePackage[numbers]{natbib}
\RequirePackage[OT1]{fontenc}
\RequirePackage[colorlinks,linkcolor=blue,citecolor=blue,urlcolor=blue]{hyperref}
\usepackage{hyperref}
\usepackage{extarrows}   
\usepackage{multirow}
\usepackage{stmaryrd}
\usepackage{color}
\usepackage{enumitem} 
\usepackage{mathtools} %for sub/superstack
\usepackage{dsfont} % for \mathds
\usepackage{comment}
\usepackage{calrsfs}
\usepackage{tikz}
\usepackage{pgfplots} 
\pgfplotsset{compat=newest}

\startlocaldefs
\numberwithin{equation}{section}
\theoremstyle{plain} 
\newtheorem{theorem}{Theorem}[section]
\newtheorem{lemma}[theorem]{Lemma}
\newtheorem{corollary}[theorem]{Corollary}
\newtheorem{proposition}[theorem]{Proposition}
\newtheorem{remark}[theorem]{Remark}
\newtheorem{definition}[theorem]{Definition}
\newtheorem{assumption}[theorem]{Assumption}

\theoremstyle{definition}
\newtheorem{notation}[theorem]{Notation}

\DeclareMathOperator{\erf}{erf}

\renewcommand{\Re}{\operatorname{Re}}
\renewcommand{\Im}{\operatorname{Im}}

\DeclareMathOperator{\E}{\mathbf{E}}
\DeclareMathOperator{\V}{\mathbf{Var}}

\renewcommand{\P}{{\mathbf P}}

\newcommand{\cO}{\mathcal{O}}
\newcommand{\co}{{\scriptstyle\mathcal{O}}}
\newcommand{\R}{{\mathbb R }}
\newcommand{\N}{{\mathbb N}}

\newcommand{\C}{{\mathbb C}}

\newcommand{\OO}{O}
\newcommand{\eps}{\epsilon}

\newcommand{\ii}{\mathrm{i}}

\newcommand{\dd}{\operatorname{d}\!{}}
\newcommand{\dif}{\operatorname{d}\!{}}
\newcommand{\ie}{\emph{i.e., }}
\newcommand{\eg}{\emph{e.g., }}
\newcommand{\cf}{\emph{c.f., }}
\newcommand{\wt}{\widetilde}
\newcommand{\ud}{\underline}
\newcommand{\wh}{\widehat}
\newcommand{\gz}{G^{z}}

\newcommand{\nc}{\normalcolor}

\newcommand{\bs}{\boldsymbol}

\def\ga{G^{z_1}}
\def\gb{G^{z_2}}
\def\ma{\mathfrak{m}^{z_1}}
\def\mb{\mathfrak{m}^{z_2}}

\DeclareMathOperator{\Tr}{Tr}

\def\F{\mathcal{F}}

\def\one{\mathds{1}}

\def\<{\langle}
\def\>{\rangle}

\DeclarePairedDelimiter{\abs}{\lvert}{\rvert}%

\DeclarePairedDelimiterX{\tuple}[1](){#1}
\DeclarePairedDelimiterX{\set}[1]\{\}{#1}
\DeclarePairedDelimiterXPP{\landauO}[1]{\cO}(){}{#1}
\DeclarePairedDelimiterXPP{\landauo}[1]{\co}(){}{#1}
\DeclarePairedDelimiterXPP{\landauOprec}[1]{\cO_\prec}(){}{#1}

\renewcommand{\mathbf}[1]{\bs{#1}}

\DeclareMathOperator{\Spec}{Spec}

\endlocaldefs

\begin{document}

\begin{frontmatter}
\title{On the rightmost eigenvalue of non-Hermitian random matrices}
%\title{A sample article title with some additional note\thanksref{T1}}
\runtitle{On the rightmost eigenvalue of non-Hermitian random matrices}
%\thankstext{T1}{A sample of additional note to the title.}

\begin{aug}
%%%%%%%%%%%%%%%%%%%%%%%%%%%%%%%%%%%%%%%%%%%%%%%
%% Only one address is permitted per author. %%
%% Only division, organization and e-mail is %%
%% included in the address.                  %%
%% Additional information can be included in %%
\author[A]{\fnms{Giorgio}~\snm{Cipolloni}\ead[label=e1]{gc4233@princeton.edu}},
\author[B]{\fnms{L\'aszl\'o}~\snm{Erd{\H o}s}\ead[label=e2]{laszlo.erdoes@ist.ac.at}},
\author[C]{\fnms{Dominik}~\snm{Schr\"oder}\ead[label=e3]{dschroeder@ethz.ch}}
\and
\author[B]{\fnms{Yuanyuan}~\snm{Xu}\ead[label=e4]{yuanyuan.xu@ist.ac.at}}
%%%%%%%%%%%%%%%%%%%%%%%%%%%%%%%%%%%%%%%%%%%%%%
%% Addresses                                %%
%%%%%%%%%%%%%%%%%%%%%%%%%%%%%%%%%%%%%%%%%%%%%%
\address[A]{Princeton Center for Theoretical Science, Princeton University \printead[presep={,\ }]{e1}}

\address[B]{Institute of Science and Technology Austria \printead[presep={,\ }]{e2,e4}}

\address[C]{Institute for Theoretical Studies, ETH Zurich \printead[presep={,\ }]{e3}}
\end{aug}

\begin{abstract}
We establish a precise three-term asymptotic expansion, with an optimal estimate of the error term,
for the rightmost eigenvalue of an $n\times n$ random matrix with
independent identically distributed complex entries as $n$ tends to infinity.
All terms in the expansion are universal.
\end{abstract}



\end{frontmatter}
%%%%%%%%%%%%%%%%%%%%%%%%%%%%%%%%%%%%%%%%%%%%%%
%% Please use \tableofcontents for articles %%
%% with 50 pages and more                   %%
%%%%%%%%%%%%%%%%%%%%%%%%%%%%%%%%%%%%%%%%%%%%%%
%\tableofcontents

%%%%%%%%%%%%%%%%%%%%%%%%%%%%%%%%%%%%%%%%%%%%%%
%%%% Main text entry area:

\section{Introduction}

Large random matrices are frequently used to model complex systems of many degrees of freedom.
Quantum Hamiltonians are naturally self-adjoint, so their conventional random matrix models
are Hermitian; the most common example is Wigner matrices. Beyond quantum mechanics,
random matrices often appear without any symmetry condition in natural
phenomenological models. For example, the time evolution of many interacting agents ${\mathbf u}= (u_1, u_2, \ldots, u_n)$
may be described by a linear system of differential equations of the form
\begin{equation}\label{ode}
	\frac{\dif}{\dif t} {\mathbf u}(t) = X {\mathbf u}(t).
\end{equation}
Lacking any specific knowledge about how $u_j$ precisely influences the evolution of $u_i$, 
the simplest phenomenological model assumes that  the coefficient matrix $X$ is random.
Despite its simplicity, 
since the ground-breaking paper of May~\cite{4559589}, this model has been extensively used to describe the evolution of complex systems both in theoretical neuroscience, see e.g.~\cite{17155583, 10039285} and in 
mathematical ecology, e.g.~\cite{25768781,26198207}, see also 
a recent comprehensive review~\cite{Allesina2015}.
The problem is often presented in the form~\cite{10039285} 
\begin{equation}\label{de}
	\frac{\dif}{\dif t} {\mathbf u}(t) = (-I + gX) {\mathbf u}(t),
\end{equation}
where the identity matrix stands for a natural exponential decay at unit rate and the coupling constant 
$g>0$ explicitly expresses the strength of the random couplings. 
The main question is to tune $g$ so that the system is stable in the sense that it is
neither exponentially decaying nor exponentially increasing. The maximal growth rate
of the solution of~\eqref{de} is determined by the maximal real part of the spectrum of 
the coefficient matrix $-I + gX$. This motivates the  main task of this paper: to 
understand very accurately the real part of the rightmost eigenvalue of  a large
non-Hermitian random matrix.

We remark that a similar optimal stability question of~\eqref{de} for uniformly random initial data ${\mathbf u}(0)$  and 
when the size of ${\mathbf u}(t)$ measured in $\ell^2$-sense has been answered
in~\cite{PhysRevLett.81.3367,MR1755501} when the coefficients $x_{ij}$ are all Gaussian  and in~\cite{MR3816180,1908.05178} 
for more general distributions even beyond the i.i.d. case. The rightmost eigenvalue studied in the current paper is 
relevant when we consider the worst-case scenario, i.e. when we measure ${\mathbf u}(t)$ in maximum norm
and we take the supremum over all initial data ${\mathbf u}(0)$
(see Corollary~\ref{cor:ode} below).

To be more specific, we consider $n\times n$ random matrices $X$ with independent, identically distributed (i.i.d.) 
matrix elements, called the \emph{i.i.d. matrix ensemble}.  This is the non-Hermitian  counterpart of the Wigner ensemble. We choose the 
normalization  such that $x_{ij} \stackrel{\text{d}}{=} n^{-1/2}\chi$, for all $i,j$, where
$\chi$ is a fixed
complex centred random variable  with  $\E |\chi|^2=1$ and $\E \chi^2=0$.
This normalization guarantees that the spectrum of $X$ 
remains essentially within the unit disk, uniformly in $n$. 
We claim  our result and present the proof only for the complex case  but our basic method works 
for the real case as well. We will comment on the necessary modifications 
that we do not carry out in this paper for brevity. %so we rigorously  claim the result only for the complex case.

More precisely, the celebrated {\it Girko's circular law}, proven in increasing generality 
in~\cite{MR1428519, MR773436, MR2409368}, asserts
that the eigenvalue density of $X$ is uniform on the unit disk of the complex plane.  Furthermore, there are no  outlier eigenvalues
far away 
since the spectral radius $\rho(X)$  converges to $1$, see~\cite{MR863545,MR3813992,MR4408512,MR866352}.
A speed of convergence of order $n^{-1/2+\epsilon}$, for any small $\epsilon>0$, with very high  probability was recently established
in~\cite{MR3770875}.
Nevertheless some extremal eigenvalues do lie outside of the unit disk, hence 
$\max\Re \Spec(X)$ is slightly larger than one. It is well known that eigenvalues
genuinely fluctuate on scale $n^{-1/2}$ near the boundary of the unit disk, in
fact the local eigenvalue statistics in this regime is universal~\cite{MR4221653}.
Therefore we know that 
$$
n^{-1/2}\ll \max\Re\Spec(X) -1 \ll n^{-1/2+\epsilon}
$$ 
holds for any $\epsilon>0$ with very high probability and our goal is to find a more accurate asymptotics.
We remark that this natural question was posed in the first version on the arXiv of~\cite{MR4408512}
in Section 1.1.8. 
Beforehand, a leading order large deviation principle was established for $\max\Re\Spec(X)$
even for the more general elliptic ensemble in~\cite{MR4305622}
and the refined asymptotics was mentioned 
as an open question.

For the complex Gaussian case, i.e., when $\chi$ is a standard complex random variable (\emph{Ginibre ensemble}),
the eigenvalues form a determinantal process with an explicit correlation kernel computed first by Ginibre~\cite{MR173726}.
Based upon these formulas in our companion paper~\cite{2206.04443} we  recently gave a new short proof 
of the exact asymptotics:
\begin{equation}\label{Cmax}
	\max\Re \Spec(X) \stackrel{\text{d}}{=} 
	1 + \sqrt{\frac{\gamma_n}{4n}} + \frac{1}{\sqrt{4n\gamma_n}} \mathfrak{G}_n, \qquad \gamma_n:=
	\frac{\log n-5\log\log n-\log (2\pi^4)}{2},
\end{equation}
where the random variable $\mathfrak{G}_n$ converges to a Gumbel distribution
$$
\lim_{n\to \infty} \P (\mathfrak{G}_n\le t) = \exp{(-e^{-t})}
$$
for any fixed $t\in \R$  with an effective error term.  We also proved a similar result for the real Ginibre ensemble. 
These results without error term 
were first proven by Bender~\cite{MR2594353} in the complex case and by Akemann and Phillips~\cite{MR3192169}
in the real case
even for the more involved Gaussian elliptic ensembles where a sophisticated saddle point analysis for the correlation kernel  was necessary, while our proof in~\cite{2206.04443} is elementary. 
In particular,  %in~\cite{2206.04443} 
we obtained that typically
\begin{equation}\label{maxre}
	\max\Re \Spec(X)  =1+ \sqrt{\frac{\log n-5\log\log n-\log (2\pi^4)}{8n}} + \OO\Big(\frac{1}{\sqrt{n\log n}}\Big)
\end{equation}
for the Ginibre ensemble. In this paper we prove~\eqref{maxre} for any i.i.d. matrix ensemble,
in particular we show that the three-term asymptotics is universal in the sense that it is 
independent of the single entry distribution $\chi$.  The
Gumbel fluctuation is also expected to be universal; this is an open problem that we leave to future work.
However, our result~\eqref{maxre} already implies the tightness of $\mathfrak{G}_n$  
in~\eqref{Cmax} even for the i.i.d. case, see Remark~\ref{tight} below.

\begin{figure}
	\begin{tikzpicture}
		\begin{axis}[xmin=-1.3,xmax=1.3,ymin=-1.3,ymax=1.3,width=20em,height=20em,axis line style={white},ticks=none]
			\addplot [only marks,draw=none,mark size=1pt,mark options={draw opacity=0,fill=darkgray}] table [col sep=comma] {500c.csv};
			\draw(axis cs:0,0) [draw=lightgray, very thick] circle[radius=1];
			%\draw [fill=gray,fill opacity=.3,draw=gray] (axis cs:1.05,.25) -- (axis cs:1.2,.25) -- (axis cs:1.2,-.25) -- (axis cs:1.05,-.25) -- cycle;
			\node at (axis cs:1,1) {\small\color{gray}\(|z|=1\)};
			%\node at (axis cs:1.55,0) {\small\color{gray}\((\gamma n)^{-1/4}\)};
			%\node at (axis cs:1.3,-.35) {\small\color{gray}\((\gamma n)^{-1/2}\)};
			%\node  at (axis cs:-1,1) {\small\(\mathrm{Gin}(\C)\)};
		\end{axis}
	\end{tikzpicture}\qquad 
	\begin{tikzpicture}
		\begin{axis}[xmin=0.7,xmax=1.4,ymin=-0.35,ymax=0.35,width=20em,height=20em,axis line style={white},ticks=none]
			\addplot [only marks,draw=none,mark size=1pt,mark options={draw opacity=0,fill=darkgray}] table [col sep=comma] {1000c.csv}; 
			\draw [fill=lightgray,fill opacity=.3,draw=lightgray] (axis cs:1.08,.25) -- (axis cs:1.15,.25) -- (axis cs:1.15,-.25) -- (axis cs:1.08,-.25) -- cycle;
			\draw(axis cs:0,0) [draw=lightgray,very thick] circle[radius=1];
			\draw[<->,thick,black!50]
			(axis cs:1.16,.25) -- (axis cs:1.16,-.25) node[midway,above=1pt,rotate=270] {$(\gamma n)^{-1/4}$};
			\draw[<->,thick,black!50]
			(axis cs:1.08,-.26) -- (axis cs:1.15,-.26) node[midway,below=1pt] {$(\gamma n)^{-1/2}$};
			\draw[<->,thick,black!50]
			(axis cs:1.,0) -- (axis cs:1.115,0) node[midway,below=1pt] {$\sqrt{\gamma/n}$};
		\end{axis}
	\end{tikzpicture}
	\caption{The figure shows the eigenvalues of complex Ginibre matrices. The eigenvalues for the left figures have been computed for \(50\) independent Ginibre matrices of size \(50\times 50\), while for the right figure \(100\) independent matrices of size \(100\times 100\) have been sampled. The gray box on the right right hand side indicates the high-probability location of the eigenvalue with the largest real part.}\label{fig circ law}
\end{figure}
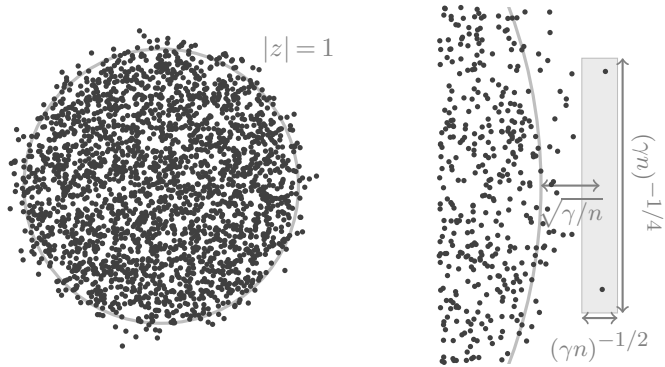

Extreme value statistics for independent random variables are fully described by the 
Fisher-Tippett-Gnedenko theorem, with the Gumbel distribution being one of the 
three main universal laws. The  three-term asymptotics for the scaling factor 
is also common, for example the maximum of $n$ standard Gaussian variables is 
given by 
$$
\sqrt{2 \log n} - \frac{\log\log n + \log (4\pi)}{2 \sqrt{2 \log n}} + \frac{1}{\sqrt{2 \log n}} \mathfrak{G}_n,
$$
where $\mathfrak{G}_n$ is again asymptotically Gumbel. 
While such precise asymptotics in the independent case is a fairly simple exercise by
the tail asymptotics of the individual random variables, it is remarkable that such precision
can be maintained in certain weakly correlated situations\footnote{Another such situation, introduced first in~\cite{MR2288065}, 
	is  the transition between Gumbel
	and Tracy-Widom distributions for GUE with an independent sizeable random deformation.}
such as $\max\Re \Spec(X)$.
The intuition is that effectively only  few rightmost eigenvalues compete for $\max\Re \Spec(X)$
and these eigenvalues are typically far away from each other, well beyond their correlation
length of order $n^{-1/2}$, hence they asymptotically form a Poisson point process and 
are essentially independent.  While this scenario could be directly verified for 
the Ginibre ensemble~\cite{2206.04443} (and even for the more general  Gaussian elliptic ensemble~\cite{MR3192169, MR2594353},
as well as for the  chiral two-matrix model with complex entries~\cite{MR2761338}),  its validity 
for a general i.i.d.\ matrix  is a highly nontrivial fact since explicit formulas for the eigenvalue correlation functions
are lacking.

We remark that starting with the ground-breaking paper of Fyodorov, Hiary and Keating~\cite{22680847}, see also \cite{MR3151088}, 
similar three-term asymptotics have recently been investigated for extreme
values  of  characteristic polynomials of various  random matrix ensembles,  
e.g.~\cite{MR3544809,MR3594368, MR3848227, MR3848391,MR4134939}, as well
as for the Riemann $\zeta$-function, e.g.~\cite{MR4164452,MR3911893, MR3851835, 1906.05783, 2007.00988}.
In a different context, the largest eigenvalues of minors have recently been shown to follow Gumbel distribution as well~\cite{2205.05732}.

\medskip

We now briefly comment on the novelties of our method to prove~\eqref{maxre};
more details will be given in Sections~\ref{sec:strategy} and~\ref{sec:sketch}. 
The standard method to extend any result on local eigenvalue statistics from the Gaussian case to a random matrix with a general 
entry distribution is the \emph{Green function comparison theorem (GFT)} going back to~\cite{MR2981427},
see also the related  \emph{Four moment theorem} of Tao and Vu~\cite{MR2784665}.  Direct application 
of GFT  in the bulk spectral regime for Hermitian 
matrices typically assumes that the third and fourth moments of the entry distribution also (almost) match,
and a more sophisticated dynamical approach relying on the Dyson Brownian motion is necessary to remove
these restrictive matching conditions~\cite{MR3699468, MR4416591, MR3914908}.
At the edge regime, however, two matching moments are sufficient for GFT~\cite{MR2871147}. Alternatively, at the edges
the Green functions can be controlled along the Ornstein-Uhlenbeck (OU) matrix flow 
for a very long time which allows one to compare a general  matrix with a Gaussian one directly.
This flow idea was first used in~\cite{MR3405746, MR3582818}
(see also~\cite{2102.04330, 2108.02728}) 
in the Hermitian context to investigate  the Tracy-Widom edge
universality and  later in~\cite{MR4221653} for the non-Hermitian situation in the edge regime of the circular law.
In the latter case first one translates the non-Hermitian eigenvalue problem to a Hermitian one
via Girko's formula
\begin{equation}\label{girko}
	\sum_{\sigma\in \Spec(X)} f(\sigma)=-\frac{1}{4 \pi} \int_{\C} \Delta f(z) A(z) \dd^2 z, \qquad
	A(z):= \int_0^\infty \Im \Tr G^{z}(\ii \eta) \dd \eta ,
\end{equation}
and then performs the analysis 
for a continuous family of Hermitized resolvents, 
parametrized by an additional spectral parameter $z\in \C$,  given by
\begin{align}\label{initial1}
	G^{z}(w):=(H^{z}-w)^{-1}, \qquad H^{z}:=\begin{pmatrix}
		0  &  X-z  \\
		X^*-\overline{z}   & 0
	\end{pmatrix},\qquad w \in \C\setminus \R.
\end{align}

Customarily one performs a cumulant expansion (see e.g. in~\cite{MR1411619, MR1689027, MR2561434, MR3678478, MR3941370}
for the random matrix context) for the time derivative 
of $F(\Tr G_t(w))$, where $F$ is a smooth test function and $G_t(w)$ is the Green function 
at time $t$. The spectral parameter $w$ is chosen sufficiently close to the real axis
to detect individual eigenvalues, i.e. $\eta:=\Im w$ is smaller than the typical eigenvalue spacing,
e.g. $\eta\ll n^{-2/3}$ in the Hermitian edge regime and $\eta\ll n^{-3/4}$ in the non-Hermitian edge regime
where the spectral density  of $H^z$ develops a cusp singularity at zero.
Typically the first and second order terms in the cumulant expansion are automatically cancelled
by the choice of the OU process, the third and fourth order terms
require careful estimates, while terms with higher order cumulants can be estimated quite crudely.
The estimates are done via the \emph{optimal local laws} that identify the leading
deterministic term of the Green functions plus an error term. In the edge regimes where $\eta:=\Im w$ is
typically much larger than $1/n$, the cumulant expansion can
be \emph{iterated}: in every step one may gain an additional factor $\psi:= 1/(n\eta)\ll 1$ in 
the error terms from the so-called \emph{un-matched indices} 
(Definition~\ref{def:unmatch_shift}),
while the leading deterministic terms can be computed and the first non-vanishing one
gives the final size.  We need to exploit an explicit cancellation of the
leading term after the $z$-integration in~\eqref{girko}, forcing us to expand beyond the usual order.
The iterative cumulant expansion has been systematically developed in~\cite{2102.04330, 2108.02728}
after several previous works using the  iterative 
gain from un-matched indices~\cite{MR2871147, MR3119922} combined with cancellations
of leading deterministic terms in certain situations~\cite{MR3800840, MR4288336, MR4089498}.

The main difference between the current work and all previous applications of sophisticated cumulant expansions
along a GFT proof
is that now we work in a very \emph{atypical} regime which means that all natural a priori estimates 
from local laws are off by a large factor (of size $n^{1/4}$). 
Indeed, due to the curvature of the unit circle near $1$,
the eigenvalues that may typically contribute to $\max\Re \Spec(X)$ are located in an elongated vertical box
of size $n^{-1/2}\times n^{-1/4}$ (modulo logarithmic factors) with center around $1+\sqrt{\gamma_n/4n}$
and this box contains typically finitely many (independently of $n$) eigenvalues, see Fig.~\ref{fig circ law} -- this 
was proven for the Gaussian case in~\cite{2206.04443}. Therefore to obtain a lower
and upper bound on   $\max\Re \Spec(X)$ we will need to use~\eqref{girko} 
for a smooth test functions $f$ supported on such box and we need to control~\eqref{girko} in expectation and  variance sense.

If $f$ in~\eqref{girko} is a smooth  function supported in this box then $\int_\C |\Delta f(z)|\dd^2z\sim n^{1/4}$
is unusually large due to  \emph{strong anisotropy} of the box.
The typical size of the fluctuation of $A(z)$ by local law is $\int_0^\infty \eta^{-1}\dd \eta \sim O(1)$
(ignoring logarithmic singularities). Thus the naive size of the fluctuation in the rhs. of Girko's formula in (\ref{girko}) is of order $n^{1/4}$ for a quantity
that is only of order one by its lhs. This overestimate has a drastic effect on the usual cumulant expansions.
Higher order terms in the cumulant expansion of $F( \int_\C \Delta f(z) A(z)\dd^2z)$
will involve higher powers of the problematic quantity $\int (\Delta f)  A$ whose a priori size we do not
control effectively. For smooth and bounded test functions $F$
the standard iterative cumulant expansion, similar to~\cite{2102.04330, 2108.02728}, is  still
effective but only if $n^{1/4}\psi= n^{1/4}/(n\eta)\ll 1$, i.e. in the regime where $\eta\gg n^{-3/4}$.
We circumvent this difficulty by considering only the expectation and the variance of $\int (\Delta f)  A$
instead of a general test function $F$ which has the advantage that the Taylor expansion of $F$ stops
at first or second order. This restricted choice of $F$ is the main reason why our current result is able to control only the
size of $\max\Re \Spec(X)$ in~\eqref{maxre} but not yet its  Gumbel fluctuation. 

The complementary  $\eta\lesssim n^{-3/4}$  regime is not accessible by robust expansions.
In fact, the regime $\eta\ll n^{-3/4}$
is dominated by the smallest (in modulus) eigenvalue $\lambda^z$
of $H^z$ (equivalently, the lowest singular value of $X-z$), especially by its lower tail behaviour. Two independent special effects
need to be exploited simultaneously. First, there is a level repulsion between $\lambda^z$ and $-\lambda^z$ (since the spectrum of $H^z$
is symmetric to the origin, hence $-\lambda^z$ is also an eigenvalue), which effectively suppresses the event that $|\lambda^z|$
is much smaller than its natural scale $n^{-3/4}$. Second, the density of non-Hermitian eigenvalues (of $X$) are suppressed by a factor $e^{-n(|z|-1)^2/2}$ in the regime where $|z|\ge 1$, expressing the
very strong concentration of the spectral radius near $1$.  Heuristically, this gives an additional small factor 
of order  $e^{-n(|z|-1)^2/2}$ for the lower tail of $\lambda^z$ as well. However, note that 
having a non-Hermitian eigenvalue extremely close to  $z$ and $\lambda^z$ being
unusually small is not an effectively controllable relation, even though $z\in \Spec(X)$ 
is equivalent to $\lambda^z=0$. Both effects are extremely delicate and 
cannot be obtained directly  for a general i.i.d. ensemble, but they can be extracted from 
the corresponding Ginibre ensemble via explicit calculations. We therefore first establish
a very accurate lower tail estimate on $\lambda^z$ in the Ginibre case (see Proposition~\ref{prop:tail} below), then
via a separate GFT argument we transfer its consequence to the i.i.d. ensemble by obtaining an improved bound on 
$\E  \Im \Tr G^{z}(\ii \eta)$ (see Proposition~\ref{prop1} below).  Up to an intermediate cutoff scale $\eta\ll n^{-7/8}$ 
this bound is sufficient to overcome the $n^{1/4}$ loss from $\int |\Delta f|$, ensuring that this small-$\eta$ regime is 
negligible in the expectation sense and proving that only the $\eta\gtrsim n^{-7/8}$ regime matters in~\eqref{girko}.

Finally, we revisit the iterative cumulant expansion for the large-$\eta$ regime and
use that we are interested only in the expectation and variance instead of a general $F$.
This means that the  factor $\int |\Delta f|\sim n^{1/4}$ occurs at most twice which 
can still be compensated  by the improved estimate on $\E  \Im \Tr G^{z}(\ii \eta)$ in the cumulant expansions.
For the comparison argument
we also need a variance bound on the large-$\eta$ regime for the Gaussian case, which is not available
directly, but which we deduce indirectly from the variance of the lhs. of~\eqref{girko} (that is available
via the Ginibre kernels~\cite{2206.04443}) and from the vanishing variance of the small-$\eta$ regime.
The optimal variance bound  for the entire $\eta\ll n^{-3/4}$ regime 
would require the precise correlation between $\lambda^z$ and $\lambda^{z'}$ for two different
spectral parameters $z, z'$ -- an information that is not available even in the Gaussian case.
Nevertheless, the suboptimal estimate, ignoring the decorrelation for large $z-z'$, is still  sufficient for 
our smaller regime $\eta\ll n^{-7/8}$.  This is another independent reason for choosing the 
threshold $n^{-7/8}$ for splitting the $\eta$-integral in~\eqref{girko}.

Note that we  use the $n^{-7/8}$ threshold to 
distinguish between the negligible small-$\eta$ regime and the large-$\eta$ regime requiring separate
GFT comparisons, although the natural cutoff threshold should have been at $n^{-3/4}$ (the threshold $n^{-7/8}$ is only a technical choice).

In summary, our proof is much more  involved   than a typical  direct iterative GFT argument 
and requires to choose an "unnatural" threshold $n^{-7/8}$ due to 
two main reasons: (i) we do not have almost matching a priori bounds due the anisotropy of the regime we consider, 
and (ii) the necessary direct information on the correlation between $\lambda^z$ and $\lambda^{z'}$ 
is lacking even in the Gaussian case.

\subsection*{Notations and conventions}

We introduce some notations we use throughout the paper. For integers \(k,l\in\N \) with $k\leq l$ we use the notation 
\(\llbracket k, l \rrbracket:= \{k,k+1,\dots, l\}\).  For positive quantities \(f,g\) we write \(f\lesssim g\) and \(f\sim g\) if \(f \le C g\) or \(c g\le f\le Cg\), 
respectively, for some constants \(c,C>0\) which depend only on the constants appearing in~\eqref{eq:hmb}.  For $n$-dependent positive
sequences $f=f_n, g=g_n$ we also introduce $f\ll g$ indicating that $f_n=o(g_n) $.
We denote vectors by bold-faced lower case Roman letters \(\mathbf{x}, \mathbf{y}\in\C ^k\), for some \(k\in\N\). 
Vector and matrix norms, \(\lVert{\mathbf x}\rVert\) and \(\lVert A\rVert\), indicate the usual Euclidean norm 
and the corresponding induced matrix norm. For any \(2n\times 2n\) matrix \(A\) we use 
the notation \(\langle A\rangle:= (2n)^{-1}\mathrm{Tr}  A\) to denote the normalized trace of \(A\). 
Moreover, for vectors \({\mathbf x}, {\mathbf y} \in\C ^n\) and matrices \(A,B\in \C ^{2n\times 2n}\) we define
\[ 
\langle{\mathbf x},{\mathbf y}\rangle:= \sum \overline{x}_i y_i, \qquad \langle A,B\rangle:= \langle A^*B\rangle. 
\]
Moreover, we use $\Delta = 4\partial_z\partial_{\bar z}$ to denote the usual Laplacian
and \(\mathrm{d}^2 z \) denotes the Lebesgue measure on $\C$. For any function $h:\C \rightarrow \C$, we define the $L^{p}$-norm by 
$\|h\|_p^p:=\int_\C |h(z)|^p \dd^2 z.$

We will use the concept of ``with very high probability'' meaning that for any fixed \(D>0\) the probability of the event is bigger
than \(1-n^{-D}\) if \(n\ge n_0(D)\). Moreover, we use the convention that \(\xi>0\) denotes an arbitrary small 
constant which is independent of \(n\). Finally, we introduce the notion of 
\emph{stochastic domination} (see e.g.~\cite{MR3068390}): given two families of non-negative random variables
\[
X=\left(X^{(n)}(u) : n\in\N, u\in U^{(n)} \right)\quad\text{and}\quad Y=\left( Y^{(n)}(u) : n\in\N, u\in U^{(n)} \right)
\] 
indexed by \(n\) (and possibly some parameter \(u\)  in some parameter space $U^{(n)}$), 
we say that \(X\) is stochastically dominated by \(Y\), if for all \(\xi, D>0\) we have \begin{equation}\label{stochdom}
	\sup_{u\in U^{(n)}} \mathbb{P}\left[X^{(n)}(u)>n^\xi  Y^{(n)}(u)\right]\leq n^{-D}
\end{equation}
for large enough \(n\geq n_0(\xi,D)\). In this case we use the notation \(X\prec Y\) or \(X= \OO_\prec(Y)\).
We also use the convention that \(\xi>0\) denotes an arbitrary small constant which is independent of \(n\). We often use the notation 
$\prec$ also for deterministic quantities, then the probability in~\eqref{stochdom} is
zero for any $\xi>0$ and sufficiently large $n$. 

\section{Main result and the proof ingredients}\label{sec:str}

In this section we first formulate our main result precisely for the complex symmetry class.
Then in Section~\ref{sec:strategy}
we collect some key ingredients of the proof from the literature: the local circular law,
the strong concentration of the spectral radius, Girko's formula, the local law for the Hermitized matrix 
$H^z$, and most importantly we present a new lower tail bound on the smallest (in modulus)
eigenvalue of $H^z$ (Proposition~\ref{prop:tail} with its proof presented in the appendix).
In the brief Section~\ref{sec:sketch} we informally explain the main strategy
that will be formalized in Section~\ref{sec:proof_main_theorem}.
Finally, in Section~\ref{sec:real} we comment on the extension of our argument to the real symmetry class.

\subsection{Statement of the main result}
We consider $n\times n$ matrices $X$ with independent identically distributed (i.i.d.) entries $x_{ab}\stackrel{\mathrm{d}}{=}n^{-1/2}\chi$. On the $n$-independent random variable $\chi$ we make the following assumption:

\begin{assumption}
	\label{ass:mainass}
	We assume that $\E \chi=0$, $\E |\chi|^2=1$; additionally in the complex case we also assume that $\E \chi^2=0$. Furthermore, for any $p\in\N$ we assume that there exists constants $C_p>0$ such that
	\begin{equation}
		\label{eq:hmb}
		\E\big|\chi^p\big|\le C_p.
	\end{equation}
Moreover, we assume that there exists $\alpha,\beta>0$ such that the probability density of $\chi$, denoted by $g$, satisfies 
	\begin{align}\label{assumption_b}
		g \in L^{1+\alpha}(\mathbb{F}), \quad \|g\|_{1+\alpha} \leq n^{\beta}, \qquad \mathbb{F}=\R~\mathrm{or}~\C.
\end{align}
\end{assumption}

\begin{remark}\label{TV_tech} 
	The condition on the density~(\ref{assumption_b}) is used only 
	  to control the  unlikely event 
 that there is a tiny singular value of $X-z$ in a very simple way; see~(\ref{density_bound}) below. 
	We make this assumption only  to simplify the presentation of the proof. 
	It can easily be removed with a separate argument from
	in~\cite[Section 6.1]{MR3306005} (see also a slightly streamlined version in \cite[Section 2.2]{1510.02987}) 
	as follows.
	To deal with a  random matrix $X$ failing to satisfy~(\ref{assumption_b}), one may add a tiny independent Gaussian 
	component  $n^{-\gamma} X_{\rm Gauss}$ to $X$ for some large fixed $\gamma>0$
	 to achieve a probability density  for the entries that satisfies~(\ref{assumption_b}),
	hence our main results hold for $X+n^{-\gamma} X_{\rm Gauss}$.
	This tiny component $n^{-\gamma} X_{\rm Gauss}$ can then be removed by using the proof of  \cite[Theorem 23]{MR3306005}
	(or its refinement \cite[Lemma 4]{1510.02987})
	that combines a sampling idea with the standard moment matching technique\footnote{This theorem shows that if the first four moments
	of the single entry distributions of two ensembles $X$ and $X'$ coincide, then their microscopic local statistics are close
	with an effective error. With sufficiently high moment matching a straightforward modification 
	of the proof of \cite[Theorem 23]{MR3306005} yields 
	much finer error estimates up to any polynomial order in $1/n$ which can be used to offset the 
	anisotropic loss of our test function $f$.
	 The same conclusion also holds if the moments
	do not exactly match, but they differ only
	 up to an order $n^{-\gamma}$.}.
	 We will not present the details here
	since they are fairly standard and they are independent of our main arguments. 
	
	% However the proof is quite sophisticated, so we will not pursue this goal in the present paper. 
	\nc
\end{remark}

Let $\{\sigma_i\}_{i\in \llbracket 1, n \rrbracket}$ be the eigenvalues of $X$, then our main result is the following:
\begin{theorem}\label{main}
	Let $X$ be an $n\times n$ matrix satisfying Assumption~\ref{ass:mainass} in the complex case, and define
	\begin{align}
		\label{eq:gamma}
		\gamma_n:= \frac{\log n-5 \log \log n - \log (2\pi^4)}{2}.  
	\end{align}
	Then
	\begin{align}\label{real_part}
		\lim_{n\to \infty}\P \left(\left|\max_{i \in \llbracket 1, n \rrbracket} \Re \sigma_i-1-\sqrt{\frac{\gamma_n }{4n}} \right| \geq \frac{C_n}{\sqrt{4 n \gamma_n}}\right)=0,
	\end{align}	
	for any sequence\footnote{ Following our proof we may obtain an effective
		control on  the probability in~\eqref{real_part} of order $O(C_n^{-\tau} +  n^{-\tau})$ 
		for some fixed $\tau>0$. } $C_n \rightarrow \infty$.
\end{theorem}
We  remark that Theorem~\ref{main} 
with the same proof holds if instead of the
eigenvalue with the largest real part we consider the largest eigenvalue in any given direction,
i.e. \eqref{real_part} holds for $\max_i \Re (e^{\ii \theta} \sigma_i)$ with any fixed $\theta\in \R$, but for simplicity 
we consider the $\theta=0$  case.
We also note that the matrix elements of $X$ do not 
necessarily have to be identically distributed. Our proof works 
with minor modifications 
as long as all the entries $\chi_{ab} = \sqrt{n} x_{ab}$ satisfy
Assumption 2.1 uniformly for any $a,b$\nc, but for simplicity we consider the i.i.d. case.

We stated our main Theorem~\ref{main} only for the complex case. The same result holds for the real case
but  we give a complete proof only for the complex case; we explain 
the reason  in Section~\ref{sec:real}.

Our precise estimate on $\max_{i } \Re \sigma_i$ has the following immediate corollary
on the solution to~\eqref{de} with an i.i.d. matrix $X$.
\begin{corollary}\label{cor:ode}
	Let  ${\mathbf u}(t)$ be the solution to~\eqref{de} with deterministic initial condition ${\mathbf u}(0)$
	and with coupling constant $g=g_n$. Then for any sequence $C_n\to \infty$,   the following 
	statements hold with probability tending to one as $n\to \infty$:
	\begin{itemize}
		\item[(i)]  If $g \le 1- \sqrt{\frac{\gamma_n }{4n}}-\frac{C_n}{\sqrt{4 n \gamma_n}}$, then
		$$
		\limsup_{t\to\infty} \frac{\max_i |u_i(t)|}{ \max_i |u_i(0)|} = 0.
		$$
		\item[(ii)]  If $g \ge 1- \sqrt{\frac{\gamma_n }{4n}}+\frac{C_n}{\sqrt{4 n \gamma_n}}$, then
		$$
		\limsup_{t\to\infty}  \frac{\max_i |u_i(t)|}{ \max_i |u_i(0)|} = \infty.
		$$
	\end{itemize}
\end{corollary}
This corollary shows  that to prevent the 
decay or blow-up  of the solution for arbitrary long time, i.e. to remain in the so-called \emph{stable regime} in 
many applications, it is \emph{necessary} to fine tune the coupling constant very accurately.
Our main theorem can also be used for \emph{sufficient} conditions for stability up to a certain time
scale of order $\sqrt{4n/\gamma_n}$  but we refrain from formalizing such statement.
We note that the maximum norm on ${\mathbf u}$ and the deterministic initial condition
indicate that we considered the worst-case scenario. As we mentioned
in the introduction, the problem with random initial data and measuring stability in  $\ell^2$-sense
has been investigated earlier and gives a somewhat different optimal tuning for $g$. 

\begin{remark}\label{tight}
	Theorem~\ref{main} implies that the sequence of random variables
	$$
	\mathfrak{G}_n:= \sqrt{ 4 n \gamma_n}\left[\max_{i \in \llbracket 1, n \rrbracket} \Re \sigma_i-1-\sqrt{\frac{\gamma_n }{4n}} \right]
	$$
	is tight, hence it has   subsequential  limits by Prokhorov's theorem. The limit
	is conjectured to be unique and to be  the standard Gumbel distribution with distribution function $F(x)=\exp( - e^{-x})$
	in the complex case and $F(x)=\exp(- \frac{1}{2} e^{-x})$ in the  real case. For the Ginibre ensembles this conjecture was recently  proven in~\cite{2206.04443}.
\end{remark}

\subsection{Proof ingredients}\label{sec:strategy}

First we recall two earlier  results that locate $\max_{i \in\llbracket 1, n \rrbracket} \Re \sigma_i$ on a cruder scale than our eventual
target precision.
The local circular law~\cite{MR3230004}  implies that for any fixed $\tau>0$
there is an eigenvalue in the rectangle
\begin{align}\label{omega_0}
	\Omega_0:=\left[1-\frac{n^{\tau}}{\sqrt{n}},~1+\frac{n^{\tau}}{\sqrt{n}}\right]\times \left[-\frac{n^{\tau/2}}{n^{1/4}},~\frac{n^{\tau/2}}{n^{1/4}}\right],
\end{align}
with very high probability, using  the curvature of the boundary. In particular, this shows that typically 
$\max_{\llbracket 1, n \rrbracket} \Re \sigma_i\ge 1- n^{-1/2+\tau}$. Furthermore, 
by~\cite[Theorem 2.1]{MR3770875} we have 
a strong concentration estimate for the spectral radius $\rho(X)=\max_{i \in \llbracket 1, n \rrbracket} |\sigma_i|$:
\begin{equation}
	\label{eq:specrad}
	\left|\rho(X)-1\right|\le\frac{n^{\tau}}{\sqrt{n}}
\end{equation}
for any $\tau>0$,
with very high probability. 
In particular, fixing a small $\tau>0$, \eqref{omega_0} and \eqref{eq:specrad} imply
that the rightmost eigenvalue is located in $\Omega_0$, see Fig.~\ref{fig Omega}.

Next we will show that with vanishing probability, 
\begin{align}\label{vanishing_prob}
	\left|\max_{i \in \llbracket 1, n \rrbracket} \Re \sigma_i-1-\sqrt{\frac{\gamma_n }{4n}} \right| \geq \frac{C_n}{\sqrt{4 n \gamma_n}},
\end{align}
with $\gamma_n$ from (\ref{eq:gamma}) and $C_n$ being any sequence $C_n \rightarrow \infty$. To see this, we define the following two sub-rectangles of $\Omega_0$ with the same height:
\begin{align}\label{omega_1}
	\Omega_1:=\left[1+\sqrt{\frac{\gamma_n}{4n}}-\frac{C_n}{\sqrt{4 n \gamma_n}},~1+\sqrt{\frac{\gamma_n}{4n}}+\frac{C_n}{\sqrt{4 n \gamma_n}}\right]\times \left[-\frac{n^{\tau/2}}{n^{1/4}},~\frac{n^{\tau/2}}{n^{1/4}}\right],
\end{align}
\begin{align}\label{omega_2}
	\Omega_2:=\left[1+\sqrt{\frac{\gamma_n}{4n}}+\frac{C_n}{\sqrt{4 n \gamma_n}},~1+\frac{n^{\tau}}{\sqrt{n}}\right]\times \left[-\frac{n^{\tau/2}}{n^{1/4}},~\frac{n^{\tau/2}}{n^{1/4}}\right],
\end{align}
see Fig.~\ref{fig Omega}. From now on without loss of generality we may assume $1\ll C_n\ll (\log n)^{1/2}$.

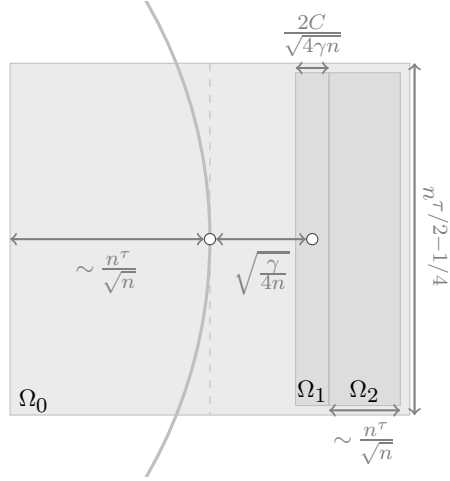
\begin{figure}
	\begin{tikzpicture}
		\begin{axis}[xmin=0.5,xmax=1.5,ymin=-0.5,ymax=0.5,width=25em,height=25em,axis line style={white},ticks=none]
			%\addplot [only marks,draw=none,mark size=1pt,mark options={draw opacity=0,fill=darkgray}] table [col sep=comma] {1000c.csv}; 
			\draw [fill=lightgray,fill opacity=.3,draw=lightgray] (axis cs:.58,.37) -- (axis cs:1.42,.37) -- (axis cs:1.42,-.37) -- (axis cs:.58,-.37) -- cycle;
			\draw [fill=lightgray,fill opacity=.3,draw=lightgray] (axis cs:1.18,.35) -- (axis cs:1.25,.35) -- (axis cs:1.25,-.35) -- (axis cs:1.18,-.35) -- cycle;
			\draw [fill=lightgray,fill opacity=.3,draw=lightgray] (axis cs:1.25,.35) -- (axis cs:1.4,.35) -- (axis cs:1.4,-.35) -- (axis cs:1.25,-.35) -- cycle;
			\draw [draw=lightgray,dashed] (axis cs:1,.37) -- (axis cs:1,-.37);
			\draw(axis cs:0,0) [draw=lightgray,very thick] circle[radius=1];
			\node[draw=none] at (axis cs:.63,-.34) (a) {$\Omega_0$};
			\node[draw=none] at (axis cs:1.215,-.32) (a) {$\Omega_1$};
			\node[draw=none] at (axis cs:1.325,-.32) (a) {$\Omega_2$};
			\draw[<->,thick,black!50] (axis cs:1.43,.37) -- (axis cs:1.43,-.37) node[midway,above=1pt,rotate=270] {$n^{\tau/2-1/4}$};
			\draw[<->,thick,black!50](axis cs:1.18,.36) -- (axis cs:1.25,.36) node[midway,above=1pt] {$\frac{2C}{\sqrt{4\gamma n}}$};
			\draw[<->,thick,black!50](axis cs:1.25,-.36) -- (axis cs:1.4,-.36) node[midway,below=1pt] {$\sim \frac{n^{\tau}}{\sqrt{n}}$};
			\node [circle,draw=darkgray,fill=white,inner sep=1.5pt] (one) at (axis cs:1,0) {};
			\node [circle,draw=darkgray,fill=white,inner sep=1.5pt] (ctr) at (axis cs:1.215,0) {};
			\draw[<->,thick,black!50](axis cs:.58,0) -- (one) node[midway,below=1pt] {$\sim \frac{n^{\tau}}{\sqrt{n}}$};
			\draw[<->,thick,black!50](one) -- (ctr) node[midway,below=1pt] {$\sqrt{\frac{\gamma}{4n}}$};
			%\draw[<->,thick,black!50]
			%(axis cs:1.,0) -- (axis cs:1.115,0) node[midway,below=1pt] {$\sqrt{\gamma/n}$};
		\end{axis}
	\end{tikzpicture}
	\caption{The figure shows the domains \(\Omega_0,\Omega_1,\Omega_2\). The set \(\Omega_1\) is chosen such that the number of eigenvalues therein are approximately Poisson-distributed with parameter \(2\sinh(C_n)\sim e^{C_n}\), while in \(\Omega_2\) no eigenvalues are expected with high probability.}\label{fig Omega}
\end{figure}

To prove the upper bound in (\ref{vanishing_prob}), 
it suffices to show that 
\begin{align}\label{omega2_exp}
	\P(\#\{\sigma_i \in \Omega_2\} \geq 1)\leq	\E [\#\{\sigma_i \in \Omega_2\}]=o(1),
\end{align}
here by $o(1)$ we denote a quantity that goes to zero as $n\to\infty$. To prove the matching lower bound in (\ref{vanishing_prob}), 
$\P\left(\#\{\sigma_i \in \Omega_1\}=0\right)=o(1)$, we need not only the expectation bound
\begin{align}\label{omega1_exp}
	\E [\#\{\sigma_i \in \Omega_1\}] \geq c_0, \qquad 
\end{align}
for some $n$-independent constant $c_0>0$, but also the concentration bound
\begin{align}\label{concentrate}
	\E \big|\#\{\sigma_i \in \Omega_1\}-\E[\#\{\sigma_i \in \Omega_1\}]\big|=o(1)\E[\#\{\sigma_i \in \Omega_1\}].
\end{align}
More precisely, using the Markov inequality in combination with (\ref{omega1_exp}) and (\ref{concentrate}), we get
\begin{align}\label{markov_argu}
	\P\left(\#\{\sigma_i \in \Omega_1\}=0\right) \leq \P \Big( \big|\#\{\sigma_i \in \Omega_1\}-\E[\#\{\sigma_i \in \Omega_1\}]\big| \geq \frac{\E[\#\{\sigma_i \in \Omega_1\}]}{2}\Big)=o(1).
\end{align}

Now we have reduced the proof of (\ref{vanishing_prob}) to proving the expectation bounds in (\ref{omega2_exp}), (\ref{omega1_exp}) and the concentration bound in (\ref{concentrate}).  Their proof consist of two main steps. 
We first use the explicit formulae for the eigenvalue correlation functions
of the Ginibre ensemble to show that (\ref{omega2_exp})-(\ref{concentrate}) hold true for the Gaussian case. 
In fact, we will need a slightly modified version of these estimates where the 
counting functions are replaced by a smooth test functions supported on the corresponding $\Omega$ domains,
see Lemma~\ref{Ginibre_estimate} below, whose proof is an easy consequence of 
our estimates on the Ginibre correlation kernel from~\cite{2206.04443}. 
In the second step we then use the Green function comparison theorem (GFT) to extend the above estimates to general i.i.d. matrices.
In the rest of this section we now introduce some tools for this second step and explain the strategy.

To perform a GFT analysis
we rely on Girko's Hermitization formula \cite{MR773436} in the form introduced by Tao and Vu \cite{MR3306005}:
\begin{align}\label{linear_stat}
	\sum_{i=1}^{n} f(\sigma_i)=&-\frac{1}{4 \pi} \int_{\C} \Delta f(z) \int_0^T \Im \Tr G^{z}(\ii \eta) \dd \eta \dd^2 z+\frac{1}{4 \pi} \int_{\C} \Delta f(z) \log |\det (H^{z}-\ii T)| \dd^2 z,
\end{align} 
for any $T>0$ and for any compactly supported smooth test function $f \in C_c^{2}(\C)$.
Here we recall the definition of the $2n \times 2n$ Hermitian matrix  $H^{z}$ and its resolvent $G^{z}$ from~\eqref{initial1}:
\begin{align}\label{initial}
	H^{z}:=\begin{pmatrix}
		0  &  X-z  \\
		X^*-\overline{z}   & 0
	\end{pmatrix}, \qquad G^{z}(w):=(H^{z}-w)^{-1}, \quad w \in \C\setminus \R,~z\in \C.
\end{align}
The $2\times 2$ block structure of $H^z$ induces a symmetric spectrum around zero, i.e. the eigenvalues of $H^{z}$ are $\{\lambda^{z}_{\pm i}\}_{i\in \llbracket 1, n \rrbracket}$ (labelled in a non-decreasing order) with $\lambda_{-i}^z=-\lambda_i^z$ for $i\in \llbracket 1, n \rrbracket$. Note that $\{\lambda_i^{z}\}_{i\in \llbracket 1, n \rrbracket}$ exactly coincide with the singular values of $X-z$.  As a consequence of the spectral symmetry of $H^{z}$, we find that
$$G_{vv}^{z} (\ii \eta)=\ii \Im G_{vv}^{z}(\ii \eta), \quad \Im G^{z}_{vv}(\ii \eta) >0, \qquad   v \in [2n], \quad \eta>0.$$

Our fundamental input, the \emph{local law for $G^z$} stated below in Theorem~\ref{local_thm}, asserts
that as $n\to \infty$ the resolvent $G^z$ becomes approximately deterministic.  Its deterministic approximation is given by
\begin{align}\label{Mmatrix}
	M^{z}(\ii \eta)=\begin{pmatrix}
		m^{z}(\ii \eta)  &  \mathfrak{m}^{z}(\ii \eta)  \\
		\overline{\mathfrak{m}^{z}}(\ii \eta)   & m^{z}(\ii \eta)
	\end{pmatrix},
\end{align}
where ${m}^z$ is the unique solution of the scalar equation
\begin{align}
	\label{eq:scmde}
	-\frac{1}{{m}^z(w)}=w+{m}^z(w)-\frac{|z|^2}{w+{m}^z(w)}, \quad \mbox{with}\quad \mathrm{Im}[m^z(w)]\mathrm{Im}w>0,
\end{align}
and 
\begin{align}\label{m_1m_2} 
	\mathfrak{m}^{z}(\ii \eta):=-zu^z(\ii \eta), \quad u^z(\ii \eta):=\frac{\Im {m}^z(\ii \eta)}{\eta+\Im {m}^z(\ii \eta)}.
\end{align}
By taking the real part of \eqref{eq:scmde} it readily follows that on the imaginary axis $m^z$ is purely imaginary, hence ${m}^z(\ii \eta)=\ii \Im {m}^z(\ii \eta)$ (which also implies that $u^z(\ii\eta)$ is real). In addition, by \cite[Lemma 3.3]{MR3770875} we have that
\begin{align}\label{m_1}
	\Im m^z(\ii \eta) \sim \begin{cases}
		\frac{\eta}{|1-|z|^2|+\eta^{2/3}}, &\qquad |z| > 1\\
		\eta^{1/3}+|1-|z|^2|^{1/2} , &\qquad |z| \leq 1
	\end{cases},
	\qquad  0\leq \eta\leq 1.
\end{align}
With these notations we have the
local law for the resolvent $G^z$ on the imaginary axis:
\begin{theorem}[ Theorem 5.2 \cite{MR3770875},  Proposition 1 \cite{MR4221653}]\label{local_thm}
	For any deterministic vectors $\mathbf{x}, \mathbf{y}\in\C^{2n}$
	and matrix $A\in \C^{2n \times 2n}$, for any $z$ with $-Cn^{-1/2} \lesssim |z|-1 \leq n^{-1/2+\tau}$ and $n^{-1} \leq \eta \leq 1$, we have
	\begin{align}\label{isotropic}
		\big| \langle \mathbf{x},  (G^{z}(\ii \eta)-M^{z}(\ii \eta)) \mathbf{y} \rangle\big| \prec \|\mathbf{x}\| \|\mathbf{y}\| \big(\frac{1}{n^{1/2} \eta^{1/3}}+\frac{1}{n\eta}\big),
	\end{align}
	\begin{align}\label{average}
		\big|\big\<A\big( G^{z}(\ii \eta)-M^{z}(\ii \eta)\big)\big\>\big| \prec \frac{\|A\|}{n\eta}.
	\end{align}
\end{theorem}

To apply Girko's formula~\eqref{linear_stat}
for the proof  of (\ref{omega2_exp})-(\ref{concentrate}) we need to regularize the indicator function on the corresponding  
$\Omega=\Omega_1,\Omega_2$ domains,  and we also split the $\eta$-integration into two regimes
which require different analysis. Hence, with some properly chosen smooth cut-off function $f$ (see \eqref{f1function}--\eqref{deri_cond_2} below),
we have
\begin{align}\label{linear_stat_2}
	\#\{\sigma_i \in \Omega\} \approx \sum_{i=1}^{n} f(\sigma_i) \approx -\frac{1}{4 \pi} \int_{\C} \Delta f(z) \left(\int_0^{\eta_0}+\int_{\eta_0}^{T}\right) \Im \Tr G^{z}(\ii \eta) \dd \eta \dd^2 z=:I_{0}^{\eta_0}+ I_{\eta_0}^{T},
\end{align}
with $T:=n^{100}$ and $\eta_0$ is an intermediate cutoff level $n^{-7/8-\tau}$ with the fixed $\tau>0$ from~\eqref{omega_0}. 
The small-$\eta$ regime, $I_0^{\eta_0}$, and the large-$\eta$ regime, $ I_{\eta_0}^{T}$, are analysed
separately. 

For very small $\eta$, the local law (Theorem~\ref{local_thm}) does not effectively control the resolvent $G^z(\ii \eta)$ 
as it is dominated by the smallest (in modulus) eigenvalue $\lambda_1^z$ of $H^z$
(equivalently,  $\lambda_1^z$ is the smallest singular value of $X-z$).
We need a separate lower tail estimate
for it, which is done again in two steps: first for Ginibre matrices and then we extend it to i.i.d. matrices via GFT.
Besides the level repulsion effect, this accurate estimate also contains an additional  small factor 
due to the fact that $z$ is far outside of the unit disk, although this second effect is 
needed only for the Ginibre ensemble in this paper. 

We now state the precise lower tail result for the Ginibre ensemble.
Its proof, which is given in the appendix,
relies on the explicit formula for the correlation functions of the eigenvalues of $(X-z)^*(X-z)$ from \cite{MR2162782},
or, alternatively, on the supersymmetric representation of its resolvent from~\cite{MR4408004}.
In the sequel we denote by $\P^\mathrm{Gin}$, $\E^\mathrm{Gin}$ and $\V^\mathrm{Gin}$ the corresponding probability, expectation and variance.  

\begin{proposition}\label{prop:tail}
	Fix $\delta:=|z|^2-1$ with $n^{-1/2}\ll\delta\ll 1$  and let  $\lambda_1^z$ be the smallest singular value of $X-z$, 
	where $X$ is a  complex Ginibre matrix.
	Then there exists a constant $C>0$, independent of $n$ and $\delta$, such that for any $y\le C/(n\delta^2)$
	we have the following lower tail bound 
	\begin{equation}\label{tail}
		\mathbf{P}^{\mathrm{Gin}}\left(\lambda_1^z\le y\delta^{3/2}\right)\lesssim y^2 (n\delta^2)^{4/3} e^{-n\delta^2(1+\OO(\delta))/2}.
	\end{equation}
\end{proposition}
Recall that $\frac{1}{\pi}\Im m^z(x+\ii 0)$, the self-consistent density of states of $H^z$, has a gap of size $4\delta^3/27$
close to $0$ justifying the $\delta^{3/2}$ scaling in~\eqref{tail} (see the paragraph below \cite[Eq. (18a)]{MR4408004}, we remark that in \cite{MR4408004} we defined $\delta$ with the opposite sign).

\subsection{Sketch of the proof of (\ref{omega2_exp})-(\ref{concentrate})}\label{sec:sketch}
Having introduced  the necessary ingredients, we briefly summarize the strategy to prove (\ref{omega2_exp})-(\ref{concentrate})
for i.i.d. matrices. These  steps will be outlined more precisely  in Section~\ref{sec:proof_main_theorem}
after the necessary cutoff functions are introduced.

The exponential factor in~\eqref{tail}, obtained for $|z|>1$ away from the boundary, ensures that with our choice $\eta_0\ll n^{-7/8}$,
$\E^{\mathrm{Gin}}\big|I_0^{\eta_0}\big|^2$ is negligible, which
implies that the main contribution to both the expectation and the variance of (\ref{linear_stat_2}) comes
from the large-$\eta$ regime, $I_{\eta_0}^{T}$, at least for the Ginibre matrices. Next we will use GFT arguments to extend these Ginibre estimates to generic i.i.d. matrices. 
We first show that $\E|I_0^{\eta_0}|$ is negligible also for i.i.d. matrices using the bound on the resolvent for Ginibre ensemble implied by Proposition \ref{prop:tail} together with a GFT argument. Then we will consider the large-$\eta$ regime
and use another GFT to show that 
$\E[I_{\eta_0}^{T}]$ and $\V[I_{\eta_0}^{T}]$ have the same bound as their  Ginibre counterparts (modulo a negligible error).
Finally, we need a bound on $\V^\mathrm{Gin}[I_{\eta_0}^{T}]$, which is not accessible directly, but can be
deduced indirectly from $\V^{\mathrm{Gin}}( I_0^{\eta_0}+I_{\eta_0}^{T})\approx \V^{\mathrm{Gin}}(\#\{\sigma_i \in \Omega\}) $
and using that 
$\E^{\mathrm{Gin}}\big|I_0^{\eta_0}\big|^2$ is negligible.

\subsection{Comments on the real symmetry class}\label{sec:real}

We stated and proved our main result only for the complex case;
the  proof for the real case would require  two changes.  
First,  the precise form of the cumulant expansions behind the GFT arguments 
slightly depends on the symmetry class: the real case yields some extra terms 
that can be treated routinely (see e.g. \cite{MR4235475, 2102.04330}
for an analogous extension). Second, we prove the lower tail bound in Proposition~\ref{prop:tail}
only for the complex case since  our detailed proof relies on the formula from \cite{MR2162782}
that has no analogue for the real case. The alternative supersymmetric method has both complex and real versions,
the latter being considerably more involved. For the sake of brevity, we  show how the complex version 
can also be used to prove Proposition~\ref{prop:tail} and we omit the more cumbersome details
of the real case, however we have no doubt that this analysis is feasible 
based upon our experience from~\cite{MR4408004, 2105.13720}.
Precise tail bounds in both symmetry classes  for the $|z|\le 1+ O(n^{-1/2})$ regime have already   
been proven in~\cite[Corollary  2.4]{MR4408004} (see also \cite{2105.13720, 2105.13719}). In the real case 
the factor $y^2$  in~\eqref{tail} is replaced with $(y^2 + y\exp{(-\frac{n}{2}(\Im z)^2})$ indicating
a weaker level repulsion near the real axis. This weaker estimate however does not affect the usage of Proposition~\ref{prop:tail}
in the main body of our proof since it is effective only for a small regime of the $z$ parameter.

The rest of the paper is organized as follows. In Section~\ref{sec:proof_main_theorem}, we prove our main theorem
using the  GFTs for the two different $\eta$ regimes as an input.  
Next, we prove the GFT used for the small-$\eta$ integral, i.e. Proposition~\ref{prop1} in Section~\ref{sec:prop}, and
then  in Section~\ref{sec:GFTGFT} we present the proof of Proposition \ref{gft} used for the large-$\eta$ integral. 
%Finally, in Appendix~\ref{app:ginb} we prove Proposition~\ref{prop:tail}. 

\section{Proof of Theorem~\ref{main}}\label{sec:proof_main_theorem}

To use Girko's formula for  proving
(\ref{omega2_exp})-(\ref{concentrate}), we need to first regularize the indicator functions. For the domains
$\Omega_k$, $k=1,2$, given in \eqref{omega_1} and \eqref{omega_2}, we  may choose two smooth 
cut-off functions $f_k^-$  and $f_k^+$ which are supported on a slightly smaller domain 
$\Omega_k^- \subset \Omega_k$ and a slightly larger domain $\Omega_k^+ \supset \Omega_k$ respectively, such that
\begin{align}\label{omega_approx_12}
	\sum_{i=1}^n f^{-}_k(\sigma_i) \leq	\#\{\sigma_i \in \Omega_k\} \leq \sum_{i=1}^n f^{+}_k(\sigma_i), \qquad k=1,2.
\end{align}
In fact, for $k=1$ we will only need the lower bound, while for $k=2$ we need only the upper
bound, so we will define only  $f_1^-$ and  $f_2^+$, the other two cut-off functions are not used in our proof.

More precisely, for the domain $\Omega_1$ given in (\ref{omega_1}), we choose the lower bound cut-off function 
\begin{align}\label{f1function}
	f^{-}_1(z):= g^{-}_1(x) h_1^{-}(y),\qquad z=x+\ii y \in \C,
\end{align} 
where $g^{-}_1(x) \in [0,1]$ and $h_1^{-}(y) \in [0,1]$ are smooth functions given by
\begin{align}\label{gfunction}
	g^{-}_1(x)=\begin{cases}
		1, &  |x-L| \leq 4l_n/5, \\
		0, & |x-L| \geq l_n, 
	\end{cases},\qquad 	h_1^-(y)=\begin{cases}
		1, &  |y|\leq 4h_n/5, \\
		0, &  |y|\geq h_n.
	\end{cases}
\end{align} 
Here, we used the shorthand notations
\begin{align}\label{l_parameter}
	L:=1+\sqrt{\frac{\gamma_n}{4n}}, \qquad l_n:=   \frac{C_n}{\sqrt{4 n \gamma_n}},\qquad h_n:=n^{-1/4+\tau/2},
\end{align}
with $1 \ll C_n \ll \sqrt{\log n}$. Additionally, $g^{-}_1$, $h^{-}_1$  are defined so that their second derivatives can be bounded by
\begin{align}\label{deri_cond}
	\|(g^{-}_1)''\|_1 \lesssim l_n^{-1}, \qquad \|(h_1^{-})''\|_1 \lesssim h_n^{-1}.
\end{align}
For the spectral domain $\Omega_2$ given in (\ref{omega2_exp}), we similarly choose $f_2^{+}(x+\ii y):=g^{+}_2(x) h_2^{+}(y)$ where $g^+_2$ is supported on the regime enlarging the $x$-domain of $\Omega_2$ by $l_n/5$ and $h^+_2$ is supported on the regime enlarging the $y$-domain of $\Omega_2$ by $h_n/5$ (\cf (\ref{gfunction})), such that
\begin{align}\label{deri_cond_2}
	\|(g^{+}_2)''\|_1 \lesssim l_n^{-1}, \qquad \|(h_2^{+})''\|_1 \lesssim h_n^{-1}.
\end{align}

We are now ready to present the proof of our main result.  

\begin{proof}[Proof of Theorem~\ref{main}]
	With the  cut-off functions $f_1^{-}$ and $f_2^{+}$ as above we have~\eqref{omega_approx_12}
	and from (\ref{deri_cond}) and (\ref{deri_cond_2}), we also have
	\begin{align}\label{deltaf}
		\|\Delta f^-_1\|_{1},~\|\Delta f^+_2\|_{1} \lesssim \frac{h_n}{l_n} \lesssim n^{\frac{1}{4}+\frac{\tau}{2}} \sqrt{\log n}.
	\end{align}

	We will show the following results in expectation 
	\begin{align}\label{main_part_1}
		\E\Big[\sum_{i=1}^n f^{+}_2(\sigma_i)\Big]=o(1);\qquad  \E\Big[\sum_{i=1}^n f^{-}_1(\sigma_i)\Big] \geq c_0,
	\end{align}
	for some constant $c_0>0$, and the concentration result
	\begin{align}\label{main_part_2}
		\E\Big|\sum_{i=1}^n f^{-}_1(\sigma_i)-\E\Big[\sum_{i=1}^n f^{-}_1(\sigma_i)\Big] \Big|=o(1) \Big( \E\Big[\sum_{i=1}^n f^{-}_1(\sigma_i)\Big] \Big).
	\end{align}
	These two key results easily imply Theorem~\ref{main}.
	More precisely, combining the second inequality in (\ref{omega_approx_12}) for $k=2$ with the first estimate in (\ref{main_part_1}), we have
	\begin{align}
		\P(\#\{\sigma_i \in \Omega_2\} \geq 1)\leq\P\Big(\sum_{i=1}^n f^{+}_2(\sigma_i) \geq 1\Big) \leq \E\Big[\sum_{i=1}^n f^{+}_2(\sigma_i)\Big]=o(1).
	\end{align}
	Moreover, using the first inequality in (\ref{omega_approx_12}) for $k=1$, the Markov inequality in combination with
	(\ref{main_part_1}) and (\ref{main_part_2}), we have
	\begin{align}
		\P&\left(\#\{\sigma_i \in \Omega_1\}=0\right) \leq 	\P\Big(\sum_{i=1}^n f^{-}_1(\sigma_i)=0\Big) \nonumber\\
		&\qquad \qquad \leq \P \left( \left|\sum_{i=1}^n f^{-}_1(\sigma_i)-\E\Big[\sum_{i=1}^n f^{-}_1(\sigma_i)\Big]\right| \geq \frac{\E\Big[\sum_{i=1}^n f^{-}_1(\sigma_i)\Big]}{2}\right)=o(1).
	\end{align}
	This shows that the key bounds (\ref{main_part_1}) and (\ref{main_part_2}) are indeed sufficient for the proof of Theorem~\ref{main}.
	
	The first step to prove these key bounds
	is to use the explicit formula for the eigenvalue correlation functions of the complex Ginibre ensemble
	to show that (\ref{main_part_1}) and (\ref{main_part_2})  hold true for the  Gaussian case. 
	\begin{lemma}\label{Ginibre_estimate} 
		For the complex Ginibre ensemble, we have 
		\begin{align}\label{main_part_1_ginibre}
			\E^{\mathrm{Gin}}\Big[\sum_{i=1}^n f^{+}_2(\sigma_i)\Big]\lesssim e^{-4C_n/5};\qquad  \E^{\mathrm{Gin}}\Big[\sum_{i=1}^n f^{-}_1(\sigma_i)\Big] \gtrsim \sinh(4C_n/5),
		\end{align}
		with any $C_n \ll \sqrt{\log n}$ and
		\begin{align}\label{main_part_2_ginibre}
			\V^{\mathrm{Gin}}\Big[\sum_{i=1}^n f^{-}_1(\sigma_i) \Big]\lesssim e^{-2C_n/5} \Big( \E^{\mathrm{Gin}}\Big[\sum_{i=1}^n f^{-}_1(\sigma_i)\Big] \Big)^2.
		\end{align}
	\end{lemma}
	
	The  proof  of Lemma \ref{Ginibre_estimate} is given in Section~\ref{sec34}.
	The next step is to extend this lemma to generic i.i.d. matrices.
	From now on, we use $f$ to denote either of the smooth cut-off functions $f^{-}_1, f^{+}_2$ for simplicity of notation.
	Lacking explicit 
	formula for eigenvalue correlations for i.i.d. matrices, we first link the linear statistics in Lemma \ref{Ginibre_estimate} to the Green function of $H^{z}$ using  Girko's formula in (\ref{linear_stat}).  Choosing $T$ sufficiently large, e.g. $T=n^{100}$, we next show that the last term in (\ref{linear_stat}) is very small with very high probability. Note that
	\begin{align}\label{logT}
		\log |\det (H^{z}-\ii T)|=&2n \log T +\sum_{j} \log \left( 1+\Big(\frac{\lambda_j^{z}}{T}\Big)^2\right) =2n \log T+ \OO\left( \frac{\Tr (H^{z})^2}{T^2}\right)\nonumber\\
		=& 2n \log T+\OO_\prec\left( \frac{n^2}{T^2}\right),
	\end{align}
	where we used that $|x_{ij}| \prec n^{-1/2}$ from (\ref{eq:hmb}). Using the $L^1$ norm of $\Delta f$ in (\ref{deltaf}), we have
	\begin{align}\label{I_3}
		\Big|\frac{1}{4 \pi} \int_{\C} \Delta f(z) \log |\det (H^{z}-\ii T)| \dd^2 z\Big|=\OO_\prec( n^{-100}).
	\end{align}
	Therefore, we have
	\begin{align}\label{I_s+I_b}
		\sum_{i=1}^n f(\sigma_i)=&-\frac{1}{4 \pi}  \int_{\C} \Delta f(z) \Big(\int_0^{\eta_0}+\int_{\eta_0}^{T}\Big)  \Im\Tr G^{z}(\ii \eta) \dd \eta \dd^2 z+\OO_\prec( n^{-100})\nonumber\\
		&=:I_{0}^{\eta_0}(f)+I_{\eta_0}^{T}(f)+\OO_\prec( n^{-100}),
	\end{align}
	where we split the $\eta$ integral into the two parts at the truncation level $\eta_0:=n^{-7/8-\tau}$, where $\tau>0$ is the fixed small number from (\ref{omega_0}).  
	Then the proof of (\ref{main_part_1}) and (\ref{main_part_2}) is reduced to studying $I_{0}^{\eta_0}(f)$ and $I^{T}_{\eta_0}(f)$ 
	respectively.  The idea is to first estimate $I_{0}^{\eta_0}(f)$ and $I^{T}_{\eta_0}(f)$ for the Ginibre ensemble and then extend these estimates to i.i.d. matrices using GFT arguments respectively.

	Next we outline the three main steps of the rest of the  proof, the precise details will be given in the following three subsections, respectively.
	\begin{enumerate}
		\item[Step 1.] For the Ginibre ensemble, we use the explicit lower tail bound for the smallest eigenvalue $\lambda_i^{z}$ 
		in Proposition~\ref{prop:tail} to show  that (see Lemma~\ref{lemma_small} below)
		\begin{equation}\label{Evar}
			\E^{\mathrm{Gin}}\Big[\big|I_{0}^{\eta_0}(f)\big|^2\Big]=o(1), \quad \qquad f=f_1^- \mbox{ or } f_2^+.
		\end{equation}
		Combining this with Lemma~\ref{Ginibre_estimate}, we have
		\begin{align}\label{some_Gini_1}
			\E^{\mathrm{Gin}}[I_{\eta_0}^{T}(f_2^{+})] =o(1), \qquad \E^{\mathrm{Gin}}[I_{\eta_0}^{T}(f^{-}_1)] \geq  c, 
		\end{align}
		for some constant $c>0$, and
		\begin{align}\label{some_Gini_2}
			\E^{\mathrm{Gin}}\big( I_{\eta_0}^{T}(f^{-}_1) -\E^{\mathrm{Gin}}[I_{\eta_0}^{T}(f^{-}_1)]\big)^2=o(1)\big(\E^{\mathrm{Gin}}[I_{\eta_0}^{T}(f^{-}_1)]\big)^2. 
		\end{align}

		\item[Step 2.] We next use a GFT argument (see Proposition \ref{prop2} below) together with the corresponding estimate of the resolvent (see Lemma \ref{lemma00} below) for the Ginibre ensemble to show 
		\begin{align}\label{step2_small}
			\E|I_0^{\eta_0}(f)|=o(1), \quad \qquad f=f_1^- \mbox{ or } f_2^+.
		\end{align}
		This directly implies that
		\begin{align}
			\E \Big[\sum_{i=1}^n f(\sigma_i)\Big]=& \E[I_{\eta_0}^{T}(f)]+o(1),\nonumber\\
			\label{mainterm}
			\E\Bigg| \sum_{i=1}^n f(\sigma_i)-\E \Big[\sum_{i=1}^n f(\sigma_i)\Big]\Bigg|=&\E\big| I_{\eta_0}^{T}(f) -\E[I_{\eta_0}^{T}(f)]\big|+o(1).
		\end{align}
		
		\item[Step 3.] For the large--$\eta$ integral $I_{\eta_0}^{T}$, we use another GFT argument (see Proposition \ref{gft} below) to extend the corresponding Ginibre estimates~\eqref{some_Gini_1}--\eqref{some_Gini_2} obtained in Step 1 to i.i.d. matrices, \ie 
		\begin{align}
			\E\Big[I_{\eta_0}^{T}(f^{+}_2)\Big]=o(1);\qquad  \E\Big[I_{\eta_0}^{T}(f^{-}_1)\Big] \geq c,\nonumber\\
			\label{genvar}
			\E\Big|I_{\eta_0}^{T}(f^{-}_1)-\E\Big[I_{\eta_0}^{T}(f^{-}_1)\Big] \Big|^2=o(1) \Big( \E\Big[I_{\eta_0}^{T}(f^{-}_1)\Big] \Big)^2.
		\end{align} 
		
	\end{enumerate}
	
	The variance bound~\eqref{genvar}  directly implies
	$$\E\big| I_{\eta_0}^{T}(f^{-}_1) -\E[I_{\eta_0}^{T}(f^{-}_1)] \big| 
	=o(1)\big|\E[I_{\eta_0}^{T}(f^{-}_1)]\big|.
	$$
	Combining this with~\eqref{mainterm} for $f=f_1^-$,
	we  proved the concentration estimate in (\ref{main_part_2}). It is also straightforward to prove the expectation estimates in (\ref{main_part_1}) using the first line of~\eqref{mainterm}  and the corresponding expectation estimates in~\eqref{genvar}. 
	Hence we completed the proof of Theorem~\ref{main}.
\end{proof}

Notice that our strategy follows a somewhat unconventional indirect route. Typical proofs based upon the Green function comparison
method assume that all necessary information is available for the Gaussian model. This does not quite hold
in our case; Step 1 above is not a purely explicit calculation.
While the statistics of the eigenvalues of the Ginibre ensemble are fully available via 
explicit formulas in both symmetry classes,  less is known about the eigenvalues of $H^z$. For a fixed $z$,
their correlation functions are known, at least in the complex case, but no explicit formula is 
available for their statistics for different $z$'s. Note that the variance of $I_{\eta_0}^{T}(f)$ in Step 1 involves
the correlation between eigenvalues of $H^z$ and $H^{z'}$ for two different $z, z'$ and
the necessary estimate requires this correlation to decay when $z-z'$ are far away.
While it is plausible  that the local spectral statistics  of $H^z$ and $H^{z'}$, especially their
lowest eigenvalues, are independent
whenever $|z-z'|\gg n^{-1/2}$, this has only been shown \cite{MR4235475,1912.04100}
in the  regime of their typical behavior; now we would need
such information in the atypical  lower tail regime.
Lacking such decorrelation bound, the variance of $I_{\eta_0}^{T}(f)$ is controlled indirectly as
$$
\V^{\mathrm{Gin}}(I_{\eta_0}^{T}(f)) \approx  \V^{\mathrm{Gin}}\big(I_{0}^{\eta_0}(f)+I_{\eta_0}^{T}(f)\big),
$$
as long as  $ \V^{\mathrm{Gin}}(I_{0}^{\eta_0}(f))\ll \V^{\mathrm{Gin}}\big(I_{0}^{\eta_0}(f)+I_{\eta_0}^{T}(f)\big)$. 
Using~\eqref{I_s+I_b}, the explicit  Ginibre eigenvalue statistics
gives the control  for $\V^{\mathrm{Gin}}\big(I_{0}^{\eta_0}(f)+I_{\eta_0}^{T}(f)\big) $. 
We cannot control $ \V^{\mathrm{Gin}}(I_{0}^{\eta_0}(f))$  optimally, but
we chose the threshold $\eta_0$ sufficiently small that 
this variance is negligible by~\eqref{Evar} and thus
no effective decorrelation estimate is necessary.

We now explain in details how Steps 1-3 are proven.

\subsection{Ginibre estimate for $I_{0}^{\eta_0}(f)$ and $I_{\eta_0}^{T}(f)$}

Using the explicit lower tail estimate of the smallest eigenvalue of $H^{z}$ in Proposition \ref{prop:tail} with $X$ being the complex Ginibre ensemble, we obtain the following improved estimate:
\begin{lemma}
	\label{lemma00}
	Fix $\delta=|z|^2-1$ with $n^{-1/2}\ll \delta \ll  1$.
	Then for any $\eta \leq C/(n\delta^{1/2})$, it holds
	\begin{align}\label{gini}
		&\E^{\mathrm{Gin}}\big[ \Im \<G^{z}(\ii \eta)\> \big]  \lesssim  \frac{\eta}{\delta}+n^{1+\xi} \eta \delta,\nonumber\\
		&\E^{\mathrm{Gin}} \Big[\big( \Im \<G^{z}(\ii \eta)\> \big)^2 \Big]  \lesssim \frac{n^{\xi}}{n^{2/3} \delta^{1/3}} e^{-n\delta^2(1+\OO(\delta))/2}+\frac{\eta^2}{\delta^2}+n^{2+\xi} \eta^2 \delta^2,
	\end{align}
	where $\xi>0$ is an arbitrary small constant.
\end{lemma}
\begin{remark}
	The exponential factor in (\ref{tail}) for the smallest eigenvalue $\lambda_1^z$
	does not manifest in the first moment estimate on the resolvent since the main contribution 
	comes from larger eigenvalues. For the second moment, however, the lowest
	eigenvalue plays the leading role.
\end{remark}
The proof of Lemma \ref{lemma00} will be given in Section~\ref{sec34}.	
Therefore from Lemma \ref{lemma00}, for any $z \in \mathrm{supp}(f^{-}_1) \cup \mathrm{supp}(f^{+}_2)$ and $\eta\leq n^{-3/4-\epsilon}$, we have
\begin{align}\label{two_moment}
	\E^{\mathrm{Gin}} \Big[ \Im \<G^{z}(\ii \eta)\> \Big] \lesssim n^{\xi}\sqrt{n}\eta, \qquad \E^{\mathrm{Gin}} \Big[\big( \Im \<G^{z}(\ii \eta)\> \big)^2 \Big] \lesssim n^{\xi}\big(n\eta^2+n^{-3/4}\big).
\end{align}
Thus we can prove that $I_{0}^{\eta_0}(f)$ is negligible in the second moment sense:
\begin{lemma}\label{lemma_small}
	With $\eta_0=n^{-7/8-\tau}$ and $f=f^{-}_1$ or $f^{+}_2$, we have
	\begin{align}\label{last_goal}
		\E^{\mathrm{Gin}}|I_{0}^{\eta_0}(f)|^2=O_{\prec}(n^{-\tau/2}).
	\end{align}
\end{lemma}
The proof of Lemma \ref{lemma_small} is given
in Section~\ref{sec34}.
Combining Lemma \ref{lemma_small} with Lemma \ref{Ginibre_estimate} and (\ref{I_s+I_b}), we have
proved~\eqref{some_Gini_1} and~\eqref{some_Gini_2}.

Next, we will use the Green function comparison to extend these Ginibre estimates to generic i.i.d. matrices satisfying Assumption \ref{ass:mainass}.

\subsection{Green function comparison for $I_{0}^{\eta_0}(f)$}
In this part, we will show that $I_{0}^{\eta_0}(f)$ is negligible in the first absolute moment for generic i.i.d. matrices. 
\begin{lemma}\label{lemma0}
	With $\eta_0=n^{-7/8-\tau}$, for any $f=f^{-}_1$ or $f^{+}_2$, we have
	\begin{align}\label{need}
		\E|I_{0}^{\eta_0}(f)|=\E\Big|\frac{n}{4\pi}\int_{\C} \Delta f(z) \int_{0}^{\eta_0}\Im\< G^{z}(\ii \eta)\> \dd \eta \dd^2 z \Big| = \OO(n^{-\tau}).
	\end{align}
\end{lemma}
We will postpone the proof of  Lemma~\ref{lemma0} to Section~\ref{sec34}.  The proof of Lemma \ref{lemma0} crucially 
relies on the following improved estimate of the resolvent at the intermediate level $\eta=\eta_0$ 
using the monotonicity of $\Im \Tr G(\ii \eta)$.
\begin{proposition}\label{prop1}
	Let $X$ be an i.i.d. complex matrix satisfying Assumption~\ref{ass:mainass}. 
	For any small $\epsilon>0$, then for any $n^{-1+\epsilon}\leq\eta \leq n^{-3/4-\epsilon}$ 
	and $-Cn^{-1/2} \lesssim |z|-1 \leq n^{-1/2+\tau}$, we have
	\begin{align}\label{im_estimate}
		\E\big[ \Im \<G^{z}(\ii \eta)\> \big] \prec n^{1/2}\eta+\frac{1}{n^{5/2}\eta^2}+\frac{1}{n^5\eta^5}+n^{-1}.
	\end{align}
\end{proposition}

Proposition \ref{prop1} is a direct result of the Ginibre estimate in Lemma \ref{lemma00} and the following 
Green function comparison which will be proved in Section \ref{sec:prop}.
\begin{proposition}\label{prop2}  For any small $\epsilon>0$, for any $n^{-1+\epsilon}\leq\eta 
	\leq n^{-3/4-\epsilon}$ and $-Cn^{-1/2} \lesssim |z|-1 \leq n^{-1/2+\tau}$, we have
	\begin{align}\label{difference}
		\big|	\E[ \<G^{z}(\ii \eta)]-	\E^{\mathrm{Gin}}[ \<G^{z}(\ii \eta)]  \big] 
		\prec\frac{1}{n^{5/2}\eta^2}+\frac{1}{n^5\eta^5}+n^{-1}.
	\end{align}
\end{proposition}

Therefore, combining Lemma \ref{lemma0} with (\ref{I_s+I_b}), for any $f=f^{-}_1$ or $f^{+}_2$, we 
proved~\eqref{mainterm} in the effective form
\begin{align}
	\label{eq:finb}
	\E \Big[\sum_{i=1}^n f(\sigma_i)\Big]=& \E[I_{\eta_0}^{T}(f)]+O(n^{-\tau}),\nonumber\\
	\E\Bigg| \sum_{i=1}^n f(\sigma_i)-\E \Big[\sum_{i=1}^n f(\sigma_i)\Big]\Bigg|=&\E\big| I_{\eta_0}^{T}(f) -\E[I_{\eta_0}^{T}(f)]\big|+O(n^{-\tau}).
\end{align}
Hence to prove the target estimates in (\ref{main_part_1}) and (\ref{main_part_2}), it suffices to study $I_{\eta_0}^{T}(f)$.

\subsection{Green function comparison for $I_{\eta_0}^{T}(f)$}\label{sec33}
In this part, we will extend the estimates of $I_{\eta_0}^{T}(f)$ in (\ref{some_Gini_1}) and (\ref{some_Gini_2}) from
Gaussian matrices to generic i.i.d. matrices. It then suffices to establish the following Green function
comparison for $I_{\eta_0}^{T}(f)$, whose proof will be presented in Section~\ref{sec:GFTGFT}.
\begin{proposition}\label{gft}
	For any $f=f^{-}_1$ or $f^{+}_2$, there exist some constants $c_1,c_2>0$ such that
	\begin{align}\label{gft_1}
		\Big| \E[I_{\eta_0}^{T}(f)]-\E^{\mathrm{Gin}}[I_{\eta_0}^{T}(f)] \Big|=O(n^{-c_1}),
	\end{align}
	and 
	\begin{align}\label{gft_2}
		\Big|	\E \big(I_{\eta_0}^{T}(f)-\E[I_{\eta_0}^{T}(f)]\big)^2-	
		\E^{\mathrm{Gin}} \big(I_{\eta_0}^{T}(f)-\E^{\mathrm{Gin}}[I_{\eta_0}^{T}(f)]\big)^2 \Big|=O(n^{-c_2}).
	\end{align}
\end{proposition}

Combining the Ginibre estimates~\eqref{some_Gini_1}--\eqref{some_Gini_2}
with the GFT estimates~\eqref{gft_1}--\eqref{gft_2}, we obtain the second line of~\eqref{genvar}
$$\E\big( I_{\eta_0}^{T}(f^{-}_1) -\E[I_{\eta_0}^{T}(f^{-}_1)]\big)^2=o(1)\big(\E[I_{\eta_0}^{T}(f^{-}_1)]
+O(n^{-c_1})\big)^2+O(n^{-c_2})=o(1) \big(\E[I_{\eta_0}^{T}(f^{-}_1)]\big)^2.
$$
The first line of~\eqref{genvar} is obtained similarly.

\subsection{Proof of some lemmas}\label{sec34}
We now give
the proofs of Lemmas~\ref{Ginibre_estimate},~\ref{lemma00},~\ref{lemma_small} and~\ref{lemma0}. 

\begin{proof}[Proof of Lemma \ref{Ginibre_estimate}]
	We first recall that the expectation and variance of the linear statistics of a general test function \(f\) can be expressed as 
	\begin{equation}\label{E Var}
		\begin{split}
			\E^{\mathrm{Gin}} \sum_i f(\sigma_i) &= \int f(z)\wt K_n(z,z)\dif^2 z\\
			\V^{\mathrm{Gin}}\sum_i f(\sigma_i)&=\int f(z)^2\wt K_n(z,z)\dif^2 z - \iint f(z) f(w) \abs[\big]{\wt K_n(z,w)}^2 \dif^2z \dif^2 w
		\end{split}
	\end{equation}
	in terms of the kernel \(\wt K_n(z,w)\) from~\cite{2206.04443}. Using the kernel asymptotics from~\cite[Lemma 6]{2206.04443} we obtain 
	\begin{equation}\label{asymp expected num evals}
		\begin{split}
			\int_{L-t/\sqrt{4\gamma n}}^{L+t/\sqrt{4\gamma n}}\int_{-s/(\gamma n)^{1/4}}^{s/(\gamma n)^{1/4}}\wt K_n(z,z)\dif \Im z\dif \Re z=2\erf(s)\sinh(t)\Bigl(1+\OO\big({\frac{\log\log n+t^2 + s^4}{\log n}}\big)\Bigr)
		\end{split}
	\end{equation}
	for \(\abs{t}+s^2\le \sqrt{\log n}/10\), 
	while for any \(t>0\) we have the bounds
	\begin{equation}\label{bound expected num evals}
		\begin{split}
			\int_{L-t/\sqrt{4\gamma n}}^{L+t/\sqrt{4\gamma n}}\biggl(\int_{-\infty}^{\sqrt{2t}/(\gamma n)^{1/4}}+\int_{\sqrt{2t}/(\gamma n)^{1/4}}^{\infty}\biggr)\wt K_n(z,z)\dif \Im z\dif \Re z&\lesssim e^{-t/4}\\
			\int_{L+t/\sqrt{4\gamma n}}^{\infty}\int_{\R}\wt K_n(z,z)\dif \Im z\dif \Re z&\lesssim e^{-t/4}.
		\end{split}
	\end{equation}
	From~\eqref{E Var}--\eqref{bound expected num evals} we immediately conclude
	\begin{equation}
		\E^{\mathrm{Gin}}\sum_i f_1^-(\sigma_i) \gtrsim \sinh\Bigl(\frac{4C_n}{5}\Bigr)
	\end{equation}
	and
	\begin{equation}
		\begin{split}
			\V^{\mathrm{Gin}}\sum_i f_1^-(\sigma_i)&\le \int f(z)^2\wt K_n(z,z)\dif^2 z \lesssim \sinh\Bigl(\frac{6C_n}{5}\Bigr)
		\end{split}
	\end{equation}
	recalling that \( C_n\ll\sqrt{\log n}\), proving~\eqref{main_part_2_ginibre}. Finally, the first bound in~\eqref{main_part_1_ginibre} follows directly from~\eqref{bound expected num evals}.
\end{proof}

\begin{proof}[Proof of Lemma \ref{lemma00}]
	We reformulate~\eqref{tail} as follows: for any $\eta\leq \wt \eta:= C/(n \delta^{1/2})$, 
	\begin{align}\label{better_exp}
		\P^{\mathrm{Gin}} \left( \lambda_1^z \leq \eta \right)\lesssim \frac{n^{4/3}\eta^2}{\delta^{1/3}} e^{-n\delta^2(1+\OO(\delta))/2}.
	\end{align}
	
	Using spectral decomposition of $H^{z}$ and the spectrum symmetry, we write
	\begin{align}\label{step_m1}
		\E^{\mathrm{Gin}}[\Im \<G^z&(\ii \eta)\>] = \frac{1}{n}  \sum_{k=0}^{K_0} \E^{\mathrm{Gin}} 
		\Big[ \sum_{3^k \eta \leq \lambda_i   < 3^{k+1} \eta} \frac{\eta}{(\lambda^z_{i})^2+\eta^2}  \Big] 
		+\frac{1}{n} \E^{\mathrm{Gin}} \Big[  \sum_{\lambda^z_i   \geq  \wt \eta }  \frac{\eta}{(\lambda^z_{i})^2+\eta^2}  \Big]\nonumber\\	
		\lesssim &   \frac{n^{\xi}}{n} \sum_{k=0}^{K_0} \frac{1}{3^{2k} \eta+\eta} 
		\P^{\mathrm{Gin}}\big( \lambda^z_1 \leq 3^{k+1} \eta \big)+\frac{1}{n} \E^{\mathrm{Gin}} 
		\Big[  \sum_{\lambda^z_i   \geq  \wt \eta }  \frac{\eta}{(\lambda^z_{i})^2+\eta^2}  \Big],
	\end{align}
	with $K_0=\lceil \log_3(\wt \eta/\eta)\rceil=O(\log n)$, where we used the rigidity of eigenvalues and $\xi>0$ is 
	an arbitrary small number. The second term in (\ref{step_m1}) can be bounded effectively 
	using the local law in (\ref{average}) and (\ref{m_1}), \ie
	\begin{align}\label{second_term}
		\frac{1}{n} \E^{\mathrm{Gin}} \Big[  \sum_{\lambda^z_i   \geq  \wt \eta }  \frac{\eta}{(\lambda^z_{i})^2+\eta^2}  \Big]
		=\frac{\eta}{\wt \eta}\E^{\mathrm{Gin}} \Big[\frac{\wt \eta}{n} \sum_{\lambda_j   \geq  \wt \eta} \frac{1}{\lambda^2_{j}
			+\wt \eta^2}\Big] \lesssim \frac{\eta}{\wt \eta} \E^{\mathrm{Gin}}[\Im \<G^z(\ii \wt \eta)\>] 
		\lesssim \frac{\eta}{\delta}+n^{1+\xi} \eta \delta,
	\end{align}
	where $\xi>0$ is an arbitary small number. 	
	Using (\ref{better_exp}), the first term in (\ref{step_m1}) can be bounded by
	$$
	\frac{n^{\xi}}{n} \sum_{k=0}^{K_0} \frac{1}{3^{2k} \eta+\eta}  \P^{\mathrm{Gin}}\big( \lambda^z_1 \leq 3^{k+1} \eta \big) 
	\lesssim \frac{n^{\xi}}{n} \sum_{k=0}^{K_0} \frac{1}{(3^{2k}+1)\eta} \frac{3^{2k+2} n^{4/3} \eta^2}{\delta^{1/3}}
	\lesssim n^{\xi} \log n\frac{n^{1/3}\eta}{\delta^{1/3}}.
	$$
	Note that we did not use the exponential factor in (\ref{better_exp}) yet since this bound is 
	already much smaller compared to the second term estimate in (\ref{second_term}) for $\delta \gg n^{-1/2}$.  
	Therefore, 
	we finished the proof of the first moment estimate in (\ref{gini}).
	
	Similarly, for the second moments, we have
	\begin{align}
		\E^{\mathrm{Gin}}[(\Im \<G^z&(\ii \eta)\>)^2] =\frac{1}{n^2}  \E^{\mathrm{Gin}} 
		\Big[ \Big( \sum_{k=0}^{K_0}\sum_{3^k \eta \leq \lambda_i^z  < 3^{k+1} \eta}
		\frac{\eta}{(\lambda^z_{i})^2+\eta^2} +\sum_{\lambda_i^z   \geq  \wt \eta }  
		\frac{\eta}{(\lambda^z_{i})^2+\eta^2} \Big)^2 \Big]\nonumber\\	
		\lesssim & \log n \frac{n^{\xi}}{n^2} \sum_{k=0}^{K_0} \frac{1}{(3^{2k}+1)^2 \eta^2}  
		\P^{\mathrm{Gin}}\big( \lambda^z_1 \leq 3^{k+1} \eta \big)+ \frac{\eta^2}{n^2} 
		\E^{\mathrm{Gin}} \Big[\Big(\sum_{i} \frac{1}{(\lambda^z_{i})^2+\wt \eta^2}\Big)^2\Big]\nonumber\\
		\lesssim & \log n \frac{n^{\xi}}{n^{2/3} \delta^{1/3}} e^{-n\delta^2(1+\OO(\delta))/2}+
		\frac{\eta^2}{\delta^2}+n^{2+\xi} \eta^2 \delta^2,
	\end{align}
	with $\xi>0$ being any arbitary small number, where we used the tail bound in (\ref{better_exp}) 
	and the local law in (\ref{average}). This finishes the proof of Lemma \ref{lemma00}.
\end{proof}
Using Lemma \ref{lemma00}, we will prove Lemma~\ref{lemma_small} and \ref{lemma0}. 
Since the proof of these two lemmas is similar, we present only the detailed proof of Lemma~\ref{lemma0}.

\begin{proof}[Proof of Lemmas~\ref{lemma_small} and \ref{lemma0}]
	We start with showing that the regime $\eta\in [0,n^{-l}]$, 
	for some very large $l>0$ is negligible, exactly  as it was done in \cite{MR3770875}.
%\cor More precisely, from \cite[Theorem 3.2]{MR2684367}, for any $l>2$, there exists $C_l>0$ such that
%\footnote{
	%Note that the bound
	%\[
	%\mathbb{P}\left(\mathrm{Spec}(H^z)\cap \left[-n^{-l},n^{-l}\right] \cor \neq \nc \emptyset\right)\le C_l n^{-l/2},
	%\]
%	for any $l\ge 2$, uniformly in $|z|\le 2$, used here directly follows by \cite[Theorem 3.2]{MR2684367}. 
%	Alternatively, a similar bound follows by \cite[Proposition 5.7]{MR3770875} (which is inspired by \cite[Lemma 4.12]{MR2908617}) 
%	under the additional assumption that, there exists $\alpha,\beta>0$ such that the probability density 
%	of $X_{ij}$, denoted by $g$, satisfies 
%	\begin{align}\label{assumption_b}
%		g \in L^{1+\alpha}(\mathbb{F}), \quad \|g\|_{1+\alpha} \leq n^{\beta}, \qquad \mathbb{F}=\R~\mbox{or}~\C.
%\end{align}}
%$$\mathbb{P}\left(\mathrm{Spec}(H^z)\cap \left[-n^{-l},n^{-l}\right] \neq \emptyset\right)\le C_l n^{-l/2},$$
%uniformly in $|z|\le 2$. 
By a direct computation,
$$\int_0^{n^{-l}} \Im \Tr G(\ii \eta) \dd \eta 
%=\frac{1}{2} \sum_{i=1}^n \log \Big( 1+\frac{n^{-2l}}{\lambda_i^2}\Big)
=\frac{1}{2} \Big( \sum_{|\lambda_i| \lesssim n^{-l}} +\sum_{|\lambda_i| \gtrsim n^{-l}} \Big) \log \Big( 1+\frac{n^{-2l}}{\lambda_i^2}\Big).$$
The second sum can easily be estimated  using~Lemma~\ref{lemma00} or Proposition~\ref{prop1}, \ie
\begin{align}
	\E\Big[\sum_{|\lambda_i| \gtrsim n^{-l}} \log \Big( 1+\frac{n^{-2l}}{\lambda_i^2}\Big) \Big]
	%\lesssim&  \E\Big[\sum_{n^{-l} \lesssim |\lambda_i| \leq \eta_1} \log \Big( 1+\frac{n^{-2l}}{\lambda_i^2}\Big)\Big] +n^{-2l} \E\Big[ \sum_{ |\lambda_i| \geq \eta_1} \frac{1}{\lambda_i^2} \Big]\nonumber\\
	\lesssim &  (\log n) \E|\{i: |\lambda_i| \leq \eta_1 \}| +\frac{n^{1-2l}}{\eta^2_1} \lesssim n^{-1/4-100\tau}+n^{-2l+3}, \nonumber
\end{align}
where we chose $\eta_1:=n^{-7/8-100\tau}$. To estimate the first sum, we recall \cite[Proposition 5.7]{MR3770875}, \ie under the density condition~(\ref{assumption_b}), there exists $C_\alpha>0$ such that
\begin{align}\label{density_bound}
\mathbb{P}\left( \min_{i=-n}^n |\lambda^z_i| \leq \frac{u}{n} \right)\le C_\alpha u^{\frac{2\alpha}{1+\alpha}} n^{\beta+1},\qquad z \in \C, u>0,
\end{align}
with $\alpha, \beta$ given in (\ref{assumption_b}). 
Then following \cite{MR3770875}[Eq. (5.34)-(5.35)], we have
$$\E\Big[\sum_{|\lambda_i| \lesssim n^{-l}} \log \Big( 1+\frac{n^{-2l}}{\lambda_i^2}\Big)\Big]  \lesssim n \E \big[ |\log \lambda^z_1| \one_{\lambda_1 \lesssim n^{-l}} \big] 
%\lesssim C_\alpha n^{\beta+1+\frac{2\alpha}{1+\alpha}} \int_{l \log n}^{\infty} \ee^{-\frac{2\alpha}{1+\alpha} t} \dd t 
\lesssim n^{-10},$$
with $l$ large enough depending on $\alpha,\beta$. Combining this bound with the $L^1$-norm of $\Delta f$ in (\ref{deltaf}), we 
conclude that the very tiny regime $\eta \in [0,n^{-l}]$ is negligible. 

Hence, it is enough to estimate the contribution to $I_0^{\eta_0}(f)$ \nc of the remaining $\eta$-integral over $[n^{-l},\eta_0]$. 
	Using that $\eta \rightarrow \eta \Im \Tr G^{z}(\ii \eta)$ is increasing in $\eta\ge 0$, we have
	\begin{align}\label{middle_step}
		\Big|\int_{\C} \Delta f(z) \E \big[ \int_{n^{-l}}^{\eta_0}\Im\Tr G^{z}(\ii \eta) \dd &\eta\big] \dd^2 z \Big| 
		\leq \int_{\C} |\Delta f(z)|  \int_{n^{-l}}^{\eta_0} \frac{\eta_0}{\eta}  \E\big[\Im \Tr G^{z}(\ii \eta_0)\big] \dd \eta  \dd^2 z\nonumber\\
		\leq & C (\log n) n \eta_0   \int_{\C} |\Delta f(z)|   \E\big[\Im \< G^{z}(\ii \eta_0)\>\big] \dd^2 z. 
	\end{align}
	Using the estimate in (\ref{im_estimate}) with $\eta_0=n^{-7/8-\tau}$ \ie
	\begin{align}\label{complex_middle}
		\E[\Im \<\gz(\ii \eta_0)\>] \prec  n^{1/2}  \eta_0 +n^{-5/8+5\tau}\lesssim n^{-3/8-\tau},
	\end{align} 
	together with the $L^1$ norm of $\Delta f$ in (\ref{deltaf}), by (\ref{middle_step}) we obtain that
	\begin{align}\label{complex_middle2}
		\Big|\int_{\C} \Delta f(z) \E \big[ \int_{n^{-l}}^{\eta_0}\Im\Tr G^{z}(\ii \eta) \dd \eta\big] \dd^2 z \Big| =\OO( n^{-\tau}).
	\end{align}
	We hence finish the proof of Lemma \ref{lemma0}. Lemma~\ref{lemma_small} can be proven similarly using the second estimate in (\ref{two_moment}).
\end{proof}

\section{Green function comparison for resolvents: Proof of Proposition~\ref{prop2}}
\label{sec:prop}

Before starting with the proof of Proposition~\ref{prop2}  we introduce some notations which we will use throughout this section.

\begin{notation}
	\label{not:index} 
	We use lower case
	letters to denote the indices taking values in $\llbracket 1, n \rrbracket$ and upper case letters to denote
	the indices taking values in $\llbracket n+1, 2n \rrbracket$. We also use calligraphic letters 
	$\mathfrak{u}, \mathfrak{v}$ to denote the indices ranging fully from $1$ to $2n$.
	
	For any index $\mathfrak{v} \in \llbracket 1, 2n \rrbracket$, the conjugate of $\mathfrak{v}$, denoted 
	by $\mathrm{conj}(\mathfrak{v})$, is given by $\mathrm{conj}(\mathfrak{v}) \in \llbracket 1, 2n \rrbracket$ 
	and it is such that $|\mathrm{conj}(\mathfrak{v})-\mathfrak{v}|=n$. In particular, for an index $ a \in \llbracket 1, n \rrbracket$, 
	we define its index conjugate $\mathrm{conj}(a)=\bar a:=a+n$,
	and  for an index $ B \in \llbracket n+1, 2n \rrbracket$ we define its index conjugate $\mathrm{conj}(B)=\ud B:=B-n$. 
	With a slight abuse of terminology,  we say two indices {\it coincide} if either they are equal or one is equal to 
	the conjugate of the other one. For instance, we say $a\in \llbracket 1, n \rrbracket$ coincides with the 
	index $B\in \llbracket n+1, 2n \rrbracket$ if $a=\ud{B}$ (or equivalently $B=\bar a$). 
	We also say that a collection of indices are {\it distinct} if there is no index coincidence among them (in the sense explained above). 
	
	Moreover, we often use generic   letters  $x$ and $y$ to denote the (first) row and the (second) column index of a Green function entry. 
	In this context the lower case letters $x,y$ do not indicate that they take values in $\llbracket 1, n \rrbracket$. 
	We say that a Green function entry $G_{xy}$ is {\it diagonal} if $x=y$ or $x=\mathrm{conj}(y)$; 
	otherwise we say that $G_{xy}$ is off-diagonal. 
	
	\bigskip

	We now explain this terminology with an example:
	\begin{align}\label{ex_average}
		\frac{1}{n^2} \sum_{a=1}^{n} \sum_{B=n+1}^{2n} G_{aB} G_{B \bar a}=	
		\frac{1}{n^2} \sum_{a=1}^{n} \sum_{B\neq \bar{a}} G_{aB} G_{B \bar a} +\frac{1}{n^2} \sum_{a=1}^{n} G_{a \bar{a}} G_{\bar a \bar a},
	\end{align}
	where we split the summation into two parts: 1) the two summation indices $a$ and $B$ are distinct; 
	2) there is an index coincidence $B=\bar{a}$ (or equivalently $a = \ud{B}$) in the summation. 
	In the first term the two Green function entries are off-diagonal, while in the second term the two Green function entries are diagonal. 
	
\end{notation}

We prove Proposition \ref{prop2} via a continuous interpolating flow. Given the initial ensemble $H^{z}$ 
in (\ref{initial}), we consider the matrix flow
\begin{align}\label{flow}
	\dd H^z_t=-\frac{1}{2} (H^z_t+Z)\dd t+\frac{1}{\sqrt{n}} \dd \mathcal{B}_t, \quad Z=\begin{pmatrix}
		0  &  zI  \\
		\overline{z}I   & 0
	\end{pmatrix}, \quad 
	\mathcal{B}_t=\begin{pmatrix}
		0  &  B_t  \\
		B^*_t   & 0
	\end{pmatrix}
\end{align}
where $B_t$ is an $n \times n$ matrix with independent standard complex valued Brownian motion entries.
The matrix flow $H_t^z$ interpolates between the initial matrix $H^{z}$ in (\ref{initial})
and an independent matrix as in (\ref{initial}) with $X$ being replaced with an independent complex Ginibre ensemble. 

The Green function of the time dependent matrix $H^{z}_t$ is denoted by $G^{z}_t$.  
Since the flow in (\ref{flow}) is stochastically H\"{o}lder continuous in time, the local law for the 
Green function in Theorem \ref{local_thm} also holds true for the time dependent Green function
$G^{z}_t$ simultaneously for all $t \ge 0$ by a grid argument, together with the H\"older 
regularity of $G_t^z$ for $0\le t\le n^{100}$ and a simple perturbation argument for $t\ge n^{100}$.
More precisely for $n^{-1+\epsilon} \leq \eta\leq n^{-3/4-\epsilon}$ and $-Cn^{-1/2} \lesssim |z|-1 \leq n^{-1/2+\tau}$, it holds uniformly
\begin{align}\label{localg}
	\sup_{t\ge 0}\max_{1\leq \mathfrak{v},\mathfrak{u} \leq 2n} \Big\{\big| \big(G^{z}_t(\ii \eta) \big)_{\mathfrak{uv}}-m^{z} \delta_{\mathfrak{v}=\mathfrak{u}} -\mathfrak{m}^{z} \delta_{|\mathfrak{v}-\mathfrak{u}|=n}\big| \Big\} \prec\Psi:=\frac{1}{n\eta},
\end{align}
where $m^{z}$, $\mathfrak{m}^{z}$ are given in (\ref{m_1m_2}) and (\ref{m_1}), and they are such that 
\begin{align}\label{localm}
	m^{z} \equiv m^{z}(\ii \eta)= \OO(\Psi), \qquad \mathfrak{m}^{z} \equiv \mathfrak{m}^{z}(\ii \eta) =\OO(1).
\end{align}
Note that in our range of parameters, the small deterministic term $m^{z} \delta_{\mathfrak{v}=\mathfrak{u}}$ may be included
in the error in~\eqref{localg}.
We remark that the diagonal Green function entry $G_{\mathfrak{u}\mathfrak{u}}$ as well as the off-diagonal
Green function entries are bounded by $\Psi$ with very high probability, while the other kind of diagonal 
Green function entry $G_{\mathfrak{u}, \mathrm{conj}(\mathfrak{u})}$ can only be bounded by $1$ with very high probability.

In the following, we often drop the dependence on $t$ and $\eta$ and set $G^{z} \equiv G^{z}_{t}(\ii \eta)$
for notational simiplicity. Without specfic mentioning, all the estimates in this section hold true 
uniformly for any $-Cn^{-1/2} \lesssim |z|-1 \leq n^{-1/2+\tau}$, $n^{-1+\epsilon} \leq \eta\leq n^{-3/4-\epsilon}$ and $t\ge 0$.

\begin{proof}[Proof of Proposition \ref{prop2}]
	
	In order to prove Proposition \ref{prop2}, it suffices to show the following estimate on the time derivative of $\<G_t^z(\ii\eta)\>$:
	\begin{lemma}\label{lemma_goal}
		For any $ n^{-1+\epsilon} \leq \eta \leq n^{-3/4-\epsilon}$, $-Cn^{-1/2} \lesssim |z|-1 \leq n^{-1/2+\tau}$ and $t\ge 0$, it holds
		\begin{align}\label{goal_1}
			\left|\frac{\dd}{\dd t} \E[\< G^{z}_t(\ii \eta)\>] \right|=\OO_\prec (n^{-1/2}\Psi^2+\Psi^{5}+n^{-1}).
		\end{align}
	\end{lemma}
	
	Integrating (\ref{goal_1}) over $t \in [0, t_0]$ with $t_0=100\log n$ we obtain 
	\begin{align}\label{0tot_0}
		\big|\E[\<G_0^{z}(\ii \eta)\>]-\E[\<G^{z}_{t_0}(\ii \eta)\>] \big| =\OO_{\prec}\big(\log n(n^{-1/2}\Psi^2+\Psi^{5}+n^{-1})\big). 
	\end{align}
	Note that $H_t^{z}$ in (\ref{flow}) is given as in (\ref{initial}) with $X$ being replaced with the time dependent matrix
	$$X_t \stackrel{{\rm d}}{=} e^{-\frac{t}{2}} X+\sqrt{1-e^{-t}} \mathrm{Gin} (\mathbb{C}), 
	\qquad t\ge 0,$$
	where $X_{\infty} \stackrel{{\rm d}}{=}  \mathrm{Gin} (\mathbb{C})$ is the complex 
	Ginibre ensemble which is independent of $X$. Then we have
	\begin{align}\label{t_0toinf}
		\|G^{z}_{t_0}(\ii \eta)-G^{z}_{\infty}(\ii \eta)\| \leq \|G^{z}_{t_0}\| \|G^{z}_{\infty}\| \|X_{t_0}-X_{\infty}\| 
		\prec n^{-48},
	\end{align}
	where we used that $\|G^z(\ii \eta)\| \leq \eta^{-1}\le n$ and that $|X_{ij}| \prec n^{-1/2}$ from  the assumption in \eqref{eq:hmb}. 
	Combining (\ref{0tot_0}) with (\ref{t_0toinf}) we conclude the proof of Proposition \ref{prop2}.
	
\end{proof}

We now present the proof of Lemma \ref{lemma_goal}.

\subsection{Proof of Lemma~\ref{lemma_goal}}
Recall the matrix flow in (\ref{flow}) with complex-valued $X$, we set 
\begin{align}\label{Wmatrix}
	W \equiv W_t=H^z_t+Z=\begin{pmatrix}
		0  &  X_t  \\
		X_t^*   & 0
	\end{pmatrix}.
\end{align}
Then $\dd H^z_t=-\frac{1}{2}W_t\dd t+\frac{1}{\sqrt{n}} \dd \mathcal{B}_t$. \
Applying Ito's formula and setting $\partial /\partial w_{aB}=\partial/\partial h_{aB}$, we have     
\begin{align}\label{sde_G}
	\dd \<G_t^z\>=&\sum_{a,B} \frac{\partial \<G_t^z\>}{\partial h_{aB}} \dd h_{aB}
	+\sum_{a,B} \frac{\partial \<G_t^z\>}{\partial \overline{h_{aB}}} \dd \overline{h_{aB}}\nonumber\\
	&+\frac{1}{2}\sum_{a,B} \frac{\partial^2 \<G_t^z\>}{\partial h_{aB} \partial \overline{h_{aB}}} \dd h_{aB} \dd \overline{h_{aB}}
	+\frac{1}{2}\sum_{a,B} \frac{\partial^2 \<G_t^z\>}{\partial  \overline{h_{aB}}\partial h_{aB} }
	\dd \overline{h_{aB}}\dd h_{aB}  \nonumber\\
	=&\left(-\frac{1}{2}\sum_{a,B} w_{aB}\frac{\partial \<G_t^z\>}{\partial w_{aB}}-
	\frac{1}{2}\sum_{a,B}\overline{w_{aB}} \frac{\partial \<G_t^z\>}{\partial \overline{w_{aB}}} +\frac{1}{n}\sum_{a,B}
	\frac{\partial^2 \<G_t^z\>}{\partial w_{aB} \partial \overline{w_{aB}}} \right) \dd t\nonumber\\
	&+\frac{1}{\sqrt{n}}\sum_{a,B} \frac{\partial \<G_t^z\>}{\partial w_{aB}}\dd (\mathcal{B}_t)_{a B}
	+\frac{1}{\sqrt{n}}\sum_{a,B} \frac{\partial \<G_t^z\>}{\partial \overline{w_{aB}}} \dd \overline{(\mathcal{B}_t)_{a B}}.
\end{align}
Note that the expectation of the martingale  term on the last line of (\ref{sde_G}) vanishes. 
It then suffices to study the expectation of the remaining terms. We note that 
\begin{align}\label{udG}
	\<\gz\> =\frac{1}{n} \sum_{v=1}^{n} \gz_{vv}=\frac{1}{n} \sum_{V=n+1}^{2n} \gz_{VV}
	=\frac{1}{2n} \sum_{\mathfrak{v}=1}^{2n} \gz_{\mathfrak{vv}},
\end{align}
which follows from the spectral symmetry induced by the $2 \times 2$ block matrix in (\ref{initial}).
Performing the cumulant expansion formula on $\{w_{aB}\}$ and $\{\overline{w_{aB}}\}$ on the right side of (\ref{sde_G}) 
(see e.g. \cite[Lemma 7.1]{MR3678478}), we observe the precise cancellations of the second order terms, 
and the summation below starts from the third order terms \ie $p+q+1=3$,
\begin{align}\label{ex_deri}
	\frac{\dd}{\dd t} \E[\<G^{z}_t\>]=	&\E\Big[\frac{\dd \<G_t^z\>}{\dd t}\Big]=-\frac{1}{2n} \sum_{v,a=1}^{n} \sum_{B=n+1}^{2n}
	\left( \sum_{p+q+1= 3}^{K_0}\frac{c_{aB}^{(p+1,q)}}{p!q!n^{\frac{p+q+1}{2}}} 
	\E \left[  \frac{\partial^{p+q+1} \gz_{vv}}{\partial w_{aB}^{p+1} \partial \overline{w_{aB}}^{q} } \right]\right)\nonumber\\
	&-\frac{1}{2n} \sum_{v,a=1}^{n} \sum_{B=n+1}^{2n} \left(\sum_{p+q+1= 3}^{K_0}\frac{c_{aB}^{(q,p+1)}}{p!q!n^{\frac{p+q+1}{2}}} 
	\E \left[  \frac{\partial^{p+q+1} \gz_{vv}}{\partial \overline{w_{aB}}^{p+1} \partial w_{aB}^{q}}\right]\right)
	+\OO_{\prec}(n^{-\frac{K_0}{2}+2}),
\end{align}
where $c^{(p,q)}_{aB}$ are the $(p,q)$-cumulants of the normalized complex-valued entries $\sqrt{n} w_{aB}$, with $c_{aB}^{(p,q)}=c_{Ba}^{(q,p)}$ from the complex symmetry, and we omit their dependence on $t$ for simplicity.   The last error stems from
truncating the cumulant expansions at a sufficiently large $K_0$-th order, say $K_0=100$, using the local law in
(\ref{localg}) and the finite moment condition in \eqref{eq:hmb}; see also \cite{MR4221653} for a similar truncation argument.

For simplicity we assume i.i.d. entries of $X$ in the our model, thus
the cumulants are independent of the indices,  $c^{(p,q)}_{aB}=c^{(p,q)}$. We next consider only the first line of (\ref{ex_deri}), \ie
\begin{align}\label{goal}
	\sum_{p+q+1=3}^{K_0}	L^{z}_{p+1,q}:=\sum_{p+q+1=3}^{K_0}\frac{c^{(p+1,q)}}{2 p!q!}\left(
	%-\frac{c^{(p+1,q)}}{2p!q!n^{\frac{p+q+3}{2}}}
	\frac{1}{n^{\frac{p+q+3}{2}}}  \sum_{v,a,B} 
	\E \left[   \frac{\partial^{p+q+1} \gz_{vv}}{\partial w_{aB}^{p+1} \partial \overline{w_{aB}}^{q} } \right]\right),
	%	 =\OO_\prec( n^{-1/2}\Psi^2+\Psi^{5}+n^{-1}),
\end{align}
and the second line of (\ref{ex_deri}) is exactly the same as (\ref{goal}) by interchanging $a$ with $B$.

Using the following differentiation rules for any $1\leq \mathfrak{u},\mathfrak{v} \leq 2n$
\begin{align}\label{rule_1}
	\frac{\partial G^{z}_{\mathfrak{uv}}}{\partial w_{aB}}
	%=\frac{\partial G^{z}_{\mathfrak{uv}}}{\partial h_{aB}}
	=-G^{z}_{\mathfrak{u}a}G^{z}_{B\mathfrak{v}}, \qquad \frac{\partial G^{z}_{\mathfrak{uv}}}{\partial \overline{w_{aB}}}
	%=\frac{\partial G^{z}_{\mathfrak{uv}}}{\partial \overline{h_{aB}}}
	=-G^{z}_{\mathfrak{u}B}G^{z}_{a\mathfrak{v}}, 
\end{align}
each term $L^z_{p+1,q}$ in (\ref{goal}) can be written as a linear combination of products of $p+q+2$ Green function entries of the form
\begin{align}\label{some_form}
	\frac{1}{n^{\frac{p+q+3}{2}}} \sum_{v=1}^{n} \sum_{a=1}^{n} \sum_{B=n+1}^{2n}  \E \Big[ \prod_{i=1}^{p+q+2} \gz_{x_iy_i}\Big].
\end{align}
Here $x_i, y_i$ denote generic row and column indices of $G^z$, respectively, to which we assign actual
summation indices $v,a,B$, depending on the precise structure of the corresponding term dictated by
(\ref{goal})--(\ref{rule_1}).
The assignment will be denoted by the symbol $\equiv $, \eg $x_i \equiv a$, $y_i \equiv B$ means that 
the generic factor $G_{x_iy_i}^z$ is replaced with the actual 
$G_{aB}^z$ in (\ref{some_form}). Note that both lower  and upper case summation indices
can be assigned to the  generic $x, y$ indices.  The  assignments  that appear from  (\ref{goal})--(\ref{rule_1}) have the
following properties: $x_1 \equiv v$, $y_{p+q+2}\equiv v$, and all the other indices $x_i,y_i \equiv $ either $a$ or $B$ such that
\begin{align}\label{pq_number}
	\#\{x_i \equiv a\}=\#\{y_i \equiv B\}=q, \qquad \#\{ x_i \equiv B\}=\#\{ y_i \equiv a\}=p+1.
\end{align}
From the local law in (\ref{localg}) and (\ref{localm}), we have $|\gz_{aa}|, |\gz_{BB}| , |\gz_{aB}|, |\gz_{Ba}| \prec \Psi$ 
unless $a=\ud{B}$. If we restrict to the summation when all the indices are distinct in (\ref{some_form}) 
(\ie $a \neq \ud{B}$, $v\neq a$, and $v\neq \ud B$), 
then the product of $p+q+2$ Green function entries in (\ref{some_form}) can be bounded by $\Psi^{p+q+2}$. 
For the remaining summation when there is some index coincidence (e.g. $a= \ud{B}$), we gain a factor $n^{-1}$ 
since the number of free summation indices is reduced by one. Therefore, we obtain the following so-called {\it naive estimate}  
\begin{align}\label{naive_g}
	|L^{z}_{p+1,q}| \prec n^{-\frac{p+q-3}{2}} \big(\Psi^{p+q+2}+n^{-1}\big).
\end{align}
Thus for any $p+q+1\geq 4$, we have
\begin{align}\label{five}
	|L_{p+1,q}^{z}| \prec \Psi^5+n^{-1}.
\end{align}
However the naive estimate in (\ref{naive_g}) for $p+q+1=3$ is not sufficiently fine to prove Lemma \ref{lemma_goal}. 

Next we focus on proving an improved estimate for these third order terms, \ie proving
\begin{align}\label{third}
	\sum_{p+q+1=3} |L^{z}_{p+1,q}| =O_{\prec} \big( n^{-1/2}\Psi^2+n^{-1}\big).
\end{align}
By direct computations, the third order terms $L^z_{p+1,q}$ with $p+q+1=3$ in (\ref{goal}) are linear combinations of the following terms
\begin{align}\label{some_term}
	&\frac{\sqrt{n}}{n^{3}}\sum_{v,a,B}\E[G^{z}_{v a} G^z_{BB} G^z_{aa}  G^z_{Bv}],
	\quad \frac{\sqrt{n}}{n^{3}}\sum_{v,a,B}\E[G^z_{{v} a} G^z_{Ba} G^z_{Ba} G^z_{Bv}], \nonumber\\
	&\frac{\sqrt{n}}{n^{3}}\sum_{v,a,B}\E[G^z_{v a} G^z_{BB} G^z_{aB}  G^z_{av}], 
	\quad \frac{\sqrt{n}}{n^{3}}\sum_{v,a,B}\E[G^z_{v B} G^z_{aB} G^z_{aa}  G^z_{B v}],
\end{align} 
as well as the other terms with the index $a$ and $B$ interchanged. 
%The naive size of the above
% terms is only $\sqrt{n}\Psi^4+n^{-1/2}$ from (\ref{naive_g}), which is however far from the truth. 
As explained below (\ref{pq_number}), we split the threefold summations in (\ref{some_term}) into the following three cases
(recall  the concept of coinciding and distinct indices from Notation~\ref{not:index}):

\begin{enumerate}
	
	\item[1)] all  three summation indices coincide in the summation (\ie $v=a=\ud{B}$
	% and thus only one free summation index is left
	): 
	the resulting sum in (\ref{some_term}) can be bounded by $O_\prec(n^{-3/2})$ using that 
	$|G_{\mathfrak{u} \mathfrak{v}}| \prec 1$ from (\ref{localg}), which is small enough to prove (\ref{third}); 
	
	\item[2)] exactly  two of the summation indices coincide  (\eg $a=\ud{B} \neq v$ or $a = v \neq \ud B$): the resulting sum
	in (\ref{some_term}) can be bounded by $O_\prec(n^{-1/2}\Psi^2)$ using the local law in 
	%that $|G_{aa}|, |G_{BB}|\prec \Psi$ from 
	(\ref{localg}) and (\ref{localm}), which is also sufficient to prove (\ref{third});
	
	\item[3)] all three summation indices are distinct (\ie $a \neq \ud{B} \neq v$): using the local law in (\ref{localg}) and (\ref{localm}) naively,
	the resulting term in (\ref{some_term}) can be bounded by $O_\prec(\sqrt{n}\Psi^4)$, which is however far from the truth.
	We observe from (\ref{some_term}) that these third order terms have indices $a$ and $B$ that 
	both appear three times as a first and as a second index of a $G$-factor. A somewhat more
	complicated version of this feature (see the concept of unmatched indices in Definition \ref{def:unmatch_shift} later)
	allows us to improve the bound on them.
\end{enumerate}

Next, we will discuss in details for the third order terms from (\ref{some_term}) in Case 3), \ie with the summation restriction
of all indices different, $a \neq \ud{B} \neq v$. We first introduce the shifted version of the Green function
\begin{align}\label{shifted}
	\wh \gz:=\gz-M^z=O_{\prec}(\Psi), \qquad M^{z}=\begin{pmatrix}
		m^{z} &  \mathfrak{m}^{z}  \\
		\overline{\mathfrak{m}^{z}}   & m^{z}
	\end{pmatrix},
\end{align}
with $m^{z}$ and $\mathfrak{m}^{z}$ given in (\ref{m_1m_2}). The shifted version $\wh \gz$ differs 
from $\gz$ only for the diagonal entries, \ie $G_{xy}=\wh{G_{xy}}$ unless $x=y$ or $x=\mathrm{conj}(y)$. 
Then the first term among the third order terms in (\ref{some_term}) with $a \neq \ud{B} \neq v$ can be
written as (omitting the factor $\sqrt{n}$)
\begin{align}\label{third_example}
	\frac{1}{n^{3}} \sum_{a \neq \ud{B} \neq v} \E[\gz_{v a}  \gz_{BB}&\gz_{aa} \gz_{B v}]=\frac{1}{n^{3}} 
	\sum_{a \neq \ud{B} \neq v} \E[\wh{\gz_{v a}} \wh{\gz_{BB}} \wh{\gz_{aa}}  \wh{\gz_{B v}}] 	+\frac{m^{z}}{n^{3}}  
	\sum_{a \neq \ud{B} \neq v} \E[\wh{\gz_{v a}} \wh{\gz_{BB}} \wh{\gz_{B v}}]\nonumber\\
	&+ \frac{ m^{z}}{n^{3}}  \sum_{a \neq \ud{B} \neq v} \E[\wh{\gz_{v a}} \wh{\gz_{aa}} \wh{\gz_{B v}}]
	+\frac{(m^{z})^2}{n^{3}}  \sum_{a \neq \ud{B} \neq v} \E[\wh{ \gz_{v a}}  \wh{\gz_{B v}} ].
\end{align}
Note that the terms on the right side above are averaged products of shifted Green function entries 
of the form defined in (\ref{form}) below. Moreover, these terms are unmatched since the index $a$ (or $B$) 
appears odd number times in the product of Green function entries which clearly does not satisfy the match 
condition in (\ref{match_condition}); see Defintion \ref{def:unmatch_shift} below. In fact, any term in (\ref{some_term}) 
with the restriction $a \neq \ud{B} \neq v$ can be written as a linear combination of unmatched terms of the form 
in (\ref{form}) as in (\ref{third_example}) with a factor $\sqrt{n}$. Using Proposition~\ref{lemma3} below in 
combination with additional contributions from Case 1) and 2) with the index coincidences, we have obtained
the improved estimate for the third order terms in (\ref{third}).

Combining (\ref{ex_deri}), (\ref{five}) and (\ref{third}), we finish the proof of Lemma \ref{lemma_goal}. \qed

\vline

Before giving the formal definition of unmatched indices (and unmatched terms) to study the third order terms in \eg (\ref{third_example}) from Case 3) systematically, we first set some notational conventions. 

For any fixed $l_1,l_2 \in \N$, we use $\mathcal{I}_{l_1,l_2}$ to denote a set of $l_1$ lower case
letters and $l_2$ upper case letters, \eg the set may contain lower case letters $a,v$ and upper case
letter $B$ as in (\ref{third_example}). In general, we may write 
$\mathcal{I}_{l_1,l_2}:=\{v_j\}_{j=1}^{l_1} \cup \{V_j\}_{j=1}^{l_2}$. 
Each element in $\mathcal{I}_{l_1,l_2}$ will represent a summation index and the font type of each letter indicates 
the range of the summation for that index; as before, the 
lower case letters $v_j$ run from 1 to $n$, and the upper case letters $V_j$ run from $n+1$ to $2n$.  
We denote the sum over these $l:=l_1+l_2$ summation indices (indicated by $\mathcal{I}_{l_1,l_2}$) by 
$$\sum_{\mathcal{I}_{l_1,l_2}}=\sum_{v_1,\cdots v_{l_1}, V_1,\cdots,V_{l_2}} :=\sum_{v_1=1}^{n} \cdots \sum_{v_{l_1}=1}^{n} \sum_{V_1=n+1}^{2n} \cdots \sum_{V_{l_2}=n+1}^{2n}.$$ 
We also introduce a partial summation restricted to distinct indices, 
\begin{align}\label{distinct_sum}
	\sum^*_{\mathcal{I}_{l_1,l_2}}:=\sum_{v_1,\cdots v_{l_1}, V_1,\cdots,V_{l_2}} \Big(\prod_{j\neq j'}^{l_1} \delta_{ v_j \neq v_{j'}} \Big)\Big(\prod^{l_2}_{j\neq j'} \delta_{ V_j \neq V_{j'}} \Big) \Big( \prod_{j=1}^{l_1} \prod_{j'=1}^{l_2} \delta_{ v_j \neq \ud{V_{j'}}}  \Big),
\end{align}
\ie
each summation index in $\mathcal{I}_{l_1,l_2}$ is different from all the other indices and their conjugates.

\begin{definition}
	\label{def:unmatch_form}
	Given $l_1, l_2\in \mathbb{N}$ and a collection of lower and upper case summation indices 
	$\mathcal{I}_{l_1,l_2}=\{v_j\}_{j=1}^{l_1} \cup \{V_j\}_{j=1}^{l_2}$,
	we consider a product of $d$ generic shifted Green function entries $ \wh{\gz_{x_1y_1}} \wh{\gz_{x_2y_2}} \cdots \wh{\gz_{x_dy_d}}$
	and assign a summation index $v_j$, $V_j$ or their conjugates $\overline{v_j}, \ud{V_j}$ to each generic index $x_i, y_i$
	(\eg $x_1\equiv v_2, y_1\equiv \ud{V_5}, x_2\equiv \overline{v_3}, y_2 \equiv V_5$, etc.). 
	A term of the form
	\begin{align}\label{form}
		\frac{1}{n^{l}} \sum^*_{\mathcal{I}_{l_1,l_2}} \wh{\gz_{x_1y_1}} \wh{\gz_{x_2y_2}} \cdots \wh{\gz_{x_dy_d}}= \frac{1}{n^{l}}  \sum^*_{\mathcal{I}_{l_1,l_2}}  \prod_{i=1}^{d}  \wh{\gz_{x_iy_i}}, \qquad l=l_1+l_2,
	\end{align}
	of degree $d$ with a concretely specified assignment is denoted by $P_d$.
	The collection of the terms of the form in (\ref{form}) with degree $d$ is denoted by $\mathcal{P}_{d}$. 
\end{definition}

Given a term $P_d \in \mathcal{P}_d$ in (\ref{form}), the local law in (\ref{localg}) yields a naive bound using power counting, \ie for any $P_d \in \mathcal{P}_d$,
\begin{align}\label{p_d_local}
	|P_d| \prec \Psi^d, \qquad n^{-1/4+\epsilon} \leq \Psi=(n\eta)^{-1}\leq n^{-\epsilon}.
\end{align}

We now give the formal definition of the (un)matched terms of the form in (\ref{form}). 
%Note that this definition applies for terms 
%of the form in (\ref{shift_form}); for a slightly different form a similar definition will apply (see Definition \ref{def:unmatch} below).

\begin{definition}[(Un)matched terms in $\mathcal{P}_d$]\label{def:unmatch_shift} 	
	Given a term $P_d \in \mathcal{P}_d$ in (\ref{form}), we say that a lower case
	index $v_j \in \mathcal{I}_{l_1,l_2}$ is {\it matched} if the number of assignments of $v_j$ and
	its conjugate $\overline{v_j}$ to a row index in the product agrees with their number of assignments to a column index, \ie 
	\begin{align}\label{match_condition}
		\#\{i : x_i\equiv {v_j}\}+\#\{i: x_i\equiv \overline{v_j}\}=\#\{i: y_i\equiv{v_j}\}+\#\{i: y_i\equiv\overline{v_j} \}.
	\end{align}
	Otherwise, we say that $v_j$ is an {\it unmatched} index. For instance, looking at the two terms,
	$$\frac{1}{n^2} \sum_{a,b}^* \wh{G_{ab}} \wh{G_{ab}} \wh{G_{aa}}  \wh{G_{BB}},\qquad \quad \frac{1}{n^2} \sum_{a,b}^* \wh{G_{ab}} \wh{G_{ba}} \wh{G_{aa}}  \wh{G_{bb}},$$
	both  $a$ and $b$ are unmatched indices in the first term, while they are matched indices in the second term.

	Similarly, we say that an upper case
	index $V_j \in \mathcal{I}_{l_1,l_2}$ is matched if 
	\begin{align}\label{match_condition_2}
		\#\{i: x_i\equiv {V_j}\}+\#\{i: x_i\equiv \ud{V_j}\}=\#\{i: y_i\equiv{V_j}\}+\#\{i: y_i\equiv\ud{V_j} \}.
	\end{align}
	Otherwise, $V_j$ is an unmatched index.

	If all the summation indices in $\mathcal{I}_{l_1,l_2}$ are matched, then $P_d$ is a {\it matched term}. 
	Otherwise, if there exists at least one 
	%(in fact, equivalently two since the number of assignments of the indices is even) 
	unmatched index, $P_d$ is an {\it unmatched term}. 
	If a term $P_d$ is unmatched, we indicate this fact by denoting it 
	by $P_d^o$. The collection of the unmatched terms of the form in (\ref{form})
	with degree $d$ is denoted by $\mathcal{P}_d^o \subset \mathcal{P}_d$. 
\end{definition}

From Definition \ref{def:unmatch_shift}, the terms on the right side of (\ref{third_example}) with $v\neq a \neq \ud{B}$ are indeed unmatched terms of the form in (\ref{form}), where both the index $a$ and $B$ are unmatched while the index $v$ is matched. Moreover, we give additional examples of unmatched terms below
\begin{align}
	\frac{1}{n^{3}}\sum^{*}_{v,a,B}\E[\wh{\gz_{v \bar a}} \wh{\gz_{\ud B v}}], \quad \frac{1}{n^{2}}\sum^{*}_{a,B}\E[\wh{\gz_{aa}} \wh{\gz_{BB}} \wh{\gz_{a \ud{B}}}  ], \quad \frac{1}{n^{2}}\sum^{*}_{a,B}\E[\wh{\gz_{B \bar a}} \wh{\gz_{B \bar a}} \wh{\gz_{a B}}  ].
\end{align}

\begin{proposition}\label{lemma3}
	Given an unmatched term $P^{o}_d$ of the form in (\ref{form}) with fixed $d\geq 1$, we have
	$$\E[P^o_d] =\OO_\prec(n^{-3/2}).$$
\end{proposition}

\begin{remark}
	The above estimate is much smaller than the naive size in (\ref{p_d_local})
	either when $d$ is small, say $1\leq d\leq 5$, or when $\eta$ is close to $n^{-1+\epsilon}$. 
	For a general unmatched term $P_d^{o}$,
	the estimate $\OO_\prec(n^{-3/2})$ is sharp due to some matched terms of order $n^{-3/2}$ stemming
	from third order terms in the cumulant expansions with an index coincidence; see (\ref{same}) below. 
\end{remark}

\begin{remark}\label{remark_unmatch} The statement of Proposition \ref{lemma3} holds true even when 
	the parameters $z$ of the shifted Green function entries in the product in (\ref{form}) have different values. 
	We also remark that the proof of Proposition \ref{lemma3} is not sensitive to the fact that $m^z$ 
	given in (\ref{m_1}) is small, in fact
	the argument works as long as $m^{z}=\OO(1)$.  
\end{remark}

The rest of  Section~\ref{sec:prop}
is devoted to proving Proposition \ref{lemma3}. The proof relies on iterative cumulant expansions for the 
unmatched indices in products of resolvents. Before we dive into the formal proof of Proposition \ref{lemma3}, 
we start with expanding a concrete example of unmatched term to explain the one-step improvement
mechanism (essentially gaining an additional small factor $\Psi$) in Section~\ref{sec:ex}. 
The reader experienced with cumulant expansions may skip Section~\ref{sec:ex}. 
In Section~\ref{sec:proof_unmatch}, we state in Lemma~\ref{lemma:expand_mechanism} the 
full version of the improvement mechanism for a general unmatched term, and subsequently 
use Lemma~\ref{lemma:expand_mechanism} iteratively to prove Proposition \ref{lemma3}. 
Finally we present the complete proof of Lemma \ref{lemma:expand_mechanism} for a general unmatched term in Section~\ref{sec:full}.

\subsection{Expansion mechanism: an example}
\label{sec:ex}
In this subsection we consider a concrete example 
of an unmatched term in (\ref{form}) with degree three and $x_1\equiv a, y_1\equiv B,x_2\equiv B, y_2\equiv a, x_3\equiv \bar a, y_3\equiv B$, \ie
\begin{align}\label{p_3_o0}
	P_3^{o}=\frac{1}{n^2} \sum_{a,B}^{*} \wh{\gz_{aB}} \wh{\gz_{B a}} \wh{\gz_{ \bar a B}},
\end{align}
whose naive estimate is $O_\prec(\Psi^3)$ from (\ref{p_d_local}).
We will show how to improve this naive estimate using cumulant expansions essentially by an additional small factor $\Psi$.

For the term $P_3^o$ in (\ref{p_3_o0}) with an unmatched index $a$ which appears twice as a row and once as a column in the product of resolvents, we aim to expand using the unmatched $x_1\equiv a$ to show that 
\begin{align}\label{p_3_o}
	\E[P_3^{o}]
	=O_\prec\big(\Psi^{4}+n^{-1}\Psi+n^{-1/2}\Psi^3+n^{-3/2}\big).
\end{align}
We will see  that the terms that contribute the first three error terms with $\Psi$-factors in (\ref{p_3_o}) still have unmatched
indices. So we can continue expanding these unmatched terms to get an arbitrary number of  $\Psi$-improvements
and ending up with the final  estimate $O_\prec(n^{-3/2})$ given in Proposition~\ref{lemma3}.
The corresponding iteration scheme will be presented directly in full generality for any unmatched term in (\ref{form})
in   the next subsection.

Recall the following identity from~\cite[Eq.~(5.2)]{1912.04100}
\begin{align}\label{identity}
	\wh{\gz}=-M^{z} \ud{ WG^{z}}+\<\wh \gz\> M^z G^z,
\end{align}
where $M^{z}=M^{z}(\ii \eta)$ is the deterministic matrix given in (\ref{Mmatrix}) and $\<\gz\>$ is given in (\ref{udG}). 
The underline notation $\ud{ WG^{z}}$ is defined as follows. For a function $f(W)$ of the random matrix 
$W$ given in (\ref{Wmatrix}), we define
\begin{align}\label{ud_W}
	\ud{ Wf(W)}:= Wf(W)-\wt \E \wt W (\partial_{\wt W} f)(W),
\end{align}
where  $\wt W$ is an independent of $W$ defined as in (\ref{Wmatrix}) with $X_t$ being replaced 
by a complex Ginibre ensemble. Here $\partial_{\widetilde{W}}$ denoted the directional derivative
in the direction $\widetilde{W}$, the expectation in \eqref{ud_W} is with respect to this matrix.

Applying the identity in (\ref{identity}) on the first Green function entry $\wh {\gz_{aB}}$ in (\ref{p_3_o0})
and performing cumulant expansion formula on the resulting $\ud{W \gz}$ given in (\ref{ud_W}), we have 
\begin{align}\label{example_expand}
	\E[P_3^{o}]=&-\frac{m^{z}}{n^{3}}\sum_{a,B}^{*} \sum_{J} 
	\E \Big[\frac{\partial \wh{\gz_{B a}} \wh{\gz_{ \bar a B}}}{\partial w_{Ja}} G^{z}_{J B} \Big]
	+\frac{m^{z}}{n^{2}}\sum_{a,B}^{*} \E \Big[\gz_{a B} \wh{\gz_{B a}} \wh{\gz_{ \bar a B}} \<\wh \gz \> \Big]\nonumber\\
	&-\frac{\mathfrak{m}^{z}}{n^{3}}\sum_{a,B}^{*} \sum_{j} 
	\E \Big[\frac{\partial \wh{\gz_{B a}} \wh{\gz_{ \bar a B}} }{\partial w_{j\bar a}} G^{z}_{j B} \Big]
	+\frac{\mathfrak{m}^{z}}{n^{2}}\sum_{a,B}^{*} \E \Big[\gz_{\bar a B} \wh{\gz_{B a}} \wh{\gz_{ \bar a B}} \<\wh \gz \> \Big]\nonumber\\
	&-\frac{m^z}{n^2} \sum_{p+q+1\geq 3} \frac{c^{(p+1,q)} }{p!q!n^{\frac{p+q+1}{2}}} 
	\Big( \sum_{a,B}^{*}  \sum_{J} 
	\E \Big[\frac{\partial^{p+q} \wh{\gz_{B a}} \wh{\gz_{ \bar a B}}  G^{z}_{JB}}{\partial w^{p}_{aJ} \partial w^q_{Ja}}  \Big] \Big)\nonumber\\
	&-\frac{\mathfrak{m}^{z}}{n^2} \sum_{p+q+1\geq 3} \frac{ c^{(q,p+1)} }{p!q!n^{\frac{p+q+1}{2}}} 
	\Big( \sum_{a,B}^{*}  \sum_{j} 
	\E \Big[\frac{\partial^{p+q} \wh{\gz_{B a}} \wh{\gz_{ \bar a B}}  G^{z}_{j B}}{\partial w^{p}_{\bar aj} \partial w^q_{j \bar a}}  \Big] \Big),
\end{align}
with $c^{(p,q)}$ being the $(p,q)$-th cumulants of the normalized i.i.d. entries $\sqrt{n}w_{aB}$. We first look at the third order terms with $p+q+1=3$ in (\ref{example_expand}). 
By direct computations using the differentiation rule in (\ref{rule_1}), since $J$ or $j$ is a fresh index 
appearing three times, the number of resulting (shifted) off-diagonal entries remains at least $d$ 
with unmatched $J$ or $j$. From the local law in (\ref{localg}), these third order terms can be 
bounded by $O_\prec(n^{-1/2}\Psi^{3}+n^{-3/2})$, where the last error $n^{-3/2}$ stems from the 
existence of index coincidences, \eg $J=B$ or $j=\ud{B}$. Similarly, the fourth order terms with $p+q+1=4$ 
can be bounded by $O_\prec(n^{-1}\Psi^3+n^{-2})$ and note that the index $J$ or $j$ could 
be matched (for $p=1,q=2$), however the index $a$ remains unmatched. 
We can truncate the expansion at the fourth order with an error $O_{\prec}(n^{-3/2})$ using (\ref{localg}). 
Thus the higher order terms (\ie the last two lines of (\ref{example_expand})) can be bounded by
\begin{align}\label{higher_cumu}
	-\frac{m^z}{n^2} \sum_{p+q+1\geq 3}\left(\cdots\right)-\frac{\mathfrak{m}^{z}}{n^2} \sum_{p+q+1\geq 
		3}\left(\cdots\right)=O_\prec(n^{-1/2}\Psi^3+n^{-3/2}).
\end{align}

We next focus on the second order terms, \ie the first two lines of (\ref{example_expand}). We start with the first term on the right side of (\ref{example_expand}). After direct computations, we split the summation over the fresh index $J\in\llbracket n+1,2n\rrbracket$ into the following three cases, \ie
%we split the summation into three parts:
\begin{align}\label{middle_step_ex}
	-\frac{m^{z}}{n^{3}}\sum_{a,B}^{*} \sum_{J}
	\E \Big[&\frac{\partial \wh{\gz_{B a}} \wh{\gz_{\bar a B}} }{\partial w_{Ja}} G^{z}_{J B} \Big]
	=\frac{m^{z}}{n^{3}}\sum_{a,B}^{*} \sum_{J} \E \Big[\big( \gz_{BJ } \gz_{aa} \wh{\gz_{\bar a B}}
	+\wh{\gz_{B a}} \gz_{\bar a J} \gz_{a B} \big) G^{z}_{J B} \Big]\nonumber\\
	=:&   \left(\frac{m^{z}}{n^{3}}\sum_{a,B,J}^{*}+\frac{m^{z}}{n^{3}}\sum_{a,B}^*\sum_{J}\delta_{J \bar a}
	+\frac{m^{z}}{n^{3}}\sum_{a,B}^*\sum_{J} \delta_{JB}\right)\E[(\cdots)].
\end{align}
We first consider the last two cases with index coincidence $J=\bar a$ or $J=B$. We will create diagonal entries with $J=\bar a$ or $J=B$ and as the result the number of off-diagonal entries will be reduced to at least one. Using (\ref{localg}), the last two cases in (\ref{middle_step_ex}) can be bounded by
\begin{align}\label{example_decrease}
	\frac{m^{z}}{n^{3}}\sum_{a,B}^*\sum_{J}\delta_{J \bar a}\E[(\cdots)]
	+\frac{m^{z}}{n^{3}}\sum_{a,B}^*\sum_{J} \delta_{JB}\E[(\cdots)]= O_\prec(n^{-1}\Psi). 
\end{align}
We remark that we did not use the smallness of $m^{z}$ given in (\ref{localm}), since this smallness is not 
essential for the $\Psi$-improvement. For the first case in (\ref{middle_step_ex}) 
with $a \neq \ud{B} \neq \ud{J}$, we transform the resulting terms into the form in (\ref{form}), 
\ie write the Green function entries with their shifted versions using (\ref{shifted}). In particular, 
the diagonal entry $\gz_{aa}$, from acting $\partial w_{Ja}$ on $\wh{\gz_{Ba}}$, will be replaced 
with $m^{z}+\wh{\gz_{aa}}$. Then 
\begin{align}\label{p_3_middle}
	\frac{m^{z}}{n^{3}}\sum_{a,B,J}^{*}\E[(\cdots)]=&\frac{(m^{z})^2}{n^{3}} \sum_{a,B,J}^{*} \E \Big[ \wh{G^{z}_{J B}} \wh{\gz_{BJ }}
	\wh{\gz_{\bar a B}} \Big]\nonumber\\
	&+\frac{m^{z}}{n^{3}}\sum_{a,B,J}^{*}\E \Big[\wh{\gz_{BJ }} \wh{\gz_{aa}} \wh{\gz_{\bar a B}} \wh{G^{z}_{J B}} \Big]+\frac{m^{z}}{n^{3}}\sum_{a,B,J}^{*}\E \Big[\wh{\gz_{B a}} \wh{\gz_{\bar a J}} \wh{\gz_{a B}} \wh{G^{z}_{J B}} \Big].
\end{align}
We note that the terms in the second line of (\ref{p_3_middle}) have degree being increased to four to accommodate 
a pair of the fresh index $J$, hence can be bounded by $O_\prec(\Psi^4)$ from (\ref{p_d_local}).  Therefore, 
combining (\ref{middle_step_ex})-(\ref{p_3_middle}) the first term on the right side of (\ref{example_expand}) 
can be estimated as
\begin{align}\label{third_term_ex_00}
	-\frac{m^{z}}{n^{3}}\sum_{a,B}^{*} \sum_{J}
	\E \Big[\frac{\partial \wh{\gz_{B a}} \wh{\gz_{\bar a B}} }{\partial w_{Ja}} G^{z}_{J B} \Big]=&\frac{(m^{z})^2}{n^{3}} \sum_{a,B,J}^{*} \E \Big[ \wh{G^{z}_{J B}} \wh{\gz_{BJ }}
	\wh{\gz_{\bar a B}} \Big]+O_\prec(\Psi^4+n^{-1}\Psi),
\end{align}
where the first error $\Psi^4$ is from the second order terms with higher degrees and $n^{-1}\Psi$ is from the second order terms with the index coincidences in (\ref{example_decrease}).

%Compared to the original term $\E[P_3^{o}]$ in (\ref{p_3_o0}), the leading term of degree 
%three in (\ref{p_3_middle}) is obtained by replacing one $a$ from the row of $\wh{\gz_{aB}}$ and 
%one $a$ from the column of $\wh{\gz_{Ba}}$ with a fresh index $J \in \llbracket n+1, 2n \rrbracket$ up to a factor $(m^z)^2$. 

The third term on the right side of (\ref{example_expand}) can be estimated similarly, \ie
\begin{align}\label{third_term_ex}
	-\frac{\mathfrak{m}^{z}}{n^{3}}\sum_{a,B}^{*} \sum_{j} \E \Big[\frac{\partial \wh{\gz_{B a}} 
		\wh{\gz_{ \bar a B}} }{\partial w_{j\bar a}} G^{z}_{j B} \Big]=&
	\frac{\mathfrak{m}^{z}}{n^3} 
	\sum_{a,B}^{*} \sum_{j} \E\Big[ \big( {\gz_{B j}} \gz_{\bar a a}\wh{\gz_{\bar a B}} 
	+\wh{\gz_{B a}}  \gz_{\bar a j} \gz_{\bar a B}\big) \gz_{jB} \Big]\nonumber\\
	=&\frac{|\mathfrak{m}^{z}|^2}{n^3} \sum_{a,B,j}^{*} \E\big[ \wh{\gz_{jB}} \wh{\gz_{B j}} \wh{\gz_{ \bar a B}}  \big]
	+O_\prec(\Psi^4+n^{-1}\Psi).
\end{align}
%where the leading term of degree three is obtained from the original term $\E[P_3^o]$ in (\ref{p_3_o0})
% by replacing one $a$ from row of $\wh{\gz_{aB}}$ and one $a$ from the column of $\wh{\gz_{Ba}}$ 
% with the fresh index $j\in \llbracket 1, n \rrbracket$ up to a factor $|\mathfrak{m}^{z}|^2$. 

It remains to estimate the second and fourth term on the right side of (\ref{example_expand}). Since we restrict to $a \neq \ud{B}$ (equivalent to $\bar a\neq B$) in the summation $\sum^*_{a,B}$, we write $\gz_{aB}=\wh{\gz_{aB}}$ and $\gz_{\bar a B}=\wh{\gz_{\bar a B}}$. We also write out $\<\wh\gz\>$ as in (\ref{udG}) and clearly these two terms are of the form in (\ref{form}) with degree increased to four and the index $a$ remains unmatched. In particular, these two terms gain additional $\Psi$-factor from $\<\wh\gz\>$ and thus can be bounded by $O_\prec(\Psi^4)$.

Combining (\ref{example_expand}), (\ref{higher_cumu}), (\ref{third_term_ex_00}) and (\ref{third_term_ex}), we conclude
\begin{align}\label{first_step_ex}
	\E[P_3^{o}] =&\frac{(m^{z})^2}{n^{3}} \sum_{a,B,J}^{*} \E \Big[ \wh{G^{z}_{J B}}  \wh{\gz_{BJ }} \wh{\gz_{\bar a B}}\Big]+\frac{|\mathfrak{m}^{z}|^2}{n^3} \sum_{a,B,j}^{*} \E\big[ \wh{\gz_{jB}}  \wh{\gz_{B j}} \wh{\gz_{ \bar a B}} \big]
	\nonumber\\
	& +O_\prec(\Psi^{4}+n^{-1}\Psi+n^{-1/2}\Psi^3+n^{-3/2}),
\end{align}
where the first two errors are from the second order terms with higher degrees (\eg in the second line of (\ref{p_3_middle})) 
and with the index coincidences (\eg with $J=\bar a$ or $J=B$ in (\ref{example_decrease})), respectively, and the last 
two errors are from the higher order terms in (\ref{higher_cumu}). Most importantly, for these leading terms of degree
three appearing on the first line of (\ref{first_step_ex}), we have replaced one pair of the index $a$ of the 
original term $P_3^o$ in (\ref{p_3_o0}) with a fresh index $J$ or $j$. We now introduce a notation for such
index replacement, \ie if
$$P_3^o=\frac{1}{n^2} \sum_{a,B}^{*} \wh{\gz_{aB}} \wh{\gz_{B a}} \wh{\gz_{ \bar a B}},$$
which is a term of the form in (\ref{form}) with $x_1\equiv a, y_1\equiv B,x_2\equiv B, y_2\equiv a, x_3\equiv \bar a, y_3\equiv B$, 
then we define
\begin{align}\label{replacement}
	P^o_3(x_1,y_2 \rightarrow J):=\frac{1}{n^3} \sum_{a,B,J}^{*} \wh{\gz_{JB}} \wh{\gz_{B J}} \wh{\gz_{ \bar a B}};\qquad 
	P^o_3(x_1,y_2 \rightarrow j):=\frac{1}{n^3} \sum_{a,B,j}^{*} \wh{\gz_{jB}} \wh{\gz_{B j}} \wh{\gz_{ \bar a B}},
\end{align}
where $j$ and $J$ are 'symbolic' lower and upper case letters indicating the range of the new summation index.
% in $\llbracket1,n \rrbracket$ and $\llbracket n+1,2n\rrbracket$, respectively.
Using these notations, the expansion in (\ref{first_step_ex}) can be written for short as 
\begin{align}\label{replacement_form}
	\E[P_3^{o}]=&(m^{z})^2\E[P^o_3(x_1,y_2 \rightarrow J)]+ |\mathfrak{m}^{z}|^2 \E[P^o_3(x_1,y_2 \rightarrow j)]\nonumber\\
	&+O_\prec(\Psi^{4}+n^{-1}\Psi+n^{-1/2}\Psi^3+n^{-3/2}).
\end{align}
Notice that in the two explicit  third order terms the
number of assignments of the unmatched index $a$ after replacement has been reduced
by two  to one, in fact it appears as its conjugation $\bar a$ with $x_3 \equiv \bar a$; see (\ref{replacement}). 
The good news is that the index $a$ (in fact $\bar a$) remains unmatched, thus we 
can further expand these leading terms using $x_3 \equiv \bar a$ to gain the $\Psi$-improvement.

We will look at only the second leading term 
on the right side of (\ref{replacement_form}), and the first one can be estimated similarly
(actually more easily if we take $m^{z} =\OO(\Psi)$ into consideration).
Omitting the factor $|\mathfrak{m}^{z}|^2 \sim 1$ and expanding the Green function entry $\wh{\gz_{\bar a B}}$, we obtain as in (\ref{example_expand}),
%the analogue of (\ref{example_expand}), \ie
\begin{align}\label{four_step_ex}
	\E[P^o_3(x_1,y_2 \rightarrow j)]=&\frac{1}{n^3} \sum_{a,B,j}^{*} \E\big[ \wh{\gz_{jB}} \wh{\gz_{B j}} \wh{\gz_{\bar a B}}\big]\nonumber\\
	=&-\frac{ m^{z}}{n^{4}}\sum_{a,B,j}^{*} \sum_{j'} 
	\E \Big[\frac{\partial \wh{\gz_{B j}}\wh{\gz_{jB}} }{\partial w_{j' \bar a}} G^{z}_{j' B} \Big]
	+\frac{ m^{z}}{n^{3}}\sum_{a,B,j}^{*} \E \Big[{\gz_{ \bar a B}} \wh{\gz_{B j}} \wh{\gz_{j B}}   \<\wh \gz \> \Big]\nonumber\\
	&-\frac{\overline{\mathfrak{m}^{z}} }{n^{4}}\sum_{a,B,j}^{*} \sum_{J'} 
	\E \Big[\frac{\partial \wh{\gz_{B j}} \wh{\gz_{j B}} }{\partial w_{J'  a}} G^{z}_{J' B} \Big]
	+\frac{ \overline{\mathfrak{m}^{z}} }{n^{3}}
	\sum_{a,B,j}^{*} \E \Big[{\gz_{a B}} \wh{\gz_{B j}} \wh{\gz_{j B}} \<\wh \gz \> \Big]\nonumber\\
	&+O_\prec(n^{-1/2}\Psi^3+n^{-3/2}).
\end{align}
where the last error is from higher order terms as in (\ref{higher_cumu}). Since the index $a$ or 
its conjugate $\bar a$ no longer appears in the remaining product of Green function entries, 
we gain additional $\Psi$ from one more off-diagonal Green function entry or a factor $\<\wh{\gz}\>$ 
on the right side of (\ref{four_step_ex}), plus an error $O_\prec(n^{-1}\Psi^2)$ from the index coincidences,
\eg $J'=B$ or $J'=\bar j$. Therefore, since the number of assignments of the unmatched index $a$ 
after replacement has been reduced to one, we obtain the improved estimate
\begin{align}\label{improve_replace}
	\E[P^o_3(x_1,y_2 \rightarrow j)]	=&O_\prec(\Psi^4+n^{-1}\Psi^2+n^{-1/2}\Psi^3+n^{-3/2}).
\end{align}
The same upper bound also applies to $\E[P^o_3(x_1,y_2 \rightarrow J)]$. 

Therefore, combining (\ref{replacement_form}) and (\ref{improve_replace}), we have improved the 
naive estimate~\eqref{p_d_local} of $\E[P_3^{o}]$ to the better bound in (\ref{p_3_o}).

\subsection{Expansion mechanism: general case and proof of Proposition \ref{lemma3}}\label{sec:proof_unmatch}

Given any unmatched term $P_d^o \in  \mathcal{P}_d^o$ in (\ref{form}), from 
Definition \ref{def:unmatch_shift}, there must exist a lower case index $v_j \in \mathcal{I}_{l_1,l_2}$ 
or an upper case index $V_j \in \mathcal{I}_{l_1,l_2}$ such that this index (or its conjugation)
is assigned to more row indices of Green function
entries in the product than  column indices. For notational 
simplicity, we may denote this special unmatched index 
by $a \in \llbracket 1,n \rrbracket$ and $B\in \llbracket n+1,2n \rrbracket$, respectively.

We will first consider the formal case with an unmatched index $a \in \llbracket 1,n \rrbracket$ satisying
\begin{align}\label{choose}
	\#\{i:x_i \equiv a\}+\#\{i:x_i \equiv \overline{a}\} > \#\{i:y_i \equiv a\}+\#\{i:y_i \equiv \overline{a}\},
\end{align}
and the latter case with $B \in \llbracket n+1,2n \rrbracket$ will follow similarly. 
%	we split the discussion into the following two cases:	
%	1) there exists an unmatched lower case index $v_j \in \mathcal{I}_{l_1,l_2}$ such that 
%	\begin{align}\label{choose}
%		\#\{x_i \equiv v_j\}+\#\{x_i \equiv \overline{v_j}\} > \#\{y_i \equiv v_j\}+\#\{y_i \equiv \overline{v_j}\};
%	\end{align}
%		2) there exists an unmatched upper case index $V_j \in \mathcal{I}_{l_1,l_2}$ such that 
%	\begin{align}\label{choose2}
%		\#\{x_i \equiv V_j\}+\#\{x_i \equiv \ud{V_j}\} > \#\{y_i \equiv V_j\}+\#\{y_i \equiv \ud{V_j}\}.
%	\end{align}
%We will first consider Case 1) and Case 2) will follow similarly.  We denote this special lower case 
%unmatched index satisfying (\ref{choose}) by $a \in \llbracket 1, n \rrbracket$. 
Then there exists an off-diagonal Green function 
entry $G_{x_i y_i}$ with $x_i \equiv a$ and $y_i \not\equiv a, \bar a$. 
Without loss of generality we may assume that this is the first Green function factor, i.e.
we set $x_1\equiv a$ and $y_1 \not\equiv a, \bar a$. We will denote this term 
by $P_d^o(x_1 \equiv a)$ to emphasize that we will expand it using the unmatched index $x_1\equiv a$. Then we have the following estimate whose proof will be given in the next subsection:  
\begin{lemma}\label{lemma:expand_mechanism}
	Let  $P_d^o\in  \mathcal{P}_d^o$ be a given  term with an unmatched index $a$ satisfying (\ref{choose}) and without loss of generality assigned to $x_1$, \ie $P^o_d=P_d^o(x_1 \equiv a)$. Let
	$$
	k^{(r)}_a: = \#\{ i: x_i \equiv a, \bar a\}; \qquad  k^{(c)}_a:=\#\{i: y_i \equiv a, \bar a\}
	$$
	denote the number of $a/\bar a$-assignments as a row or a column index of 
	the Green function entries, respectively, such that $k^{(r)}_a>k^{(c)}_a$.
	Then there exist  finite (bounded by some constant depending only on $d$) subsets 
	\begin{align}\label{subset}
		{\mathcal{A}}^o_{d}\subset \mathcal{P}^o_{d},
		\quad {\mathcal{A}}^o_{>d}\subset  
		%\bigcup_{d'\ge d+1}
		\mathcal{P}^o_{d+1}, \quad
		{\mathcal{B}}^o_{\ge d}\subset  \bigcup_{d'\ge d}\mathcal{P}^o_{d'}, \quad
		{\mathcal{C}}^o_{\ge d-2}\subset  \bigcup_{d'\ge d-2}\mathcal{P}^o_{d'}
	\end{align}
	with the property that the number of $a/\bar a$-assignments as a row or column index in all
	elements of ${\mathcal{A}}^o_{d}$ is reduced to $k^{(r)}_a-1$ and $k^{(c)}_a-1$, respectively, 
	so that we have the bound
	\begin{align}\label{aa}
		\big|\E[P_d^o(x_1 \equiv a)] \big| \lesssim &\sum_{P^o_{d'} \in{\mathcal{A}}^o_{d}} \big|\E [ P^o_{d'} ]\big|+	
		\sum_{P^o_{d'} \in {\mathcal{A}}^o_{>d}} \big|\E [ P^o_{d'} ]\big| \nonumber\\
		&+ \frac{1}{\sqrt{n}}\sum_{P^o_{d'} \in {\mathcal{B}}^o_{\ge d}} \big|\E [ P^o_{d'} ]\big| 
		+ \frac{1}{n}\sum_{P^o_{d'} \in {\mathcal{C}}^o_{\ge d-2}} \big|\E [ P^o_{d'} ]\big| +O_\prec(n^{-3/2}),
	\end{align}
	here $d' $ denotes a degree compatible with (\ref{subset}). In particular, if $k^{(c)}_{a}=0$, then ${\mathcal{A}}^o_{d}$ is an empty set.
\end{lemma}
The precise structure of  the terms in the rhs. of~\eqref{aa} is
irrelevant, hence we do not follow them explicitly,  we will only need a few properties.
Note that all terms in the rhs. of~\eqref{aa} remain unmatched; this key
feature will allow us to iterate this estimate. 
We now briefly  explain the origin and the main features of each sums
and show that every term in the rhs.
is "better" in a certain sense than the initial term.

The set ${\mathcal{A}}^o_{d}$ contains four types of terms (if exist) of degree $d$ 
obtained by index replacements defined in (\ref{replacement}). 
%e.g. one of them is
%$$
%  (m^{z})^2 \sum_{i\geq 2: y_i \equiv a} \E\big[P^o_d(x_1, y_i \rightarrow J)\big].
%$$
%te
%
%
They can be written explicitly as   
\begin{align}\label{aaexpl}
	&(m^{z})^2 \sum_{i\geq 2: y_i \equiv a} \E\big[P^o_d(x_1, y_i \rightarrow J)\big] 
	+m^z \mathfrak{m}^{z} \sum_{i\geq 2: y_i \equiv \bar a} \E\big[P^o_d(x_1, y_i \rightarrow J)\big]\nonumber\\
	&+m^z \mathfrak{m}^{z} \sum_{i\geq 2:  y_i \equiv \bar a} \E\big[ P ^o_d(x_1, y_i \rightarrow j)\big]
	+|\mathfrak{m}^{z}|^2\sum_{i\geq 2: y_i \equiv a} \E\big[ P ^o_d(x_1, y_i \rightarrow j)\big],
\end{align}	
although  the only important fact is that the number of $a/\bar a$-indices 
is reduced by two (\ie one from the row and one from the column) compared with the initial term $P_d^o(x_1 \equiv a)$.   These are the generalisations
of the first two terms in the rhs. of~\eqref{first_step_ex} for the concrete example.

The set ${\mathcal{A}}^o_{>d}$ corresponds to the second order terms with higher degrees, 
\eg in the second line in~\eqref{p_3_middle}; their degree
is increased by one compared to the original term.

The set ${\mathcal{B}}^o_{\ge d}$ comes from the third order cumulant expansion,
indicated by the additional $1/\sqrt{n}$
prefactor (see the last two lines of~\eqref{example_expand} with $p+q+1=3$). 
The number of off-diagonal Green function entries remains at least $d$ and we gained $1/\sqrt{n}$ from the third order cumulants.

Finally, the set ${\mathcal{C}}^o_{\ge d-2}$ coming with a prefactor $1/n$ has two very different sources. 
On the one hand, it comes from the fourth order cumulant expansion
carrying an extra $1/n$ and the number of off-diagonal Green function entries remains at least $d$. 
On the other hand, in the second order cumulant expansion 
the fresh index $J$ or $j$ may coincide with an old index creating a diagonal term. 
Each diagonal term has to be re-written, e.g., as $G^z_{aa} = m^z
+\wh G^z_{aa}$, and thus the term carrying $m^z$  "loses" a $G$-factor. 
Thus the degree may be reduced by two from these diagonal elements; see \eg (\ref{middle_step_ex}) with $J=B$.
In this case the $1/n$ comes from
the reduced number of summation indices.

For definiteness, we stated  and explained Lemma~\ref{lemma:expand_mechanism} for the lower case index $a$,  the modifications for the upper case index $B$ are 
very minor. In the latter case
we may set $x_1\equiv B, y_1 \not\equiv B,\ud{B}$ and denote the term by $P_d^o(x_1 \equiv B)$. This term can be expanded using the unmatched index $x_1 \equiv B$. The abstract bound (\ref{aa}), with the index $a$
replaced with $B$, remains unchanged, only the (irrelevant) explicit formula
changes:  $J$ and $j$ are interchanged within both lines of (\ref{aaexpl}).

We are now ready to prove Proposition \ref{lemma3} by iteratively invoking Lemma~\ref{lemma:expand_mechanism} 
for an unmatched lower case index $a \in \llbracket 1, n \rrbracket$ and the analogous formula for $B \in \llbracket n+1, 2n \rrbracket$. 
The proof in fact relies on iterations on two different levels: the first level  uses
Lemma~\ref{lemma:expand_mechanism} to gain one $\Psi$-factor improvement as explained in the previous Section~\ref{sec:ex}; 
the second level is to iterate this one-step $\Psi$-improvement to an arbitrary power of $\Psi$
until it becomes negligible and only the $O_\prec(n^{-3/2})$ error survives.

\begin{proof}[Proof of Proposition \ref{lemma3}]
	
	Given an unmatched term $P_d^o$ in (\ref{form}), without loss of generality, we may assume that there
	exists an unmatched index $a\in \llbracket 1, n \rrbracket$ satisfying the assignment condition
	in (\ref{choose}) and denote this term by $P^o_d(x_1 \equiv a)$. The case with $B\in \llbracket n+1, 2n \rrbracket$ follows similarly.

	We define the number of the assignments of the index $a$ and $\bar a$ to a row and column index 
	of the Green function entries in the product, \ie
	\begin{align}\label{ncr}
		k^{(r)}_{0}=\#\{i: x_i \equiv a\}+\#\{i: x_i \equiv \bar {a}\}, \quad k^{(c)}_{0}=\#\{i: y_i \equiv a\}+\#\{i: y_i \equiv \bar a\}, 
	\end{align}
	with $k^{(r)}_{0} >k^{(c)}_{0}$ from (\ref{choose}), here we use the subscript $0$ to indicate this quantity is applied to the original term before iterations.

	Applying Lemma \ref{lemma:expand_mechanism}, if $k^{(c)}_{0}=0$, then the first type of subset $\mathcal A_d^o$ in (\ref{subset}) is empty. However if $k^{(c)}_{0} \geq 1$, we need to repeatedly invoke Lemma \ref{lemma:expand_mechanism} to eliminate resulting terms in non-empty $\mathcal A_d^o$. This is our first-level iteration procedure. In the first step, using Lemma \ref{lemma:expand_mechanism}, we have
	\begin{align}
		\big|\E[P_d^o(x_1 \equiv a)] \big| \lesssim &\sum_{P^o_{d'} \in{\mathcal{A}}^o_{d,1}} \big|\E [ P^o_{d'} ]\big|+	
		\sum_{P^o_{d'} \in {\mathcal{A}}^o_{>d,1}} \big|\E [ P^o_{d'} ]\big| \nonumber\\
		&+ \frac{1}{\sqrt{n}}\sum_{P^o_{d'} \in {\mathcal{B}}^o_{\ge d,1}} \big|\E [ P^o_{d'} ]\big| 
		+ \frac{1}{n}\sum_{P^o_{d'} \in {\mathcal{C}}^o_{\ge d-2,1}} \big|\E [ P^o_{d'} ]\big| +O_\prec(n^{-3/2}),
	\end{align}
	where we use the subscript $1$ in the four types of subsets to indicate the iteration step, and each term in the first subset ${\mathcal{A}}^o_{d,1}$ still has the unmatched index $a$ satisfying (\ref{choose}), with \cf (\ref{ncr}),
	$$k^{(r)}_{1}=k^{(r)}_0-1, \qquad k^{(c)}_{1}=k^{(c)}_0-1,  \qquad k^{(r)}_{1}>k^{(c)}_{1}.$$
	Hence we can further apply Lemma \ref{lemma:expand_mechanism} on these leading terms of degree $d$ in ${\mathcal{A}}^o_{d,1}$.  In general, in the $s$-th iteration step, we have
	\begin{align}
		\big|\E[P_d^o(x_1 \equiv a)] \big| \lesssim &\sum_{P^o_{d'} \in{\mathcal{A}}^o_{d,s}} \big|\E [ P^o_{d'} ]\big|+	
		\sum_{P^o_{d'} \in {\mathcal{A}}^o_{>d,s}} \big|\E [ P^o_{d'} ]\big| \nonumber\\
		&+ \frac{1}{\sqrt{n}}\sum_{P^o_{d'} \in {\mathcal{B}}^o_{\ge d,s}} \big|\E [ P^o_{d'} ]\big| 
		+ \frac{1}{n}\sum_{P^o_{d'} \in {\mathcal{C}}^o_{\ge d-2,s}} \big|\E [ P^o_{d'} ]\big| +O_\prec(n^{-3/2}),
	\end{align}
	where each term in the first subset ${\mathcal{A}}^o_{d,s}$ (if exists) satisfies 
	$$k^{(r)}_{s}=k^{(r)}_0-s, \qquad k^{(c)}_{s}=k^{(c)}_0-s,  \qquad k^{(r)}_{s}>k^{(c)}_{s}.$$
	We stop the iterations at step $s=k^{(c)}_{0}+1$ so that the resulting subset ${\mathcal{A}}^o_{d,s}$ is empty, 
	we hence obtain the following estimate  for $P_d^o=P_d^o(x_1 \equiv a)$;
	\begin{align}\label{iterate}
		\big|\E[P_d^o] \big| \lesssim &	
		\sum_{P^o_{d'} \in {\mathcal{A}}^o_{>d,\ast}} \big|\E [ P^o_{d'} ]\big| + \frac{1}{\sqrt{n}}\sum_{P^o_{d'} \in {\mathcal{B}}^o_{\ge d,\ast}} \big|\E [ P^o_{d'} ]\big| \nonumber\\
		&\qquad \qquad \qquad + \frac{1}{n}\sum_{P^o_{d'} \in {\mathcal{C}}^o_{\ge d-2,\ast}} \big|\E [ P^o_{d'} ]\big| +O_\prec(n^{-3/2}),
	\end{align}
	where the three  subsets, ${\mathcal{A}}^o_{>d,\ast},{\mathcal{B}}^o_{\ge d, \ast}$ and ${\mathcal{C}}^o_{\ge d-2,\ast}$ are 
	defined as the union of the corresponding subsets in (\ref{subset}) generated in the above $s$-iterations.
	Their precise form is irrelevant,  beyond their degree what  matters is that their
	cardinality can be bounded by some constant only depending on $d$.

	In this way, we have improved the estimate essentially by an additional small factor $\Psi=(n\eta)^{-1}$ 
	from (\ref{p_d_local}), \ie the first group of terms on the right side of (\ref{iterate}) has degree 
	at least $d+1$, while the remaining terms gain extra $n^{-1/2}$ or $n^{-1}$ from the prefactors. 
	The above iteration procedure generalizes the $\Psi$-improvement mechanism explained in
	the previous subsection for a concrete example. Moreover, we obtain a similar formula in 
	(\ref{iterate}) for $P_d^o=P_d^o(x_1 \equiv B)$ using the analogous version of 
	Lemma~\ref{lemma:expand_mechanism} for the upper case index $B$.

	Next, we will perform our second-level iteration, \ie iterating the $\Psi$-improvement 
	mechanism stated in (\ref{iterate}) to increase the degree further. We note that the resulting terms on the right 
	side of (\ref{iterate}) remain unmatched either with one $\Psi$-improvement from ${\mathcal{A}}^o_{>d,\ast}$ 
	(\ie with higher degrees) or with the gain from the prefactor $1/\sqrt{n}$ or $1/n$. 
	Iterating (\ref{iterate}) for $D-d$ times with a large fixed $D>0$ chosen later, we have	
	\begin{align}\label{dD}
		\big|\E[P_d^o] \big| \lesssim &	
		\sum_{P^o_{d'} \in {\mathcal{A}}^o_{>D,\ast}} \big|\E [ P^o_{d'} ]\big| +
		\frac{1}{\sqrt{n}}\sum_{P^o_{d'} \in {\mathcal{B}}^o_{\ge D,\ast}} \big|\E [ P^o_{d'} ]\big| \nonumber\\
		&\qquad \qquad \qquad + \frac{1}{n}\sum_{P^o_{d'} \in {\mathcal{C}}^o_{\ge D-2,\ast}} 
		\big|\E [ P^o_{d'} ]\big| +O_\prec(n^{-3/2}),
	\end{align}
	where the sets ${\mathcal{A}}^o_{>D,\ast}$, ${\mathcal{B}}^o_{\ge D,\ast}$ and ${\mathcal{C}}^o_{\ge D-2,\ast}$ 
	denote the union of the corresponding sets in (\ref{iterate}) generated in the second-level iterations, 
	whose cardinality can be bounded by a constant depending only on $d$ and $D$. 
	Using the naive estimate in (\ref{p_d_local}), we have
	\begin{align}
		\E[P_d^o]=O_\prec\big(\Psi^{D}+n^{-1/2} \Psi^{D-1}+n^{-1}\Psi^{D-3}+n^{-3/2}\big)=O_{\prec}(n^{-3/2}+\Psi^{D}).
	\end{align}
	For any $\eta \geq n^{-1+\epsilon}$
	with a fixed small $\epsilon>0$, we choose $D=\lfloor2/\epsilon\rfloor$ so
	that $\Psi^{D} \lesssim n^{-3/2}$. In particular, if we choose
	$\eta=\eta_0=n^{-7/8-\tau}$ (in fact, used to prove Lemma \ref{lemma0}), we can choose smaller $D=\lceil\frac{12}{1-8\tau}\rceil$ so that $\Psi^{D} \lesssim n^{-3/2}$.
	This  completes the proof of Proposition~\ref{lemma3}.
\end{proof}

\subsection{Proof of Lemma \ref{lemma:expand_mechanism}}
\label{sec:full}

Let $P_d^o(x_1 \equiv a)$ be a given term in $\mathcal{P}_d^o$ with an unmatched index $a$ satisfying (\ref{choose}) and without loss of generality $x_1 \equiv a$, $y_1 \neq a,\bar a$. Using the identity in (\ref{identity})
on the first Green function factor $\wh{\gz_{ay_1}}$ and performing the cumulant expansions as in (\ref{example_expand}), we have
\begin{align}\label{case1}
	\E[P_d^o(x_1 \equiv a)]=&
	-\frac{m^{z}}{n^{l+1}}\sum^{*}_{\mathcal{I}_{l_1,l_2}} \sum_{J} 
	\E \Big[\frac{\partial \prod_{i=2}^{d} \wh{\gz_{x_iy_i}}}{\partial w_{Ja}} G^{z}_{J y_1} \Big]+\frac{m^{z}}{n^{l}}\sum^{*}_{\mathcal{I}_{l_1,l_2}}
	\E \Big[\gz_{a y_1} \prod_{i=2}^{d} \wh{\gz_{x_iy_i}} \<\wh \gz \> \Big]\nonumber\\
	&-\frac{\mathfrak{m}^{z}}{n^{l+1}}\sum^{*}_{\mathcal{I}_{l_1,l_2}}\sum_{j} 
	\E \Big[\frac{\partial \prod_{i=2}^{d} \wh{ \gz_{x_iy_i}}}{\partial w_{j\bar a}} G^{z}_{j y_1} \Big]+\frac{\mathfrak{m}^{z}}{n^{l}}\sum^{*}_{\mathcal{I}_{l_1,l_2}} \E \Big[\gz_{\bar a y_1}
	\prod_{i=2}^{d} \wh{\gz_{x_iy_i}} \<\wh \gz \> \Big]
	\nonumber\\
	&+\sum_{p+q+1 \geq 3} \Big(H^{(1)}_{p+1,q}+H^{(2)}_{q,p+1}\Big),
\end{align}
where $H^{(1)}_{p+1,q}$ and $H^{(2)}_{q,p+1}$ are the higher order terms given by
\begin{align}\label{higher}
	H^{(1)}_{p+1,q}:=&- \frac{m^{z} c^{(p+1,q)} }{p!q! n^{\frac{p+q+1}{2}+l}}  \sum^{*}_{\mathcal{I}_{l_1,l_2}} \sum_{J} 
	\E \Big[\frac{\partial^{p+q} \prod_{i=2}^{d} \wh{\gz_{x_iy_i}}G^{z}_{J y_1}}{\partial w^{p}_{aJ} \partial w^q_{Ja}}  \Big];\nonumber\\
	H^{(2)}_{q,p+1}=	&- \frac{\mathfrak{m}^{z}c^{(q,p+1)} }{p!q! n^{\frac{p+q+1}{2}+l}}  \sum^{*}_{\mathcal{I}_{l_1,l_2}} 
	\sum_{j} \E \Big[\frac{\partial^{p+q} 
		\prod_{i=2}^{d} \wh{\gz_{x_iy_i}}G^{z}_{j y_1}}{\partial w^{p}_{\bar aj} \partial w^q_{j \bar a}}  \Big],
\end{align}
with $c^{(p,q)}$ the $(p,q)$-th cumulants of the normalized i.i.d. entries $\sqrt{n}w_{aB}$.
We will only estimate $H^{(1)}_{p+1,q}$ in (\ref{higher}) and $H^{(2)}_{p+1,q}$ can be treated similarly. 
We remark that the smallness of $m^{z}$ in (\ref{localm}) will not be used in the proof. We now split
of $H^{(1)}_{p+1,q}$ in (\ref{higher}) into the following two parts, \ie $J$ is distinct from or coinciding with the old indices in
$\mathcal{I}_{l_1,l_2}=\{v_k\}_{k=1}^{l_1} \cup \{V_k\}_{k=1}^{l_2}$ (omitting the irrelevant prefactor $c^{(p+1,q)}/p!q!$),
\begin{align}\label{summation}
	H^{(1)}_{p+1,q}=&- \frac{m^z}{n^{\frac{p+q+1}{2}+l}}  \sum^{*}_{\mathcal{I}_{l_1,l_2},J} \E[(\cdots)]
	- \frac{m^z}{n^{\frac{p+q+1}{2}+l}} \sum^{*}_{\mathcal{I}_{l_1,l_2}} 
	\sum_{J} \big( \sum_{k=1}^{l_1} \delta_{J \overline{v_k}} +\sum_{k=1}^{l_2} \delta_{J {V_k}}  \big) \E[(\cdots)] \nonumber\\
	%	  =&:H^{(1)}_{p+1,q}\Big\vert_{\sum^*}+H^{(1)}_{p+1,q}\Big\vert_{ J \in \mathcal{I}_{l_1,l_2}}
	=:&-\frac{m^z}{n^{\frac{p+q-1}{2}}} \Big(\frac{1}{n^{l+1}}   \sum^{*}_{\mathcal{I}_{l_1,l_2+1}} \E[(\cdots)]\Big) 
	-\frac{m^z}{n^{\frac{p+q+1}{2}}} \Big(\frac{1}{n^{l}}   \sum^{*}_{\mathcal{I}_{l_1,l_2}} \sum_{J} \delta_{J \in \mathcal{I}_{l_1,l_2}} \E[(\cdots)]\Big),
\end{align}
where the notation $\sum^{*}$ is given in (\ref{distinct_sum}) indicating that all the summation indices are distinct, and we use the short hand notation 
$\delta_{J \in \mathcal{I}_{l_1,l_2}}$ to indicate the part with index coincidence. 

We first estimate the second part in (\ref{summation}) with index coincidences $J \in \mathcal{I}_{l_1,l_2}$. Using that $|G_{\mathfrak{u} \mathfrak{v}}| \prec 1$ from (\ref{localg})
naively, we gain an additional $n^{-1}$ from the summation and have
\begin{align}\label{same}
	\left|	\frac{m^z}{n^{\frac{p+q+1}{2}}} \Big(\frac{1}{n^{l}}   \sum^{*}_{\mathcal{I}_{l_1,l_2}} \sum_{J} \delta_{J \in \mathcal{I}_{l_1,l_2}} \E[(\cdots)]\Big)\right|=\OO_\prec(n^{-\frac{p+q+1}{2}}).
\end{align}
We remark that for $p+q+1=3$, the error $n^{-3/2}$ is sharp in general. By setting $J=\overline{v_k}$ or $J=V_k$, the terms in (\ref{higher}) might switch to matched terms; see \eg the last two lines of (\ref{example_expand}) with $J=B$ or $j=\ud{B}$.

Next, the first part in (\ref{summation}) with distinct summation indices can be written as a linear combination of 
averaged products of shifted Green function entries of the form in (\ref{form}) with an additional factor
$n^{-\frac{p+q-1}{2}}$. Since $J$ is a fresh index, the number of (shifted) off-diagonal Green function entries remains at least $d$ in the product. If $q \neq p+1$, then from (\ref{rule_1}) and Definition \ref{def:unmatch_shift}, the fresh index $J$ becomes an unmatched index. Otherwise if $q=p+1$, the index $J$ is matched, but the index $a$ remains unmatched using (\ref{rule_1}) and Definition \ref{def:unmatch_shift}. Thus the first part of (\ref{summation}) yields a collection of 
unmatched terms of the form in (\ref{form}) with degrees at least $d$ and with an additional factor
$n^{-\frac{p+q-1}{2}}$. 
%We hence truncate the cumulant expansion at the fourth order with an error $O_\prec(n^{-3/2})$. 
Similar estimates also apply to $H^{(2)}_{p+1,q}$ in (\ref{higher}). 

Therefore for the third and fourth order terms with $p+q+1=3,4$ in (\ref{higher}), we denote by ${\mathcal{B}}^o_{\ge d}\subset \bigcup_{d'\ge d}\mathcal{P}^o_{d'}$ and ${\mathcal{C}}^o_{\ge d}\subset \bigcup_{d'\ge d}\mathcal{P}^o_{d'}$, respectively, the set of the resulting unmatched terms of degrees at least $d$. With these notations and combining with (\ref{same}), we write for short that
\begin{align}\label{higher_terms}
	&\sum_{p+q+1= 3} \left(H^{(1)}_{p+1,q}+ H^{(2)}_{q,p+1}\right)=\frac{1}{\sqrt{n}}\sum_{P^o_{d'} \in {\mathcal{B}}^o_{\ge d}} \E [ P^o_{d'} ]+O_{\prec}(n^{-3/2});\nonumber\\
	&\sum_{p+q+1= 4} \left(H^{(1)}_{p+1,q}+ H^{(2)}_{q,p+1}\right)=\frac{1}{n}\sum_{P^o_{d'} \in {\mathcal{C}}^o_{\ge d}} \E [ P^o_{d'} ]+O_{\prec}(n^{-2}),
\end{align}
and we truncate the cumulant expansion at the fourth order with an error $O_\prec(n^{-3/2})$ using (\ref{localg}).

We next estimate the second order terms, \ie the first two lines of (\ref{case1}). Writing $\<\wh\gz\>$ as in (\ref{udG}) and $\gz_{a y_1}=\wh{\gz_{a y_1}}$ and $\gz_{\bar a y_1}=\wh{\gz_{\bar a y_1}}$ since $y_1 \not\equiv a$ and $\bar a$, the second and fourth term on the right side of (\ref{case1}) are of the form in (\ref{form}) and the degrees of these terms are increased to $d+1$. For the first term on the right side of (\ref{case1}), we split the summation into two parts as in (\ref{summation}). 
By direct computations, the part with index coincidences $J \in \mathcal{I}_{l_1,l_2}$ is given by
$$
-\frac{m^{z}}{n^{l+1}}\sum^{*}_{\mathcal{I}_{l_1,l_2}}  \E \Big[\frac{\partial \prod_{i=2}^{d}
	\wh{\gz_{x_iy_i}}}{\partial w_{Ja}} G^{z}_{J y_1}\delta_{J \in \mathcal{I}_{l_1,l_2}} \Big]
=\frac{m^{z}}{n^{l+1}}\sum^{*}_{\mathcal{I}_{l_1,l_2}}  \E \Big[\sum_{p=2}^{d} 
\big(\prod_{i=2}^{(p)} \wh{\gz_{x_iy_i}} \big) \gz_{x_p J} \gz_{a y_p} G^{z}_{J y_1}\delta_{J \in \mathcal{I}_{l_1,l_2}} \Big].
$$ 
If we set $J=\bar a$, then the index $a$ remains unmatched since the index $J$ appeared once as a row and once as a column of the Green function entries in the product. Otherwise if $J \in \mathcal{I}_{l_1,l_2}\setminus \{a\}$, then the index $a$ obviously remains unmatched. After transforming the Green function entries into their shifted versions using \eg $\gz_{aa}=m^{z}+\wh \gz_{aa}$, the degrees of the resulting terms might be decreased to $d-2$ when all the entries $\gz_{x_p J}, \gz_{a y_p}, G^{z}_{J y_1}$ are diagonal; see (\ref{example_decrease}) for a concrete example. Thus we obtain a collection of unmatched terms of degrees at least $d-2$ with a factor $n^{-1}$ from the index coincidence. Together with the subset ${\mathcal{C}}^o_{\ge d}$ in (\ref{higher_terms}) with the same prefactor $1/n$ from the fourth order cumulant expansion, we denote by ${\mathcal{C}}^o_{\ge d-2}$ the union of these unmatched terms with degrees at least $d-2$, \ie we write them together for short as 
\begin{align}\label{diagonal}
	\frac{1}{n}\sum_{P^o_{d'} \in {\mathcal{C}}^o_{\ge d-2}} \E [ P^o_{d'} ].
\end{align}

For the remaining summation with $J$ distinct from the indices in $\mathcal{I}_{l_1,l_2}$, writing $\gz_{Jy_1}=\wh{\gz_{Jy_1}}$ and $\gz_{x_p J}=\wh{\gz_{x_p J}}$, we have
$$
-\frac{m^{z}}{n^{l+1}}\sum^{*}_{\mathcal{I}_{l_1,l_2},J} 
\E \Big[\frac{\partial \prod_{i=2}^{d} \wh{\gz_{x_iy_i}}}{\partial w_{Ja}} G^{z}_{J y_1} \Big]
=\frac{m^{z}}{n^{l+1}}\sum^{*}_{\mathcal{I}_{l_1,l_2},J} 
\E \Big[\sum_{p=2}^{d} \big(\prod_{i=2}^{(p)} \wh{\gz_{x_iy_i}} \big) \wh{\gz_{x_p J}} \gz_{a y_p} \wh{G^{z}_{J y_1}} \Big].
$$	
If $y_p \not\equiv a, \bar a$, then $\gz_{a y_p}$ from acting $\partial w_{Ja}$ on $G_{a y_p}$ is an extra off-diagonal entry and the degree is thus increased to $d+1$.
Otherwise if there exists some $y_p \equiv a$ or $\bar a$, then the resulting diagonal entry $G_{a a}$ or $G_{a \bar a}$ 
which will be replaced with the deterministic function $m^z$ or $\mathfrak{m}^{z}$. In both cases, the index $a$ remains unmatched. Then we have
\begin{align}\label{m_1_term}
	-\frac{m^z}{n^{l+1}}\sum_{\mathcal{I}_{l_1,l_2},J}^* 
	\E \Big[&\frac{\partial \prod_{i=2}^{d} \wh{\gz_{x_iy_i}}}{\partial w_{Ja}} G^{z}_{J y_1} \Big] 
	=(m^{z})^2 \sum_{i\geq 2: y_i \equiv a} \E[P^o_d(x_1, y_i \rightarrow J)]\nonumber\\
	& \qquad +m^z \mathfrak{m}^{z} \sum_{i\geq 2: y_i \equiv \bar a} \E[P^o_d(x_1, y_i \rightarrow J) ]
	+\sum_{P^o_{d'} \in \mathcal{P}^o_{d+1}} \E[P^o_{d'}],
\end{align}
where $P^o_d(x_1, y_i \rightarrow J)$ given in (\ref{replacement}) denotes a term obtained from the original term $P_d^{o}$ 
with the row index $x_1 \equiv a$ and column index $y_i \equiv a$ or $\bar a$ of 
the Green function entries being replaced with a fresh (averaged) summation index $J$, and with a slight abuse of notations, the last sum denotes a specific linear combination of unmatched terms with higher degree $d+1$.

Similarly we estimate the third term on the right side of (\ref{case1}). For the cases with index 
coincidences $\delta_{j \in \mathcal{I}_{l_1,l_2}}$, we obtain unmatched terms with degree 
at least $d-2$ and with a factor $n^{-1}$ which will be added to (\ref{diagonal}). 
For the cases with distinct summation indices, we have \cf (\ref{m_1_term})
\begin{align}\label{m_2_term}
	-\frac{\mathfrak{m}^{z}}{n^{l+1}}\sum^{*}_{\mathcal{I}_{l_1,l_2},j} 
	\E \Big[\frac{\partial \prod_{i=2}^{d} \wh{\gz_{x_iy_i}}}{\partial w_{j \bar a}} 
	&G^{z}_{j y_1} \Big]=|\mathfrak{m}^{z}|^2\sum_{i\geq 2: y_i \equiv a} \E[P^o_d(x_1, y_i \rightarrow j)] \nonumber\\
	& +m^z \mathfrak{m}^{z} \sum_{i\geq 2: y_i \equiv \bar a} \E[P^o_d(x_1, y_i \rightarrow j)]
	+\sum_{P^o_{d'} \in \mathcal{P}^o_{d+1}} \E[P^o_{d'}].
\end{align}
The collection of all the unmatched terms with higher degree $d+1$ in both (\ref{m_1_term}) and (\ref{m_2_term}) is denoted by 
${\mathcal{A}}^o_{>d}$. Moreover, the collection of the leading terms of degree $d$ (if exists) defined
by index replacements in both (\ref{m_1_term}) and (\ref{m_2_term}) is denoted by ${\mathcal{A}}^o_{d}$. 
We note that, for any term in ${\mathcal{A}}^o_{d}$, from the index replacement defined in (\ref{replacement}), 
the number of $a/\bar a$-assignments as a row and column index of the Green function entries has been reduced by one,
respectively.

To sum up, with the above notations, combining (\ref{case1}), (\ref{higher_terms}), (\ref{diagonal}), (\ref{m_1_term}) and (\ref{m_2_term}),  we have proved (\ref{aa}) and hence finish the proof of Lemma \ref{lemma:expand_mechanism}.  \qed

\begin{remark}
	Though here we present only the expansions starting from an off-diagonal Green function entry $\wh{\gz_{a y_1}}$, a similar expansion also holds true if we start from a diagonal entry $\wh{\gz_{a a}}$. 
	We remark that the above expansion is not unique since it depends on the choice of the Green 
	function entry to start performing expansions. The proof of Proposition \ref{lemma3}, however,
	does not rely on the uniqueness of the expansions. 
\end{remark}

\section{Green function comparison in Girko's formula: Proof of Proposition~\ref{gft}}
\label{sec:GFTGFT}

Recall the matrix flow in (\ref{flow}). To prove Proposition \ref{gft}, it then suffices to show the following:
\begin{lemma}\label{lemma}
	Set $\eta_0=n^{-7/8-\tau}$ with a small fixed $\tau>0$ from (\ref{omega_0}) and $T=n^{100}$. Let $f=f^{-}_1$ or $f^{+}_2$. 
	Then there exists some constant $c>0$ such that 
	\begin{align}\label{one_point_result}
		\frac{\dd }{\dd t}	\int_{\C} \Delta f(z) \E \left[ \int_{\eta_0}^T \Im \Tr G_t^{z}(\ii \eta) \dd \eta \right]\dd^2 z=\OO(n^{-c}),
	\end{align}
	and
	\begin{align}\label{two_point_result}
		\frac{\dd}{\dd t}	\int_{\C} \int_{\C} &\Delta f(z_1){ \Delta f(z_2) } \E\bigg[ \int_{\eta_0}^T\int_{\eta_0}^T 
		\Big( (1-\E) \Im\Tr G_t^{z_1}(\ii \eta_1)\Big) \times \nonumber\\
		&\Big(  (1-\E) \Im \Tr G_t^{z_2}(\ii \eta_2)\Big)  \dd \eta_1 \dd \eta_2 \dd^2 z_1 \dd^2 z_2\bigg]=\OO(n^{-c}).
	\end{align}
\end{lemma}

\begin{proof}[Proof of Proposition~\ref{gft}]
	Integrating the bounds from Lemma~\ref{lemma} over $t \in [0, 100\log n]$ and using  standard 
	perturbation theory as in (\ref{t_0toinf}) we conclude the proof of Proposition~\ref{gft}. 
\end{proof}

\subsection{Expectation estimate: Proof of (\ref{one_point_result})}\label{sec:exp}
We introduce the short-hand notation
\begin{align}\label{imaginary}
	\F^{z}_t:=\int_{\eta_0}^T \Im \Tr G_t^{z}(\ii \eta) \dd \eta=-\ii \int_{\eta_0}^T \Tr G_t^{z}(\ii \eta) \dd \eta,
\end{align}
and we will prove a slightly stronger estimate than needed in (\ref{one_point_result}), \ie
\begin{align}\label{middle2}
	\left|	\int_{\C} \Delta f(z) \Big( \frac{\dd }{\dd t}	\E\big[\F_t^{z}\big] \Big) \dd^2 z \right| =\OO( n^{-1/4+5\tau}).
\end{align}
Using the $L^1$ norm of $\Delta f$ in (\ref{deltaf}), it suffices to show
\begin{align}\label{middle}
	\Big|\frac{\dd }{\dd t}	\E\big[\F_t^{z}\big]\Big| =\OO_\prec( n^{-1/2+4\tau}).
\end{align}

Recall the matrix flow in (\ref{flow}). Using Ito's formula and performing the cumulant expansion on
the expectation of the drift term, we obtain the analogue of (\ref{ex_deri}),
\begin{align}\label{third_fourth}
	\frac{\dd \E\big[\F_t^{z}\big]}{\dd t}
	%=\E\Big[\frac{\dd \F_t^{z}}{\dd t}\Big]
	=&-\frac{1}{2} \sum_{a=1}^n \sum_{B=n+1}^{2n}
	\left( \sum_{p+q+1= 3}^{K_0}\frac{c^{(p+1,q)}}{p!q!n^{\frac{p+q+1}{2}}}
	\E \left[   \frac{\partial^{p+q+1} \F_t^{z}}{\partial w_{aB}^{p+1} \partial \overline{w_{aB}}^{q} } \right] \right)\nonumber\\
	&-\frac{1}{2} \sum_{a=1}^n \sum_{B=n+1}^{2n} \left(\sum_{p+q+1 = 3}^{K_0}\frac{c^{(q,p+1)}}{p!q!n^{\frac{p+q+1}{2}}}
	\E \left[   \frac{\partial^{p+q+1} \F^z_{t}}{\partial \overline{w_{aB}}^{p+1} \partial w_{aB}^{q}}\right]\right)+O_\prec(n^{-\frac{K_0}{2}+2}),
\end{align}
with $K_0=100$ and $c^{(p,q)}$ the $(p,q)$-th cumulants of the normalized i.i.d. entries $\sqrt{n}w_{aB}$. It then suffices to consider the first line of (\ref{third_fourth}) to show
\begin{align}\label{K_pq}
	\sum_{p+q+1=3}^{K_0} K^{z}_{p+1,q}:=\sum_{p+q+1=3}^{K_0}\frac{c^{(p+1,q)}}{2p!q!}\left(\frac{1}{n^{\frac{p+q+1}{2}}} \sum_{a,B}  
	\E \left[    \frac{\partial^{p+q+1} \F_t^{z}}{\partial w_{aB}^{p+1} \partial \overline{w_{aB}}^{q} } \right]  \right)
	=\OO_\prec( n^{-1/2+4\tau}),
\end{align}
and the second line of (\ref{third_fourth}) is the same with $a$ and $B$ interchanged.

Recall the differentiation rule in (\ref{rule_1}). We further have 
\begin{align}\label{rule_2}
	\frac{\partial \F_t^{z}}{\partial w_{aB}}=\ii \int_{\eta_0}^T  \sum_{\mathfrak{v}=1}^{2n} 
	\Big( G^z_{\mathfrak{v}a} (\ii \eta)G^z_{B\mathfrak{v}}(\ii \eta)\Big) \dd \eta
	=&\ii \int_{\eta_0}^T \big((G^z)^2(\ii \eta)\big)_{Ba} \dd \eta\nonumber\\
	=&G^{z}_{Ba}(\ii T) - G^{z}_{Ba}(\ii \eta_0) = -\gz_{Ba}(\ii \eta_0)+\OO(n^{-100}),
\end{align}
where we used that $(G^2)(\ii \eta)=-\ii \frac{\dd G(\ii \eta)}{\dd \eta}$, the deterministic norm bound $\|G(\ii T)\|\leq T^{-1}$ with $T=n^{100}$. By direct computations using (\ref{rule_1}) and (\ref{rule_2}), 
each term $K^{z}_{p+1,q}$ given in (\ref{K_pq}) is a linear combination of products of $p+q+1$ 
Green function entries of the following form
\begin{align}\label{F_1}
	\frac{1}{n^{\frac{p+q+1}{2}}}  \sum_{a,B}  \E \Big[  \prod_{i=1}^{p+q+1} \gz_{x_i,y_i} (\ii \eta_0)
	+O(n^{-100})\Big]=\OO_\prec\Big( n^{-\frac{p+q-3}{2}}(\Psi^{p+q+1}+n^{-1})\Big)
\end{align}
where $x_i, y_i\equiv a$ or $B$ satisfying 
the assignment condition in (\ref{pq_number}), and the last estimate follows from the local law in (\ref{localg})
and (\ref{localm}) with $\Psi=(n\eta_0)^{-1}=n^{-1/8+\tau}$. We remark that the error term $n^{-1}$ is from the cases with index coincidence, \ie $a = \ud{B}$. In particular, we have from (\ref{F_1}) that
\begin{align}\label{ex4}
	|K^{z}_{p+1,q}|\prec n^{-1/2+4\tau}, \qquad p+q+1\geq 4,
\end{align}
which is enough to prove (\ref{middle}) except for the third order terms. 

We next improve the estimate for these third order terms in (\ref{K_pq}) with $p+q+1=3$. Transforming the Green function entries in these terms to their shifted versions by (\ref{shifted}), these third order terms with the summation restriction $a\neq \ud{B}$ are of the form in (\ref{form}) with a factor $n^{1/2}$ and with unmatched indices $a$ and $B$ from Definition~\ref{def:unmatch_shift}. Using Proposition~\ref{lemma3}, these unmatched terms with a factor $\sqrt{n}$ can be bounded by $O_\prec(n^{-1})$. For the remaining summation with the index coincidence $a =\ud{B}$, they can be bounded by $\OO_\prec(n^{-1/2})$ using that $|G_{\mathfrak{u} \mathfrak{v}}| \prec 1$ from (\ref{localg}). Therefore, the third order terms in (\ref{K_pq}) can be bounded by
\begin{align}\label{third_estimate_two}
	\sum_{p+q+1=3} K_{p+1,q}^{z} =
	%\sqrt{n}\sum_{P^o_d \in \mathcal{P}^o_d} \E[P_d^{o}] +O_\prec(n^{-1/2})=
	O_\prec(n^{-1/2}).
\end{align}
Combining (\ref{third_fourth}), (\ref{ex4}) and (\ref{third_estimate_two}), we have proved the expectation estimate in (\ref{middle}).

\subsection{Variance estimate: Proof of (\ref{two_point_result})}\label{sec:var}
We start with the short-hand notation, $j=1,2$
\begin{align}\label{F}
	\wh{ \F_t^{z_j}} :=\F_t^{z_j}-\E[\F_t^{z_j}]=-\ii \int_{\eta_0}^T \Big( \Tr G_t^{z_j}(\ii \eta)
	-\E\big[ \Tr G_t^{z_j}(\ii \eta)\big]\Big) \dd \eta \prec 1,
\end{align}
where the last estimate follows from the local law in (\ref{localg}). We will prove a slightly stronger
estimate than needed in (\ref{two_point_result}), i.e. we will prove
\begin{align}\label{aim_two}
	\int_{\C} \int_{\C} \Delta f(z_1){ \Delta f(z_2) } 
	\Big( \frac{\dd}{\dd t}\E\big[\wh{ \F_t^{z_1}} \wh{\F_t^{z_2}}\big]\Big) \dd^2 z_1 \dd^2 z_2 =\OO(n^{-1/8+3\tau}).
\end{align}
Using the $L^1$ norm of $\Delta f$ in (\ref{deltaf}), it suffices to prove that
\begin{align}\label{middle_two}
	\frac{\dd}{\dd t}\E\big[\wh{ \F_t^{z_1}} \wh{\F_t^{z_2}}\big]=\mbox{M-terms}(z_1,z_2)+\OO(n^{-5/8+\tau}),
\end{align}
where M-terms$(z_1,z_2)$ is a deterministic function 
%as computed 
%in (\ref{M_term_error}) below, though the function itself is too large, it 
satisfying the following integral condition
\begin{align}\label{delta_int}
	\left|\int_{\C} \int_{\C} \Delta f(z_1){ \Delta f(z_2) }~ \mbox{M-terms}(z_1,z_2) ~ \dd^2 z_1 \dd^2 z_2 \right|\ll n^{-1/8+3\tau}.
\end{align}

Now we focus on proving (\ref{middle_two}). Using Ito's formula and performing the cumulant expansion on the drift term, we obtain the analogue of (\ref{third_fourth})
\begin{align}\label{deri_var}
	\frac{\dd	\E\big[ \wh{ \F_t^{z_1}} \wh{\F_t^{z_2}} \big]}{\dd t}
	%=\E\Big[\frac{\dd \wh{ \F_t^{z_1}} \wh{\F_t^{z_2}} }{\dd t}\Big]
	=&-\frac{1}{2} \sum_{a=1}^n\sum_{B=n+1}^{2n}\left( \sum_{p+q+1=3}^{K_0}    \frac{c^{(p+1,q)}}{p!q!n^{\frac{p+q+1}{2}}}   \E \left[  \frac{\partial^{p+q+1} \wh{\F_t^{z_1}} \wh{\F_t^{z_2}}}{\partial w_{aB}^{p+1} \partial \overline{w_{aB}}^{q} }\right] \right)\nonumber\\
	&-\frac{1}{2} \sum_{a=1}^n\sum_{B=n+1}^{2n} \left(\sum_{p+q+1=3}^{K_0}   \frac{c^{(q,p+1)}}{p!q!n^{\frac{p+q+1}{2}}}   \E \left[  \frac{\partial^{p+q+1} \wh{\F_t^{z_1}} \wh{\F_t^{z_2}}}{\partial \overline{w_{aB}}^{p+1} \partial w_{aB}^{q} }\right] \right)+O_\prec(n^{-\frac{K_0}{2}+2}),
\end{align}
with $K_0=100$ and $c^{(p,q)}$ the $(p,q)$-th cumulants of the normalized i.i.d. entries $\sqrt{n}w_{aB}$. It then suffices to consider the first line of (\ref{deri_var}) to show
\begin{align}\label{pq_order_terms}
	\sum_{p+q+1=3}^{K_0}  \mathcal{K}^{z_1,z_2}_{p+1,q}:=&\sum_{p+q+1=3}^{K_0}\frac{c^{(p+1,q)}}{2p!q!}\left(  \frac{1}{n^{\frac{p+q+1}{2}}}  \sum_{a,B} \E \left[  \frac{\partial^{p+q+1} \wh{\F_t^{z_1}} \wh{\F_t^{z_2}}}{\partial w_{aB}^{p+1} \partial \overline{w_{aB}}^{q} }\right]\right),
\end{align}
and the second line of (\ref{deri_var}) is the same with $a$ and $B$ interchanged.

Using the differentiation rules in (\ref{rule_1}) and (\ref{rule_2}), each term $\mathcal{K}^{z_1,z_2}_{p+1,q}$ 
in (\ref{pq_order_terms}) is a linear combination of products of $p+q+1$ Green function entries 
(either $\ga$ or $\gb$) with a possible factor $\wh{\F_t^{z_1}}$ or $\wh{\F_t^{z_2}}$, \ie in the following general form
\begin{align}\label{some_form_F}
	\frac{1}{n^{\frac{p+q+1}{2}}}  \sum_{a,B}  \E \Big[  (\wh{\F_t^{z^{(0)}}}) \prod_{i=1}^{p+q+1} G^{z^{(i)}}_{x_i,y_i} (\ii \eta_0)+\OO(n^{-100})\Big],
\end{align}
with $\{z^{(i)} \in \C\}_{i=0}^{p+q+1}$ being either $z_1$ or $z_2$, and $x_i,y_i \equiv a$ 
or $B$ satisfying the assignment condition in (\ref{pq_number}).
Using the local law in (\ref{localg}), (\ref{localm}) and that $|\F_t^{z}| \prec 1$, we have the following naive bound
\begin{align}\label{naive_g_two}
	|\mathcal{K}^{z_1,z_2}_{p+1,q}|=O_{\prec}\big(n^{-\frac{p+q-3}{2}}(\Psi^{p+q+1}+n^{-1})\big),\qquad \Psi=n^{-1/8+\tau},
\end{align}
where the error $n^{-1}$ is from the cases with index coincidence, \ie $a = \ud{B}$. In particular we have 
\begin{align}\label{var5}
	|\mathcal{K}^{z_1,z_2}_{p+1,q}| \prec n^{-9/8+5\tau}, \qquad p+q+1 \geq 5,
\end{align}
which is sufficiently small to prove (\ref{middle_two}) except for the third and fourth order terms with $p+q+1=3,4$. 

We next estimate the third and fourth order terms more carefully. 
In the following, we will drop the argument $\ii \eta_0$ with $\eta_0=n^{-7/8-\tau}$ for notational 
simplicity and ignoring the error $\OO(n^{-100})$ in (\ref{some_form_F}).

\subsubsection{Third order terms}

By direct computations using (\ref{rule_1}) and (\ref{rule_2}), the third order terms 
$\mathcal{K}^{z_1,z_2}_{p+1,q}$ in (\ref{pq_order_terms}) with $p+q+1=3$ are given by linear combinations of the following terms (ignoring the irrelevant $c^{(p,q)}$ coefficients)
\begin{align}\label{third_order_term}
	&\frac{1}{n^{3/2}}\sum_{a,B}\E \Big[ \wh{ \F_t^{z_2}} G^{z_1}_{aa} G^{z_1}_{BB}G^{z_1}_{aB} \Big],
	\qquad \frac{1}{n^{3/2}}\sum_{a,B}\E \Big[ \wh{ \F_t^{z_2}} G^{z_1}_{aB} G^{z_1}_{aB}G^{z_1}_{aB} \Big],\nonumber\\
	&\frac{1}{n^{3/2}} \sum_{a,B} \E \Big[ \ga_{aB} \gb_{aa}\gb_{BB}\Big], \quad 
	\frac{1}{n^{3/2}} \sum_{a,B} \E \Big[ \ga_{aB}\gb_{aB} \gb_{aB} \Big], \quad 
	\frac{1}{n^{3/2}} \sum_{a,B} \E \Big[ \ga_{aB} \gb_{Ba} \gb_{Ba} \Big],
\end{align}
together with the other terms by interchanging $a$ with $B$ and $z_1$ with $z_2$.  

We first consider the terms in (\ref{third_order_term}) with the index coincidence $B= \bar a$ in the summations. The resulting terms except
from the last two terms in (\ref{third_order_term}) can be bounded by $\OO_\prec(n^{-1/2}\Psi)=\OO_\prec(n^{-5/8+\tau})$
using the local law in (\ref{localg}), (\ref{localm}) and that $\wh{\F^{z}_t}$ is centered.
For the last two terms in (\ref{third_order_term}) consisting of factors $G_{aB}$ and $G_{Ba}$ only,
we have from the local law in (\ref{localg}) that
\begin{align}\label{M_term_error}
	&\frac{1}{n^{3/2}} \sum_{a = \ud{B}} \E \Big[ \ga_{aB}\gb_{aB} \gb_{aB} \Big]
	= \frac{1}{\sqrt{n}} \ma \big(\mb\big)^2+\OO_\prec(n^{-5/8+\tau}), \nonumber\\ 
	&\frac{1}{n^{3/2}} \sum_{a = \ud{B}} \E \Big[ \ga_{aB} \gb_{Ba} \gb_{Ba} \Big]
	= \frac{1}{\sqrt{n}} \ma \big(\overline{\mb}\big)^2+\OO_\prec(n^{-5/8+\tau}),
\end{align}
where the leading deterministic functions can only be bounded by $\OO(n^{-1/2})$.  
However from (\ref{m_1m_2}) and (\ref{m_1}), for any $z \in \mathrm{supp}(f^{-}_1) \cup \mathrm{supp}(f^{+}_2)$, we have 
\begin{align}\label{m_delta}
	\mathfrak{m}^{z}=-z+\frac{z \eta}{\eta+\Im m^{z}}=-z+O\left(|z|(|1-|z|^2|+\eta^{2/3})\right)=-z+ \OO(n^{-1/2+\tau}).
\end{align}
Hence the leading deterministic terms in (\ref{M_term_error}) satisfy the intergral condition in (\ref{delta_int}), \ie
for simplicity, we only consider the first term in (\ref{M_term_error}),
\begin{align}
	&\int_{\C} \int_{\C} \Delta f(z_1) \Delta f(z_2) \left( \frac{1}{\sqrt{n}}  \ma     (\mb)^2\right)  \dd^2 z_1 \dd^2 z_2 \nonumber\\
	=&\frac{1}{\sqrt{n}}\left( \int_{\C} \Delta f(z_1) \big(-z_1+\OO(n^{-1/2+\tau})\big) \dd^2 z_1 \right)\left( \int_{\C}  
	\Delta f(z_2)\big((z_2)^2 +\OO(n^{-1/2+\tau})\big)  \dd^2 z_2 \right)\nonumber\\
	=&\OO(n^{-1+4\tau}), 
\end{align}
where we used (\ref{m_delta}), the $L^1$ norm of $\Delta f$ in (\ref{deltaf}),
and that both $z$ and $(z)^2$ are harmonic functions.

Next we study the remaining summations with the restriction $B \neq \bar a$ in (\ref{third_order_term}). 
In contrast to the form of averaged products of shifted Green function entries in (\ref{form}), 
we introduce a slightly different abstract form to adapt this notation to the terms in (\ref{third_order_term}), \ie with a possibe factor $\wh{\F_t^{z_1}}$ or $\wh{\F_t^{z_2}}$,
\begin{align}\label{form_F_ab}
	(\wh{\F_t^{z^{(0)}}}) \frac{1}{n^{l}}  \sum^*_{\mathcal{I}_{l_1,l_2}}   \prod_{i=1}^{d} \wh{G^{z^{(i)}}_{x_i,y_i}},\qquad l=l_1+l_2,
\end{align}
with each $\{z^{(i)} \in \C\}_{i=0}^{p+q+1}$ being either $z_1$ or $z_2$, where the restricted sum $\sum^*_{\mathcal{I}_{l_1,l_2}}$ is defined in (\ref{distinct_sum}), and we assign a summation index $v_j$ or $V_j \in \mathcal{I}_{l_1,l_2}$ or their conjugates $\overline{v_j}, \ud{V_j}$ 
to each row index $x_i$ and column index $y_i$ of the shifted Green function entries in the product. We also define unmatched indices and unmatched terms of the form in (\ref{form_F_ab}) as in Definition \ref{def:unmatch_shift}. Since the proof of Proposition \ref{lemma3} is not sensitive to the modifications in the abstract form, the statement still holds true for the general form in (\ref{form_F_ab}). We omit the proof details.

Provided the assignment condition in (\ref{pq_number}) with $p+q+1=3$, all the third order terms in (\ref{third_order_term}) with restricted summations $B \neq \bar a$ can be tranformed by (\ref{shifted}) to linear combinations of unmatched terms of the form in (\ref{form_F_ab}) with a factor $\sqrt{n}$.  Thus by analogous Proposition \ref{lemma3} for general unmatched term in (\ref{form_F_ab}), we have
\begin{align}\label{third_two_point_off}
	\sum_{p+q+1=3} \mathcal{K}^{z_1,z_2}_{p+1,q} \Big|_{a\neq \ud{B}}= \OO_\prec(n^{-1}). 
\end{align}

Therefore, combining (\ref{M_term_error}) and (\ref{third_two_point_off}), we have 
\begin{align}\label{var3}
	\sum_{p+q+1=3} \mathcal{K}^{z_1,z_2}_{p+1,q} = \mbox{M-terms}(z_1,z_2) +\OO_\prec(n^{-5/8+\tau}),
\end{align}
where the function M-terms$(z_1,z_2)$ is a linear combination of leading deterministic
functions in (\ref{M_term_error}) which satisfy the integral condition in (\ref{delta_int}).

\subsubsection{Fourth order terms}
By direct computations using (\ref{rule_1}) and (\ref{rule_2}), the fourth order terms $\mathcal{K}^{z_1,z_2}_{p+1,q}$ in 
(\ref{pq_order_terms}) with $p+q+1=4$ are averaged products of Green function entries satisfying the assignment condition in (\ref{pq_number}). From  Definition \ref{def:unmatch_shift}, these fourth order terms with restricted summations $a\neq \ud{B}$ are unmatched unless $p=1$ and $q=2$.  Then by Proposition \ref{lemma3}, we have
\begin{align}\label{unmatch_four}
	\sum_{p+q+1=4; p+1\neq q} \mathcal{K}^{z_1,z_2}_{p+1,q}=\sum_{p+q+1=4; p+1\neq q}
	\mathcal{K}^{z_1,z_2}_{p+1,q} \Big|_{a \neq \ud{B}}+\OO_\prec(n^{-1})=\OO_\prec(n^{-1}),
\end{align}
where the error $\OO_\prec(n^{-1})$ comes from the cases with index coincidence, \ie $a = \ud{B}$. 

We next estimate the remaining term $\mathcal{K}^{z_1,z_2}_{2,2}$ for $p=1$ and $q=2$ in (\ref{pq_order_terms}). By direct computations,  $\mathcal{K}^{z_1,z_2}_{2,2}$ is a linear combination of the following terms
\begin{align}\label{fourth_order_term}
	&\frac{1}{n^2} \sum_{a,B}\E\Big[ \wh{\F^{z_2}_t} (\ga_{BB}\ga_{aa})^2\Big], \qquad
	\frac{1}{n^2} \sum_{a, B}\E\Big[ \wh{ \F^{z_2}_t} \ga_{aa}\ga_{BB}\ga_{aB}\ga_{Ba}\Big],\nonumber\\
	&\frac{1}{n^2} \sum_{a, B}\E\Big[ \ga_{aa} \ga_{BB} \gb_{aa} \gb_{BB}\Big], 
	\quad 	\frac{1}{n^2} \sum_{a, B}\E\Big[ \ga_{aB} \gb_{Ba} \gb_{aa} \gb_{BB}\Big], 
	\quad 		\frac{1}{n^2} \sum_{a, B}\E\Big[ \ga_{aB} \ga_{aB} \gb_{Ba} \gb_{Ba}\Big]
\end{align}
as well as the other terms by interchanging $a$ with $B$ and $z_1$ with $z_2$.
Those terms in (\ref{fourth_order_term}) containing $G_{aa}$ or $G_{BB}$ in the product of Green function 
entries can be bounded using the improved estimate of resolvent in (\ref{im_estimate}) and the local law in (\ref{localg}) and (\ref{localm}), \eg the first term in (\ref{fourth_order_term}) is bounded by
\begin{align}\label{simple_1}
	\frac{1}{n^2} \sum_{a, B}\E\Big[ \wh{\F^{z_2}_t} (\ga_{BB}\ga_{aa})^2\Big] \prec \Psi^3  \frac{1}{n}\sum_{a}\E \big[ |\ga_{aa}| \big]=\Psi^3 \E[\Im \<\ga\>] =\OO_\prec(n^{-3/4+2\tau}),
\end{align}
where we used the estimate in (\ref{im_estimate}) with $\eta=n^{-7/8-\tau}$. 
The last term in (\ref{fourth_order_term}) is bounded similarly using the Ward identity, \ie
\begin{align}\label{simple_2}
	\frac{1}{n^2} \sum_{a, B}\E\Big[ \ga_{aB} \ga_{aB} \gb_{Ba} \gb_{Ba}\Big] \prec \frac{1}{n^2} \Psi^2 \sum_{a, B}\E[|\ga_{aB}|^2] \leq \Psi^2\frac{\E[\Im \<\ga\>]}{n\eta} =\OO_\prec( n^{-3/4+2\tau}).
\end{align}
Combining  these bounds with (\ref{unmatch_four}), we conclude that
\begin{align}\label{var4}
	\sum_{p+q+1=4} \mathcal{K}^{z_1,z_2}_{p+1,q} =\OO_\prec (n^{-3/4+2\tau}).
\end{align}

Therefore, using (\ref{deri_var}), (\ref{var5}), (\ref{var3}) and \eqref{var4}, we prove (\ref{middle_two}) and (\ref{delta_int}), hence finish the proof of the variance estimate in (\ref{two_point_result}).

%%%%%%%%%%%%%%%%%%%%%%%%%%%%%%%%%%%%%%%%%%%%%%
%% Single Appendix:                         %%
%%%%%%%%%%%%%%%%%%%%%%%%%%%%%%%%%%%%%%%%%%%%%%
%\begin{appendix}
%\section*{???}%% if no title is needed, leave empty \section*{}.

%\end{appendix}
%%%%%%%%%%%%%%%%%%%%%%%%%%%%%%%%%%%%%%%%%%%%%%
%% Multiple Appendixes:                     %%
%%%%%%%%%%%%%%%%%%%%%%%%%%%%%%%%%%%%%%%%%%%%%%
\begin{appendix}
\section*{Proofs of Proposition~\ref{prop:tail}}
In this appendix  we prove 
a lower tail bound on the smallest eigenvalue of 
$$
Y^z:=(X-z)^*(X-z),
$$
which can also be viewed as the square of the smallest singular value of $X-z$ or as the smallest (in modulus) eigenvalue of $H^z$,
for a standard complex Ginibre matrix $X$.
Recall that the parameter $\delta:=|z|^2-1$  monitors the distance of $z$ from the unit circle. We point out that
in earlier papers \cite{MR4221653, MR4408004, 2105.13720, 2105.13719} we defined
$\delta$ with an opposite sign (i.e. $\delta:=1-|z|^2$) because in those works we were primarily focused on the regime where $|z|\le 1$.
Proposition~\ref{prop:tail} in the current paper our focus is on the regime $|z|>1$ so $\delta$ is positive with the new definition.

A simple redefinition of the variable $y$ shows that~(\ref{tail}) is equivalent to
\begin{equation}
	\label{eq:1a} 
	\mathbf{P}^{\mathrm{Gin}}
	\left((\lambda_1^z)^2\le \frac{x}{n^2\delta}\right)\lesssim \frac{x}{(n\delta^2)^{2/3}}  e^{-n\delta^2(1+\OO(\delta))/2},
	\qquad 0\le x\le C.
\end{equation}

We point out that the $(n\delta^2)^{-2/3}$ prefactor  in \eqref{eq:1a} is not optimal, but it is sufficient for our purposes. 
To make the presentation clearer here we present only the proof of the  simpler version  \eqref{eq:1a}, while
the following remark  explains the possible improvements.

\begin{remark}
	\label{rem:bettb}
	First, the bound \eqref{eq:1a} should hold all the way up to $x\le c (n\delta^2)^2$ 
	with some small constant $c$, corresponding to the fact that~(\ref{tail}) should hold up to $y\le c$, i.e.
	for an entire regime comparable with the gap size of order $\delta^{3}$
	in the spectrum of  $Y^z$. Second,
	we can improve the bound \eqref{eq:1a} to
	\begin{equation}
		\label{eq:bbb}
		\mathbf{P}^{\mathrm{Gin}}\left((\lambda_1^z)^2\le \frac{x}{n^2\delta}\right)
		\lesssim \frac{x}{n\delta^2}  e^{-n\delta^2(1+\OO(\delta))/2}, \qquad 0\le x\le \frac{C}{(n\delta^2)^2},
	\end{equation}
	by exploiting an extra improvement choosing a different contour along the proof
	(see Remark~\ref{rem:imprem} below for a detailed explanation). 
	A simple asymptotic expansion indicates that the bound~\eqref{eq:bbb} is actually optimal.
\end{remark}

\begin{remark}
	
	In Proposition~\ref{prop:tail} (and in Remark~\ref{rem:bettb}) we presented the bound 
	on $(\lambda_1^z)^2$ for $n^{-1/2}\ll\delta\ll 1$ to make our presentation more concise.
	However, a similar analysis  gives an analogous bound for $\delta\sim n^{-1/2}$ and
	$\delta\sim 1$ as well (see also \cite[Section 5.2]{MR4408004} for the case $\delta\sim n^{-1/2}$).
	
\end{remark}

This rest of this section is devoted to the proof of Proposition~\ref{prop:tail} in the form of~\eqref{eq:1a}.
We present two arguments. Our first proof with all details relies on an explicit formula for 
the eigenvalue correlation kernel for $Y^z$ from~\cite{MR2162782}. This approach 
is fairly elementary but it works only for the complex symmetry class. An alternative method
is based upon the supersymmetric (SUSY) representation for the resolvent in~\cite{MR4408004}
which also has a version for the real symmetry class.
We sketch the rigorous argument for the simpler complex case and  we comment on  the
considerably more cumbersome details of the real case. Note that~\eqref{eq:1a}
is formulated for the complex case, the factor $x$ is expected to 
be replaced with $\sqrt{x} + x \exp{( -\frac{n}{2}(\Im z)^2)}$ 
as it was the case in~\cite[Corollary  2.4]{MR4408004} for the $|z|\le 1$ regime (see also \cite{2105.13719,2105.13720}).

\begin{proof}[First proof of Proposition~\ref{prop:tail}]
	By \cite[Theorem 7.1]{MR2162782} the correlation kernel for $Y^z$ is given by (to make the notations consistent we set the dimension $N \equiv n$)
	\begin{equation}
		\label{eq:corrkera}
		K_n(u,v)=\frac{n^3}{\ii \pi}\int_\Gamma \dif \zeta \int_\gamma \dif w e^{n[f(w)-f(\zeta)]} 
		K_B(2n\zeta \sqrt{u}, 2nw \sqrt{v}) \zeta w\left(1-\frac{|z|^2}{(|z|^2-w^2)(|z|^2-\zeta^2)}\right),
	\end{equation}
	where $\Gamma$ is any contour symmetric around $0$ which encircles $\pm |z|^2$, $\gamma$
	is the imaginary axis positively oriented $0\to+\infty$, $0\to -\infty$, and
	\begin{equation}
		\label{eq:phasefa}
		f(w):=w^2+\log(|z|^2-w^2).
	\end{equation}
	Here, for any $x,y\in\mathbf{C}$, the kernel $K_B$ is defined by
	\begin{equation}
		\label{eq:Besskera}
		K_B(x,y)=\frac{xI_0'(x)I_0(y)-yI_0'(y)I_0(x)}{x^2-y^2},
	\end{equation}
	with $I_0(x)$ being the zeroth modified Bessel function:
	\begin{equation}
		\label{eq:intrepbesa}
		I_0(x):=\frac{1}{\pi}\int_0^\pi e^{x\cos \theta}\,\dif \theta.
	\end{equation}
	Note that
	\[
	K_B(x,y)=K_B(x,-y)=K_B(-x,y)=K_B(-x,-y)
	\]
	as a consequence of $I_0$, $I_0'$ being even and odd functions, respectively.
	
	We are interested in the case when $|z|^2=1+\delta$, with $1\gg \delta\gg n^{-1/2}$, and $u=v$.
	In this case the formula~\eqref{eq:corrkera} reduces to
	\begin{equation}
		\label{eq:1ptfa}
		\begin{split}
			K_n(u,u)&=\frac{2n^3}{\ii \pi}\int_\Gamma \dif \zeta \int_{\widehat{\gamma}} \dif w \, e^{n[f(w)-f(\zeta)]}
			K_B(2 n \zeta \sqrt{u}, 2n w \sqrt{u}) \zeta w\left(1-\frac{1+\delta}{(1+\delta-w^2)(1+\delta-\zeta^2)}\right), \\
			f(w)&=w^2+\log(1+\delta-w^2),
		\end{split}
	\end{equation}
	with $\widehat{\gamma}$ being the line $0\to \ii\infty$. We point out that here we used the symmetry 
	with respect to the variable $w$ of the integrand in \eqref{eq:corrkera} to replace the contour $\gamma$ by $\widehat{\gamma}$.

	The main technical result  is the following lemma:
	\begin{lemma}
		\label{lem:necin}
		For any $n^{-1/2}\ll \delta\ll 1$ and $u\lesssim 1/(n^2\delta)$ it holds
		\begin{equation}
			\label{eq:impba}
			K_n(u,u)\lesssim n^{4/3} \delta^{-1/3} e^{-n\delta^2(1+\OO(\delta))/2}.
		\end{equation}
	\end{lemma}

	Hence, for any $0\le x\lesssim 1$ we compute
	\[
	\mathbf{P}^{\mathrm{Gin}}\left((\lambda_1^z)^2\le \frac{x}{n^2\delta}\right)\le\int_0^{x/(n^2\delta)} 
	K(\lambda,\lambda)\, \dif \lambda\lesssim \frac{x}{(n\delta^2)^{2/3}} e^{-n\delta^2(1+\OO(\delta))/2},
	\]
	which concludes the proof  of~\eqref{eq:1a}, hence Proposition~\ref{prop:tail}.
\end{proof}

We now conclude this section with the proof of Lemma~\ref{lem:necin}.

\begin{proof}[Proof of Lemma~\ref{lem:necin}]
	
	By explicit computations we get
	\[
	f'(w)=2w\left(1-\frac{1}{1+\delta-w^2}\right).
	\]
	We thus find that the saddles of $f$, i.e. the solutions of $f'(w_*)=0$, 
	are given by $w_*\in\{0,\pm \sqrt{\delta}\}$. Additionally, by Taylor expansion, we get
	\begin{equation}
		\label{eq:in1a}
		f(\zeta)=\delta-\frac{\delta^2}{2}+\delta \zeta^2-\frac{\zeta^4}{2}+\OO(\delta^3+|\zeta|^6).
	\end{equation}

	\noindent	
	{\it Step 1: Deformation of the contours.}
	We parametrize the $\widehat{\gamma}$-contour as $w=\ii s$, with $s\ge 0$, then
	\[
	f(\ii s)=-s^2+\log(1+\delta+s^2).
	\]
	Note that by \eqref{eq:in1a} it follows
	\begin{equation}
		\label{eq:newsexp}
		\Re[f( \ii s)]=f( \ii s)=\delta-\frac{\delta^2}{2}-\delta s^2-\frac{s^4}{2}+\OO(\delta^3+s^6).
	\end{equation}
	Additionally, simple calculus shows that the map
	\begin{equation}
		\label{eq:ina}
		s\mapsto \Re f(\ii s)
	\end{equation}
	is decreasing for $s\ge 0$. In particular, this implies $f(\ii s)\le f(0)=\log(1+\delta)$ for any $s \ge 0$. 
	
	We choose the contour $\Gamma$ to consist of two disjoint closed curves around
	$\sqrt{1+\delta}$ and $-\sqrt{1+\delta}$, respectively.
	We focus on the contour encircling  $\sqrt{1+\delta}$, the other one can be handled in the same way, hence we neglect
	it from the discussion.
	Next, we parametrize the part of the $\Gamma$-contour lying on the region $\Re\zeta>0$ as $\zeta= \sqrt{\delta}+ t\pm\ii t$, 
	with $t\ge 0$. The curve may be closed with a circular arc $|\zeta|= R$ with some  very large $R$, this regime 
	of integration is negligible; for practical purposes we consider $R=\infty$.	Note that by \eqref{eq:in1a} we get
	\begin{equation}
		\label{eq:newphf}
		\Re f(\sqrt{\delta}+ t\pm\ii t)=\delta+2t^4+4\sqrt{\delta}t^3+\OO(\delta^3+t^6).
	\end{equation}
	Additionally, by an elementary calculation, we have that
	\begin{equation}
		\label{eq:incra}
		t\mapsto \Re f(\sqrt{\delta}+ t\pm\ii  t)=\delta+2\sqrt{\delta}t+\frac{1}{2}\log\left[1-4\sqrt{\delta} t+ 8t^2\delta+4t^4+8\sqrt{\delta} t^3\right]
	\end{equation}
	is a strictly increasing function on $t\ge 0$, as a consequence of
	\[
	\frac{\dif}{\dif t} \Re f(\sqrt{\delta}+ t\pm\ii t)\ge 0, \qquad t\ge  0;
	\]
	the equality holds only for $t=0$.
	
	Finally, along the chosen contours we compute
	\begin{align}
		\label{eq:expcompa}
		&1-\frac{1+\delta}{(1+\delta-w^2)(1+\delta-\zeta^2)}\nonumber\\
		=&\frac{-(1+\delta)(2\sqrt{\delta}t\pm 2\ii(\sqrt{\delta}+t)t+s^2)-(\delta+2\sqrt{\delta}t\pm 2\ii (\sqrt{\delta}+t)t)s^2}{(1+\delta+s^2)(1-2\sqrt{\delta}t-\pm 2\ii t (\sqrt{\delta}+t))}.
	\end{align}

	\noindent
	{\it Step 2. Estimates of the integrals along the contours.}
	We split the analysis into four regimes: $(s,t)\in [0, K \sqrt{\delta}]^2$, $(s,t)\in [0, K \sqrt{\delta}]
	\times (K\sqrt{\delta},+\infty)$, $(s,t)\in (K\sqrt{\delta},+\infty)\times[0, K \sqrt{\delta}]$, $(s,t)\in (K\sqrt{\delta},+\infty)^2$.
	Here $K>0$ is a large constant independent of $n$ and $\delta$ that we will choose shortly.
	
	\noindent \textbf{Regime $(s,t)\in [0, K \sqrt{\delta}]^2$:} In this regime we start with the following expansion
	for the kernel $K_B$:
	\begin{equation}
		\label{eq:kbappr}
		K_B(x,y)=\frac{1}{2\pi^2}+\OO(|x|^2+|y|^2),
	\end{equation}
	which follows by standard asymptotic of modified Bessel functions for $|x|,|y|\lesssim 1$.
	
	Since in our regime $u\lesssim 1/(n^2\delta) $ and we have
	\[
	\left|1-\frac{1+\delta}{(1+\delta-w^2)(1+\delta-\zeta^2)}\right|\lesssim \sqrt{\delta}t+s^2
	\]
	for $s,t\le K\sqrt{\delta}$, by \eqref{eq:expcompa}, together with $w=\ii s$ and $\zeta= \sqrt{\delta}+ t\pm\ii t$
	we find that for $(s,t)\in [0, K \sqrt{\delta}]^2$ it holds
	\begin{equation}
		\label{eq:b1b}
		\begin{split}
			K_n(u,u)&\lesssim n^3  \delta \left(\int_0^{K\sqrt{\delta}} s e^{n\delta-n\frac{\delta^2+s^4+2\delta s^2}{2}
				+\OO(n\delta^3)} \, \dif s\right)\left(\int_0^{K\sqrt{\delta}} t e^{-n\delta-2nt^4-4n\sqrt{\delta}t^3+\OO(n\delta^3)}\,\dif t\right) \\
			&\quad+ n^3  \sqrt{\delta} e^{-n\delta}\int_0^{K\sqrt{\delta}} s^3 e^{n\delta-n\frac{\delta^2+s^4+2\delta s^2}{2}
				+\OO(n\delta^3)} \, \dif s \\
			&\lesssim n^3\delta e^{-n\delta^2(1+\OO(\delta)/2}
			\left(\int_0^{K\sqrt{\delta}} s e^{-n\delta s^2}\, \dif s \right)\left(\int_0^{K\sqrt{\delta}} t e^{-n\sqrt{\delta}t^3} \dif t\right) \\
			&\quad+n^3\sqrt{\delta} \int_0^{K\sqrt{\delta}} s^3 e^{-n\delta s^2}\,\dif s \\
			&\lesssim  n^{4/3}\delta^{-1/3} e^{-n\delta^2(1+\OO(\delta))/2}.
		\end{split}
	\end{equation}
	where we also used \eqref{eq:newsexp} and \eqref{eq:newphf}.
	
	\begin{remark}
		\label{rem:imprem}
		The improved bound \eqref{eq:bbb} (compared to \eqref{eq:1a})
		can by achieved by choosing the $\widehat{\gamma}$-contour as in Step 1, 
		i.e. $w= \ii s$, and the $\Gamma$-contour to be any admissible contour which is given by $\zeta=\sqrt{\delta}+t\pm \ii c t$, 
		with some $1<c\le 2$, for $t\ll \sqrt{\delta}$ and by $\zeta=\sqrt{\delta}+t\pm \ii  t$ for $t\gg \sqrt{\delta}$.
		In particular an additional gain of a factor $(n\delta^2)^{-1/3}$ is due to the fact that for this new $\zeta$-contour
		the expansion in \eqref{eq:newphf} is replaced by
		\[
		\Re f(\sqrt{\delta}+ t\pm\ii ct)=\delta+2(c^2-1)t^2\delta+t^4\left(3c^2-\frac{1}{2}-\frac{c^4}{2}\right)
		+2(3c^2-1)\sqrt{\delta}t^3+\OO(\delta^3+t^6).
		\]
		In particular, the term $2(c^2-1)t^2\delta$, which non vanishes only for $c>1$, in the exponent
		ensures an additional gain $(n\delta^2)^{-1/3}$ compared to \eqref{eq:b1b} 
		where we only gained using the smaller (for $t\ll \sqrt{\delta}$) factor  $2(3c^2-1)\sqrt{\delta}t^3$.
		
	\end{remark}

	\noindent \textbf{Regime $(s,t)\in [0, K \sqrt{\delta}]\times (K\sqrt{\delta},+\infty)$:} We start with the bound
	\begin{equation}
		\label{eq:bkba}
		\big|K_B(x,y)\big|\lesssim e^{|y|},
	\end{equation}
	for $|x|\lesssim 1$. We remark that a similar bound holds for $|y|\lesssim 1$ after replacing $y$ by $x$ in the r.h.s. of \eqref{eq:bkba}.
	
	Define
	\begin{equation}
		\label{eq:defga}
		g(t):= \Re f(\sqrt{\delta}+ t\pm\ii t)-\delta,
	\end{equation}
	then, by \eqref{eq:incra}, it follows that $t\mapsto g(t)$ is strictly increasing on $t\ge 0$. 
	Hence, using \eqref{eq:ina} together with $f(0)=\log(1+\delta)$, we get
	\[
	e^{n[f(w)-f(\zeta)]}\le e^{-n g(t)}\le e^{-Kn\delta^2/4}e^{-n g(t)/2},
	\]
	where we used \eqref{eq:incra} to estimate one of the two $e^{-ng(t)/2}$ factors.
	
	Then, using that
	\[
	|1+\delta-(\sqrt{\delta}+ t\pm\ii t)^2|^2\gtrsim 1-4\sqrt{\delta} t+ 8t^2\delta+4t^4+8\sqrt{\delta} t^3\ge \frac{1}{2},
	\]
	we readily conclude
	\begin{equation}
		\label{eq:estlarge}
		K_n(u,u)\lesssim n^3\delta K^2 e^{-Kn\delta^2/4} \int_{K\sqrt{\delta}}^\infty
		e^{-ng(t)/2} e^{t/\sqrt{\delta}} t^3\, \dif t\lesssim e^{-Kn\delta^2/8},
	\end{equation}
	where we used \eqref{eq:bkba},  \eqref{eq:defga}, and that 
	\[
	\left| \frac{\delta-(\zeta^2+w^2)+\zeta^2w^2}{(1+\delta-w^2)(1+\delta-\zeta^2)}\right|\lesssim t^2
	\]
	uniformly in $t$ in this regime. To ensure that the error term in \eqref{eq:estlarge}
	is smaller than our goal in \eqref{eq:impba} we choose $K\ge 5$.
	
	\noindent \textbf{Regimes $(s,t)\in (K\sqrt{\delta},+\infty)\times[0, K \sqrt{\delta}]$ and $(s,t)\in (K\sqrt{\delta},+\infty)^2$:} 
	Given the bound
	\[
	\big|K_B(x,y)\big|\lesssim e^{|x|+|y|},
	\]
	and using the monotonicity properties \eqref{eq:ina},\eqref{eq:incra} of $\Re f$ along the contours chosen in 
	Step 1,  the analysis of these regimes is analogous to the regime $(s,t)\in [0, K \sqrt{\delta}]\times (K\sqrt{\delta},+\infty)$ 
	and so omitted. In particular, the contribution of both these regimes is bounded as in \eqref{eq:estlarge}. 
	Combining this fact with \eqref{eq:b1b} and \eqref{eq:estlarge} we conclude the proof of \eqref{eq:impba}.	
\end{proof}

Next we sketch the alternative  proof relying on SUSY.

\begin{proof}[Second proof of Proposition~\ref{prop:tail}] The  starting point is the following 
	contour integral representation of the trace of the resolvent of  $Y^z=(X-z)^*(X-z)$ for
	a complex Ginibre matrix $X$ at any  spectral parameter $w= E+\ii \eps$, with $E\in\R$, $\eps>0$,
	(see \cite[Eq. (28)]{MR4408004}):
	\begin{equation}\label{susy}
		\begin{split}
			\E^{\mathrm{Gin}} \Tr\frac{1}{ Y^z-w} & = \frac{n^2}{2\pi\ii} \int_0^{\ii \infty} \dif x\oint \dif y
			\; e^{-n f(x)+ nf(y)} y\cdot G(x,y), \\
			G(x,y) & : = \frac{1}{xy} -\frac{1}{(1+x)(1+y)}\Big[ 1 +  \frac{|z|^2}{1+x} +  \frac{|z|^2}{1+y} \Big],\\
			f(x) &:= \log \frac{1+x}{x} -  \frac{|z|^2}{1+x}  - wx,
		\end{split}
	\end{equation}
	where the $x$-integration is over the half imaginary axis and the $y$-integration is over a positively oriented circle around 
	the origin that does not enclose $-1$. Since the integrand is analytic away  from 0 and $-1$, the integration
	contours can be freely deformed away from these two  singularities. We need to  investigate  the imaginary part of
	\eqref{susy} in the regime where
	\begin{equation}\label{regime}
		0\le E^{1/3} \ll n^{-1/2} \ll \delta = |z|^2-1, \qquad \eps = +0,
	\end{equation}
	to detect the density $\rho(E)$ of eigenvalues $(\lambda^z)^2$ of $Y^z$ at $E\ll n^{-3/2}$ that would directly
	imply (\ref{tail}).
	Here $\eps$ is an infinitesimally small positive regularization parameter, its only role is to specify in which direction the $x$-contour 
	goes out to infinity.  
	
	Typically, the large $n$ asymptotics of such contour  integral is obtained by saddle point analysis. Both contours
	are deformed through the saddle point $x_*$ of $f$, defined by $f'(x_*)=0$, where a second order Taylor expansion is performed
	both for $f$ and $G$ and the main contribution comes from the value of these functions and their
	derivatives at the saddle. The exponential factors cancel and the result is
	typically polynomial in $n$.
	Among others, this strategy is followed in our analysis in~\cite{MR4408004} for the regime $\delta=|z|^2-1<0$,
	where we found that the saddle has a non-zero imaginary part.
	The current regime~\eqref{regime} behaves quite differently since now $E\ll \delta^3$ lies  outside of the support
	of  $\frac{1}{\pi}\Im m^z(x+\ii 0)$ (see (\ref{eq:scmde})),
	which implies that the relevant saddle $x_*$ is on the positive real axis, in fact by a simple calculation
	we have\footnote{ In the first displayed formula in Section 6.2 of~\cite{MR4408004}
		we erroneously claimed that $x_*\approx 3\delta^{-1}/2$ in the regime $E\ll \delta^3$, the correct
		behavior is $x_*\approx \delta^{-1}$. This wrong factor does not influence the arguments in~\cite{MR4408004}.}
	\begin{equation}\label{saddle}
		x_* = \delta^{-1}\big[ 1+ (E/\delta^3) + \OO (E/\delta^3)^2\big]
	\end{equation}
	for the unique positive solution to $f'(x_*)=0$. 
	
	The spectral density at $E$ is given by
	\begin{equation}\label{susy1}
		\rho(E):=\E^{\mathrm{Gin}} \frac{1}{\pi} \Im \Tr\frac{1}{ Y^z-E-\ii 0}  =  \frac{1}{2\pi\ii}
		\E^{\mathrm{Gin}} \Big[ \Tr\frac{1}{ Y^z-E-\ii 0} - \Tr\frac{1}{ Y^z-E+\ii 0}\Big],
	\end{equation}
	i.e. we need to evaluate the difference of two copies of~\eqref{susy} with an opposite sign in front of the regularization $\eps$.
	Note that $G$ and large part of $f$ is independent of $\eps$, this parameter appears only as a $\pm \ii  \eps x$ term in 
	$f(x)$ and  is relevant only for the non-compact $x$-integration as $\epsilon$  is infinitesimally small.
	We deform the $x$-contour to $\gamma_\pm : =
	[0, a] \cup [a, a\pm \ii \infty]$, where $a>x_*$ is a large real parameter, i.e. we first go from the origin along the real axis 
	to $a$ and then we move  vertically up or down depending on the sign in front of $\eps=+0$ in \eqref{susy1}.
	When taking the difference in~\eqref{susy1}, the contributions of the $x$-integration from the horizontal segment $[0,a]$ 
	exactly cancel. The only contribution comes from the opposite vertical $x$-integration regimes, that can now be
	estimated separately,  yielding the bound
	\begin{equation}\label{susy2}
		\rho(E)\lesssim n^2 \Big| \int_{a }^{a+\ii \infty} \dif x \oint \dif y
		\; e^{-n f(x)+ nf(y) } y\cdot G(x,y) \Big|
	\end{equation}
	that we need to estimate in the regime $E\lesssim n^{-2}\delta^{-1}$ in order to prove~\eqref{eq:1a}.
	We choose $a:=(nE)^{-1}$ and note that $a\gg  \delta^{-1}$ since $n\delta^2\gg1$, i.e.  $a\gg x_*$ by using~\eqref{saddle}.
	Thus by deforming the $y$ contour  to pass through the 
	saddle $x_*$, the two contours will not intersect, analogously to the situation in Step 1 of the
	previous proof. 
	
	The rest of the computation is a standard saddle point analysis for the $y$-integration. 
	Using~\eqref{saddle}, in our
	regime of parameters we have
	$$
	f(x_*) = -\frac{\delta^2}{2} (1+  \OO (\kappa)), \quad f''(x_*)   =3\delta^4(1+ \OO(\kappa)), \quad x_*\cdot G(x, x_*) = 
	\frac{\delta^2}{x}\Big[1+  \OO\Big(\frac{1}{\delta |x|} \Big)\Big]
	$$
	uniformly, whenever $x=a+\ii t$, $t\in [0,\infty)$, with a small parameter $\kappa:=\delta+1/(n\delta^2)\ll1$.
	This yields
	\begin{equation}\label{susyest}
		\begin{split}
			\rho(E) & \lesssim \frac{n^2}{\sqrt{ nf''(x_*) }} e^{nf(x_*)} \Big| \int_{a }^{a+\ii \infty} \dif x 
			\; e^{-n f(x) } x_* \cdot G(x,x_*) \Big| \\
			&\lesssim n^{3/2}e^{-n\delta^2(1+\OO(\delta))/2} \Big| \int_{a }^{a+\ii \infty} \dif x 
			\; \frac{e^{-n f(x) }}{x} \Big[1+  \OO\Big(\frac{1}{\delta |x|} \Big)\Big]\Big|,
		\end{split}
	\end{equation}
	assuming for the moment that the main contribution comes from the $y$-region around the saddle.
	
	In the large $x$ regime, where $|x|= |a+\ii t|\gg  \delta^{-1}$ we have
	the expansion
	\begin{equation}\label{fexp}
		f(x)=-\frac{\delta}{1+x}      - Ex  + \OO (|x|^{-2}).
	\end{equation}
	Note that
	$$
	-n\Re f(a+\ii t) \le \frac{n\delta}{|a+\ii t|} + nEa  \lesssim1, \qquad t\in [0,\infty),
	$$
	therefore the error terms in the integrand can be handled trivially and we have
	\begin{equation}\label{intest}
		\Big| \int_{a }^{a+\ii \infty} \dif x 
		\; \frac{e^{-n f(x) }}{x} \Big[1+  \OO\Big(\frac{1}{\delta |x|} \Big)\Big]\Big| \lesssim
		\Big| \int_{0 }^{\infty} \dif t  \; \frac{e^{\ii nE t }}{a+\ii t}\Big| + \Big| \int_{0 }^{\infty} \dif t  \; \frac{\delta^{-1}+n\delta}{|a+\ii t|^2}\Big|
		\lesssim 1
	\end{equation}
	using $n\delta/a = n^2 E \delta\lesssim 1$. This yields $\rho(E)\lesssim n^{3/2}e^{-n\delta^2(1+\OO(\delta))/2}$
	in the regime $E\lesssim n^{-2}\delta^{-1}$, which gives~\eqref{eq:1a} with a slightly weaker $(n\delta^2)^{-1/2}$
	prefactor instead of $(n\delta^2)^{-2/3}$. 
	
	Finally, the $y$-integration in the regime away from the saddle is estimated  by using monotonicity of 
	$\Re f(y)$ along an appropriate contour found by plotting the level sets of $\Re f$. We omit these uninteresting details.
\end{proof}

Compared with~\eqref{susy}, for the real case an analogous but more involved representation formula holds,
see \cite[Eq. (34)--(36)]{MR4408004}.  It carries an additional integration variable $\tau\in[0,1]$ 
related to the nontrivial dependence on $\Im z$. The phase function $f(y)$ involving the integration variable 
on the compact contour in~\eqref{susy} is also present in the real case; this gives the critical $e^{-n\delta^2/2}$
factor exactly as in the complex case. The analogue of the phase function $f(x)$ for the non-compact integration 
(called $g$ in~\cite{MR4408004}) 
will now depend on the additional parameter $\tau$, but for the relevant regime of very large $|x|$
its asymptotic expansion is similar to $f(x)$ in~\eqref{fexp}.
Both phase functions depend trivially on $\eps$, hence we have exactly the same
cancellation effect  in~\eqref{susy1} as in the complex case, thus
we  indeed need to consider only the large $|x|$ regime. The precise estimates analogous to~\eqref{susyest}--\eqref{intest}
and the control of the regimes far away from the saddle are  more cumbersome and we
do not pursue them in this paper.

\end{appendix}

%%%%%%%%%%%%%%%%%%%%%%%%%%%%%%%%%%%%%%%%%%%%%%
%% Support information, if any,             %%
%% should be provided in the                %%
%% Acknowledgements section.                %%
%%%%%%%%%%%%%%%%%%%%%%%%%%%%%%%%%%%%%%%%%%%%%%
%\begin{acks}[Acknowledgments]
% The authors would like to thank ...
%\end{acks}
%%%%%%%%%%%%%%%%%%%%%%%%%%%%%%%%%%%%%%%%%%%%%%
%% Funding information, if any,             %%
%% should be provided in the                %%
%% funding section.                         %%
%%%%%%%%%%%%%%%%%%%%%%%%%%%%%%%%%%%%%%%%%%%%%%
\begin{funding}

 The second and the fourth author were supported by the ERC Advanced Grant ``RMTBeyond'' No.~101020331. The third author was supported by Dr.\ Max R\"ossler, the Walter Haefner Foundation and the ETH Z\"urich Foundation.

\end{funding}

%%%%%%%%%%%%%%%%%%%%%%%%%%%%%%%%%%%%%%%%%%%%%%
%% Supplementary Material, including data   %%
%% sets and code, should be provided in     %%
%% {supplement} environment with title      %%
%% and short description. It cannot be      %%
%% available exclusively as external link.  %%
%% All Supplementary Material must be       %%
%% available to the reader on Project       %%
%% Euclid with the published article.       %%
%%%%%%%%%%%%%%%%%%%%%%%%%%%%%%%%%%%%%%%%%%%%%%
%\begin{supplement}
%\end{supplement}

%%  MR numbers will be added by VTeX.                      %%
%%                                                         %%
%%  Use \cite{...} to cite references in text.             %%
%%                                                         %%
%%%%%%%%%%%%%%%%%%%%%%%%%%%%%%%%%%%%%%%%%%%%%%%%%%%%%%%%%%%%%

%% if your bibliography is in bibtex format, uncomment commands:
\bibliographystyle{imsart-number} % Style BST file (imsart-number.bst or imsart-nameyear.bst)
\bibliography{extracted}       % Bibliography file (usually '*.bib')

%% or include bibliography directly:
% \begin{thebibliography}{}
% \bibitem{b1}
% \end{thebibliography}

\end{document}